\documentclass[11pt]{amsart}
\usepackage[english]{babel}
\usepackage{amstext}
\usepackage{amsthm}
\usepackage{amssymb}
\usepackage{mathtools}
\usepackage{amsfonts}
\usepackage{amssymb}
\usepackage{anysize}
\usepackage{hyperref}
\usepackage{appendix}
\usepackage{mathtools}

\makeatletter
\numberwithin{equation}{section}
\numberwithin{figure}{section}

\usepackage{mathrsfs}
\usepackage{amsfonts}
\usepackage{amsfonts}\usepackage{mathrsfs}\usepackage{mathrsfs}\usepackage{graphicx}\usepackage{enumerate}\usepackage{epsfig}
\usepackage{subfigure}\usepackage{amsfonts}\usepackage{amscd}\usepackage{amsthm}\usepackage{bbm}
\usepackage{extpfeil}
\usepackage{graphicx}
\usepackage{hyperref}
\usepackage{xcolor}
\numberwithin{equation}{section}

\title[]{On the trend to global equilibrium for Kuramoto Oscillators}

\author[Javier Morales]{Javier Morales}
\address[Javier Morales]{CSCAMM, 4141 CSIC 
8169 Paint Branch Drive
University of Maryland
College Park, MD 20742-3289}
\email{javierm1@cscamm.umd.edu}

\author[David Poyato]{David Poyato}
\address[David Poyato]{Departamento de Matem\'{a}tica Aplicada and Research Unit ``Modeling Nature'' (MNat), Universidad de Granada, Granada, 18071, Spain}
\email{davidpoyato@ugr.es}

\newtheorem{theorem}{Theorem}[section]
\newtheorem{lemma}{Lemma}[section]
\newtheorem{corollary}{Corollary}[section]
\newtheorem{proposition}{Proposition}[section]
\newtheorem{remark}{Remark}[section]

\newtheorem{definition}{Definition}[section]

\DeclareMathOperator{\supp}{supp}
\DeclareMathOperator{\divop}{div}
\DeclareMathOperator{\diam}{diam}

\makeatother

\begin{document}

\date{\today}

\subjclass[2010]{34C15, 34D06, 35B40, 35Q70, 35Q83, 70F99, 92B20, 92B25}

\keywords{Kuramoto model, synchronization, Wasserstein distance, order parameters, gradient flow, entropy production, logarithmic Sobolev inequality, Talagrand inequality.}

{\color{black}
\thanks{\textbf{Acknowledgment.} The work of J. Morales was supported by NSF grants DMS16-13911, RNMS11-07444 (KI-Net) and ONR grant N00014-1812465. The work of D. Poyato has been partially supported by the MECD (Spain) research grant FPU14/06304, the MINECO-Feder (Spain) research grant number MTM2014-53406-R and the Junta de Andalucia (Spain) Project FQM 954. Part of this research has been done while the second author was visiting CSCAMM at the University of Maryland, USA. We are grateful for the hospitality of CSCAMM, Prof. P.-E. Jabin and Prof. E. Tadmor.}
}

\begin{abstract}
{\color{black}
In this paper, we study the convergence to the stable equilibrium for Kuramoto oscillators. Specifically, we derive estimates on the rate of convergence to the global equilibrium for solutions of the Kuramoto-Sakaguchi equation in a large coupling strength regime from generic initial data. As a by-product, using the stability of the equation in the Wasserstein distance, we quantify the rate at which discrete Kuramoto oscillators concentrate around the global equilibrium. In doing this, we achieve a quantitative estimate in which the probability that the oscillators will concentrate at the given rate tends to one as the number of oscillators increases.  Among the essential steps in our proof are: 1) An entropy production estimate inspired by the formal Riemannian structure of the space of probability measures, first introduced by F. Otto in  \cite{O-01}; 2) A new quantitative estimate on the instability of equilibria with antipodal oscillators based on the dynamics of norms of the solution in sets evolving by the continuity equation; 3) The use of generalized local logarithmic Sobolev and Talagrand type inequalities, similar to the ones derived by F. Otto and C. Villani in \cite{OV-00};  4) The study of a system of coupled differential inequalities, by a treatment inspired by the work of L. Desvillettes and C. Villani  \cite{DV-05}. Since the Kuramoto-Sakaguchi equation is not a gradient flow with respect to the Wasserstein distance, we derive such inequalities under a suitable fibered transportation distance.}
\end{abstract}

\maketitle

\tableofcontents


\section{Introduction}

{\color{black}

\noindent In the present paper, we quantify the rate of convergence to the global equilibrium for $C^{1}$ solutions to the Kuramoto-Sakaguchi equation from generic initial data, providing a first quantitative result in this context. As a by-product, we derive a quantitative statistical estimate, on the rate of concentration for the original agent-based Kuramoto model. Such a model was introduced by Y. Kuramoto several decades \cite{K-75,K-84} ago and is one of the paradigms to study collective synchronization phenomena in biological and mechanical systems in nature. It has gained extensive attention from the physics and mathematics community, see \cite{ABPRS-05,ADKMZ-08,BCM-15,BHD-10,CCHKK-14,CHJK-12,DFG-18,HHK-10,HLX-13,HKPZ-16,MMZ-04,PRK-01,S-00,VMM-14}.\\

\noindent The main motivation to perform our study on the Kuramoto-Sakaguchi equation is three-fold. First, such a model has become a starting point for a broad family of models in collective dynamics. Historically, many of the central analytical techniques developed to study such models were first applied to the Kuramoto model and later generalized to the rest of the field. Second, the Kuramoto model provides a concrete example of a gradient flow structure in which the energy functional is not convex. Such lack of convexity generates challenges to use theory of gradient flows to derive rates of convergence. Third, we are interested in quantifying the relaxation time of a nondeterministic event. Indeed,  in a large coupling strength regime, one expects relaxation to the global equilibrium of the particle system with almost sure probability.  However, such relaxation fails for some well prepared initial data.\\

\noindent In the case of identical oscillators, the Kuramoto-Sakaguchi equation exhibits a gradient flow structure in the space of probability measure under the Wasserstein distance. Nowadays, it is well-known that transportation distances between measures can be successfully used to study evolutionary equations. More precisely, one of the most surprising achievements of \cite{JKO-98,O-01,O-99} has been that many evolutionary equations of the form
\[
\frac{\partial\rho}{\partial t}=\divop\bigg(\nabla\rho+\rho\nabla V+\rho(\nabla W\ast\rho)\bigg),
\]
 can be seen as gradient flows of some entropy functional in the spaces of probability measures with respect to the Wasserstein distance:
\[
 W_{2}(\mu,\nu):=\left(\inf_{\gamma\in \Pi(\mu,\nu)}\int_{\mathbb{R}^d\times \mathbb{R}^d}\vert x-y\vert^2\,d\gamma\right)^{\frac{1}{2}},
\]
where the infimum ranges over all the possible transference plans, i.e.,
\[\Pi(\mu,\nu):=\{\gamma\in \mathbb{P}(\mathbb{R}^d\times \mathbb{R}^d):\,\pi_{1\#}\gamma=\mu\ \mbox{ and }\ \pi_{2\#}\gamma=\nu\}.\]
When such entropy functionals are convex with respect to the Wasserstein distance, such an interpretation allows proving entropy estimates and functional inequalities (see \cite{V-09} for more details on this area). Such tools, in turns, can be used to obtain convergence rates and stability estimates of the corresponding equations. \\

\noindent There are two main difficulties when one tries to use such a theory in the Kuramoto-Sakaguchi equation. First, even in the identical case, as for the Kuramoto model, the entropy functional associated with the equation does not satisfy the necessary convexity hypothesis. Second, in the nonidentical case, the  Wasserstein   gradient flow structure of the  equation is not available. On the other hand, the Kuramoto-Sakaguchi equation has the virtue that the broad family of unstable equilibria is characterized easily. Thus, it provides an ideal setting in which to develop techniques to attack the lack of convexity.\\

\noindent In this article, we adapt the techniques developed by L. Desvilletes and C. Villani in \cite{DV-05} to derive quantitative convergence rates for a nonconvex gradient flow in the particular context of the Kuramoto-Sakaguchi equation. We hope that this provides insight on how to attack this difficulty in more general situations.\\

\noindent In this section, we shall first introduce the model. Then, we will recall the current state of the art regarding the asymptotics of the model in a strong coupling strength regime. Finally, we will state our main result, the proof of which will be the object of the rest of the paper.}

\subsection{The Kuramoto model}

{\color{black}
The Kuramoto model governs the synchronization dynamics of $N$ oscillators - each identified by its phase and natural frequency pair $(\theta_{i}(t),\omega_{i})$ in $\mathbb{T}\times\mathbb{R}.$ Such dynamics is given by the system
\begin{equation}\label{Kmodel}
\begin{cases}
\displaystyle \dot{\theta}_{i}=\omega_{i}+\frac{K}{N}\sum_{j=1}^{N}\sin(\theta_{j}-\theta_{i}),\\
\displaystyle \theta_i(0)=\theta_{i,0},
\end{cases}
\end{equation}
for $i=1,\cdots,N$. The large crowd dynamics, $N\rightarrow\infty,$ is captured by the kinetic description, given by the Kuramoto-Sakaguchi equation, which governs the probability distribution of oscillators $f(t,\theta,\omega)$ at $(t,\theta,\omega)\in\mathbb{R}^+\times\mathbb{T}\times\mathbb{R}$
\begin{equation}\begin{cases}\label{E-KS}
\displaystyle \frac{\partial f}{\partial t}+\frac{\partial}{\partial\theta}(v[f]f)=0, & (\theta,\omega)\in\mathbb{T}\times\mathbb{R},\,t\geq 0,\\
f(0,\theta,\omega)=f_{0}(\theta,\omega), & (\theta,\omega)\in\mathbb{T}\times\mathbb{R}.
\end{cases}
\end{equation}
We denote the velocity field by $v[f]$, that is,
\begin{equation}\label{E-velocity-field}
v[f](t,\theta,\omega):=\omega+K\int_{\mathbb{T}}\sin(\theta^{\prime}-\theta)\rho(t,\theta^{\prime})\hspace{1mm}d\theta^{\prime},
\end{equation}
and we define
\[\rho(t,\theta):=\int_{\mathbb{R}}f(t,\theta,\omega)\,d\omega, \qquad g(\omega):=\int_{\mathbb{T}}f(t,\theta,\omega)\,d\theta=\int_{\mathbb{T}}f_{0}(\theta,\omega)\,d\theta.\]
Here, $K$ is the positive coupling strength and measures the degree of the interaction between oscillators, and $\rho$ and $g$ respectively describe the macroscopic phase density and natural frequency distribution. The rigorous derivation from \eqref{Kmodel} to \eqref{E-KS} was done by Lancellotti \cite{L-05} using Neunzert's method \cite{N-84}.
\subsection{The gradient flow structure and stationary solutions} The Kuramoto model in $\mathbb{T}^{N}$ can be lifted to a dynamical system in $\mathbb{R}^{N}$. J. L. van Hemmen and W. F. Wreszinki \cite{vHW-93} observed that by doing this the Kuramoto model can be formulated as a gradient flow of the energy
\begin{equation}\label{Potential_energy}V(\Theta)=-\frac{1}{N}\sum_{j=1}^{N}\omega_{j}\theta_{j}+\frac{K}{2N^{2}}\sum_{k,j=1}^{N}\bigg(1-\cos(\theta_{j}-\theta_{k})\bigg),\ 
\end{equation}
under the  metric of $\mathbb{R}^{N}$ induced by the scaled inner product
\begin{equation}\label{scaled_inner}
\langle v,w\rangle_{N}=\frac{v\cdot w}{N}. 
\end{equation}
Here,
$
\Theta=(\theta_{1},...,\theta_{N}),$ $v,$ and $w$ belong to $\mathbb{R}^{N}.$ Specifically, \eqref{Kmodel} solves the gradient flow problem
\begin{equation}\label{gflow}
\begin{cases}
\dot{\Theta}(t)=-\nabla_{N} V(\Theta(t)),\\
\Theta(0)=\Theta_{0},
\end{cases}
\end{equation}
where $\nabla_{N}$ denotes the gradient with respect to the scaled inner product. Let us also recall that if we define the order parameters $\Theta\longmapsto r(\Theta),\phi(\Theta)$ by the relation
\[r(\Theta)e^{i\phi(\Theta)}=\frac{1}{N}\sum_{k=1}^{N}e^{i\theta_{k}},\ 
\]
then, we have that the potential reads
\begin{equation}\label{Potential_energy-2}
V(\Theta)=-\frac{1}{N}\sum_{j=1}^{N}\omega_{j}\theta_{j}+\frac{K}{2}\big(1-r^{2}(\Theta)\big),\ 
\end{equation}
and the gradient slope take the form
\begin{equation}\label{dGrad}
|\nabla_{N} V(\Theta)|_{N}^{2}=\frac{1}{N}\sum_{j=1}^{N}\bigg|\omega_{j}-Kr\sin(\theta_{j}-\phi)\bigg|^{2}.
\end{equation}
The main interest of the order parameter is that $r(\Theta)$ represents a measure of coherence for the ensemble of oscillators. Specifically, when $r(\Theta)$ is close to $1$, then all the phases $\theta_i$ within $\Theta$ tend to be synchronized around the same phase value. Moreover, using them we can rewrite system \eqref{Kmodel} as follows
\[{\dot{\theta}}_{i}=\omega_{i}-Kr\sin(\theta_{i}-\phi),
\]
for every $i=1,\cdots,N$. Without lost of generality we may assume that, the natural frequencies are centered, i.e.,
\begin{equation}\label{par_omega_centered}
\frac{1}{N}\sum_{i=1}^{N}\omega_{i}=0.
\end{equation} 
We observe that such a condition is not restrictive because we can always perform a linear change of the reference frame to guarantee it. However, such condition is necessary to show the existence of stationary states and we shall assume it throughout the paper. For any such a stationary state $\Theta_{\infty}$ so that $r_\infty>0$, we must have that $\nabla V(\Theta_{\infty})=0$. Using \eqref{dGrad}, we readily obtain that at equilibria the following condition holds
\[\max_{1\leq j\leq N}\vert \omega_j\vert\leq Kr_\infty,\] 
and phases $\theta_j$ must take some of the following two forms
\begin{align}\label{particle_equilibria}
\begin{aligned}
\theta_{j,\infty}&=\phi_\infty+\arcsin\bigg(\frac{\omega_{j}}{Kr_\infty}\bigg),\\
\theta_{j,\infty}&=\phi_\infty+\pi-\arcsin\bigg(\frac{\omega_{j}}{Kr_\infty}\bigg),
\end{aligned}
\end{align}
for every $j=1,\ldots,N$.\\

\noindent In the same spirit, the Hessian operator of the potential $V$ is given by
\begin{equation}\label{dHes}\langle D_{N}^{2}V(\Theta)v,v\rangle_{N}=\frac{K}{N}\sum_{j=1}^{N}r\cos(\theta_{j}-\phi)|v_{j}|^{2}-K\bigg|\frac{1}{N}\sum_{j=1}^{N}v_{j}e^{i\theta_{j}}\bigg|^{2},
\end{equation}
$D^2_N V$ denotesthe Hessian operator with respect to the scaled inner product \eqref{scaled_inner} and $v=(v_{1},...v_{N})$ is contained in $\mathbb{R}^{N}.$ From this, after accounting for the rotational invariance of the model, we deduce that the stable equilibrium must satisfy that
\[\theta_{j,\infty}=\phi+\arcsin\bigg(\frac{\omega_{j}}{Kr_\infty}\bigg),\]
for every $j=1,\ldots,N$.\\

\begin{remark}\label{noniso}
When $r=0$ there are plenty more equilibria. In the identical case it can be shown that they are non-isolated even after accounting for rotation invariance.
\end{remark}

\noindent For the Kuramoto-Sakaguchi equation, in the case of identical oscillators, the equation enjoys a Wasserstein gradient flow structure (we refer the reader to Appendix A from \cite{HKMP-19}). In the nonidentical case, this structure is not strictly available. Nonetheless, in our analysis, we use several techniques and objects inspired by theory of gradient flows in the space of probability measures.
\noindent Similarly, if we consider the continuous version of the order parameters
\begin{equation}\label{Rdef}
Re^{i\phi}=\int_{\mathbb{T}\times \mathbb{R}} e^{i\theta}f(t,\theta,\omega)\,d\theta\,d\omega,
\end{equation}
equation \eqref{E-KS} can be restated as follows
\begin{equation*}
\left\{ \begin{array}{ll}
{\displaystyle \frac{\partial f}{\partial t}+\frac{\partial}{\partial \theta}(\omega f-KR\sin(\theta-\phi)f)=0,} & (\theta,\omega)\in\mathbb{T}\times\mathbb{R},\,t\geq0,\\
{\displaystyle f(0,\theta,\omega)=f_{0}(\theta,\omega),} & (\theta,\omega)\in\mathbb{T}\times\mathbb{R}.
\end{array}\right.
\end{equation*}
Again, without loss of generality, we can assume that $g$ is centered as well, i.e.,
\begin{equation}\label{E-omega-centered}
\int_{\mathbb{R}}\omega g(\omega)\,d\omega=0.
\end{equation}
Again, this is a necessary condition for equilibria to exist and we shall assume it throughout the paper. For any such equilibria $f_{\infty}$ with corresponding $R_{\infty}>0,$ we obtain that
\[\supp g\subseteq[-KR_{\infty},KR_{\infty}],\]
and $f_\infty$ takes the form
\begin{equation}\label{equilibria}
f_{\infty}(\theta,\omega)=g^{+}(\omega)\otimes\delta_{\vartheta^{+}(\omega)}(\theta)+g^{-}(\omega)\otimes\delta_{\vartheta^{-}(\omega)}(\theta),\ 
\end{equation}
where,
\[g=g^{-}+g^{+},\ 
\]
for some non-negative $g^{-},$ and $g^{+}$ and
\begin{align*}
\vartheta^{+}(\omega)&=\phi_{\infty}+\arcsin\left(\frac{\omega}{KR_{\infty}}\right),\\
\vartheta^{-}(\omega)&=\phi_{\infty}+\pi-\arcsin\left(\frac{\omega}{KR_{\infty}}\right),
\end{align*}
for each $\omega\in\supp g$. As it will become apparent later along the paper, the stable equilibria with $R_\infty>0$ correspond to the case $g^-=0$ where there is no antipodal mass. 
}

\subsection{Statement of the problem and main results}

{\color{black}
By direct inspection of the Hessian of the energy \eqref{dHes}, one can see that, in a large coupling strength regime, out of all of the possible equilibria up to rotations; there is only one that is stable. That is the equilibrium in which the Hessian operator is strictly positive on the subspace orthogonal to rotations. One expects that with probability one, the system \eqref{Kmodel} should converge to such equilibria if the coupling strength is sufficiently large. Such phenomenon has been widely observed in numerical simulations. However, to the date, this result is absent from the literature. It has only been verified for restricted initial configurations where all of the oscillators are constrained in an arc of the circle \cite{CHJK-12}.\\

\noindent There have been many approaches in the literature to show the convergence of the system to the critical points of \eqref{dHes} in the large coupling strength regime. Since stable equilibria have oscillators contained within an interval of size less than $\pi$, convergence results have been mainly addressed in the particular case where initial data is originally confined to such a basin of attraction, namely a half-circle. Specifically, in \cite{CHJK-12,HHK-10} a system of differential inequalities was found for the phase and frequency diameter, that yields the convergence of the system to a phase-locked state. Recall that \eqref{Kmodel} is a gradient flow \eqref{gflow} governed by a potential energy \eqref{Potential_energy}. In \cite{HLX-13,LXY-15} the authors derived the convergence to phase-locked states using \L ojasiewicz gradient's inequality for analytic potentials \cite{L-62} and it was used to obtain convergence rates (after some unquantified initial time) in some particular cases where the \L ojasiewicz exponent can be explicitly computed. For general initial data along the whole circle, the literature is rare and the main contribution is \cite{HKR-16}, but rates are not available. One of the main difficulties when trying to use standard theory from dynamical systems to show this is the fact that critical points of \eqref{dHes} are not isolated (see Remark \ref{noniso}).\\

\noindent In the continuum case, accumulation of oscillators in the hemisphere opposite of the order parameter was excluded in \cite{HKMP-19}. However,  convergence towards a stationary solution was not established yet for generic initial data. See \cite{CCHKK-14} for a particular proof when the phase diameter is smaller than $\pi$. Additionally, see \cite{BCM-15} for a description of the equilibrium in the kinetic case, where a conditional convergence result is presented, without rates. To date, regarding generic initial data, there are only arguments based on compactness that do not give any bound on the rate of convergence.\\

\noindent Our goal here is precisely to investigate the long-time relaxations of solutions to the global equilibrium. We are interested in the study of rates of convergence for the Kuramoto-Sakaguchi equation towards the stable equilibria from generic initial data. Additionally, we wish to derive constructive bounds for this convergence and use them to obtain quantitative information about the convergence of the particle system to the global equilibria as well.  There are several reasons why one may be interested in explicit bounds on the rate of convergence. In particular, one may look for the qualitative properties of solutions. More importantly, only after getting convergence rates, we can use the dynamics of the kinetic equations to deduce quantitative statistical information about the particle system.\\

\noindent The first thing that one might be tempted to do is to apply   linearization techniques around the equilibria. This analysis has been done in \cite{D-16,D-17-arxiv,DF-18,DFG-18}, and is connected with the methods in Landau Damping. However, there is a fundamental reason not to be content with that analysis, which has to do with the nature of linearization. Quoting L. Desvillettes and C. Villani : \\

\textit{“This technique is likely to provide excellent estimates of convergence only after the solution has entered a narrow neighborhood of the equilibrium state, narrow enough that only linear terms are prevailing in the equation. 
But by nature, it cannot say anything about the time needed to enter such a neighborhood; the later has to be estimated by techniques which would be well-adapted to the nonlinear equation.”}\\

\noindent Here is where our contribution takes places, and this is why we shall not rely on linearization techniques. Instead, we shall stick as close as possible to the physical mechanism of entropy production. Our main result is here:

\begin{theorem}\label{mainresult}
Let $f_{0}$ be contained in $C^{1}(\mathbb{T}\times\mathbb{R})$ and let $g$ be  compactly supported in $[-W,W]$. Consider the unique global-in-time classical solution $f=f(t,\theta,\omega)$ to \eqref{E-KS}. Then, there exists a universal constant $C$ such that if
\begin{equation}\label{main_cali}\frac{W}{K}\leq CR_{0}^{3},\ 
\end{equation}
then we can find a time $T_{0}$ with the property that 
\begin{equation}\label{T0_bound_main} T_{0}\lesssim\frac{1}{KR_{0}^{2}}\log\bigg(1+W^{1/2}||f_{0}||_{2}+\frac{1}{R_{0}}\bigg),
\end{equation}
and
\[W_2(f,f_\infty)\lesssim e^{-\frac{1}{40}K(t-T_{0})},\ 
\]
for every $t$ in $[T_{0},\infty)$. Here, is the unique global equilibrium of the Kuramoto-Sakaguchi equation up to rotations (see Proposition \ref{unique}).
\end{theorem}

\noindent In the above theorem and throughout the rest of the paper, given two function $h_{1}$ and $h_{2}$ involving the different parameters in our system, we say that $h_{1}\lesssim h_{2}$ if there exists a universal constant $C$ such that $h_{1}\leq C h_{2}$. Since our argument is constructive, every time we use such a notation, we could compute $C$ explicitly. Additionally, because we often deal with absolutely continuous measures, by abuse of notation, we will sometimes use $f$ to denote the measure  $f\hspace{1mm}dx.$\\

\noindent As a direct consequence of our main theorem, we obtain the following quantitative concentration estimate for the particle system.

\begin{corollary}\label{Main-corollary}
Let $\mu_{t}^{N}$ be a sequence of empirical measures associated to solutions of the particle system \eqref{Kmodel} starting at independent and identically distributed random initial data with law $f_{0}$ (see Section 6 for further details). Assume that $f_{0},$ $R_{0},$ $K,$ and $W$ satisfy the hypotheses of Theorem \ref{mainresult} and let $L$ be an interval with diameter $2/5$ centered around the phase $\phi_{\infty}$ of the global equilibrium $f_{\infty}$. Then, there exists a positive time $T_{0}$ satisfying \eqref{T0_bound_main} and an integer $N^*$ with the property that 
\begin{align*}
\log N^{*} \lesssim\frac{1}{R_{0}^{2}}\log\bigg(1+W^{1/2}||f_{0}||_{2}+\frac{1}{R_{0}}\bigg),
\end{align*}
and for any $N\geq N^{*}$ and any $s$ contained in the interval 
\[\bigg[T_{0},T_{0}+\frac{1}{25K}\log\left(\frac{N}{N^{*}}\right)\bigg],
\] 
we can quantify the probability of mass concentration and diameter contraction of the particle system with $N$ oscillators. Indeed, we have that 
\[\mathbb{P}\bigg(\forall\, t\geq s, \exists\, L^N_s(t)\subseteq \mathbb{T}:\,L^N_s(s)=L\mbox{ and }\eqref{E-M-concentration}-\eqref{E-D-contraction}\mbox{ holds}\bigg)\geq 1-C_1e^{-C_2N^\frac{1}{2}}.\]
Here, conditions \eqref{E-M-concentration}  and \eqref{E-D-contraction} yield mass concentration  and diameter contraction. More precisely, such properties  are given by
\begin{align}
\mu_{t}^{N}(L_{s}^N(t)\times \mathbb{R})&\geq 1-\frac{1}{5}e^{-\frac{1}{20}K(s-T_{0})}, &  \mbox{ for every }t&\mbox{ in }[s,\infty),\label{E-M-concentration}\tag{M}\\
\diam(L^N_s(t))&\leq\max\left\{\frac{4}{5}e^{-\frac{K}{20}(t-s)},12\frac{W}{K}\right\}, & \mbox{ for every }t&\mbox{ in }[s,\infty).\label{E-D-contraction}\tag{D}
\end{align}
Additionally, $C_1$ and $C_2$ are universal positive constants which could be explicitly computed.
\end{corollary}

{\color{black}

\subsection{Ingredients}

The proof of Theorem \ref{mainresult} is the first quantitative proof for the relaxation problem for Kuramoto oscillators with generic initial data. It is intricate but rests on a few well-identified principles. Such principles apply with a lot of generalities to many variants of the Kuramoto model. The proof builds upon the following ingredients.\\
\noindent
\begin{itemize}
\item[-] A quantitative entropy production estimate inspired by the formal Riemannian calculus of the probability measures under the Wasserstein distance, first introduced by F. Otto in \cite{O-01}, which we address in Sections 2.3 and 5. See also \cite[Appendix A]{HKMP-19} for an overview in the context of Kuramoto-Sakaguchi with identical oscillators.
\item[-] A fibered Wasserstein distance $W_{2,g}$ presented independently in \cite{M-17-thesis} and \cite{P-19-arxiv}. Such a distance is well adapted to the nonlinear problem. By using this distance, in Section 3 we will derive new logarithmic Sobolev and Talagrand type  inequalities associated with it (see \cite{OV-00}).
\item[-] A quantitative instability estimate excluding the equilibria with mass in the opposite pole of the order parameter, that we derive in Section 4. A form of such an estimate was originally presented in \cite{HKMP-19}, but we use a more refined version in this work.
\item[-] A new estimate on the norms of the solution on sets evolving by the flow of the continuity equation that allows us to propagate information along the different parts of the system. We discuss these estimates in Section 4.1.
\end{itemize}
\vspace{1em}
\noindent For pedagogical reasons, before entering into the details of the proof, we shall provide first a summary of the strategy. Such a summary will be the objective of the next section.

\section{Strategy}\noindent In this section, we shall describe the plan of the proof of Theorem \ref{mainresult}, and the system of differential inequalities upon which our estimates of convergence are based.\\

\noindent Two of the most attractive features of our proof are the fact that it follows the intuition derived from the mechanism of entropy production, and it is systematic. Additionally, it capitalizes on the behavior observed in numerical simulations under a large coupling strength regime.\\

\noindent We shall overcome three crucial difficulties. First, the order parameter $R$ defined in \eqref{Rdef} is not monotonic and when it vanishes so does the mean-field force between particles. Additionally, our description of the equilibria is only valid when it is positive (this difficulty plays an essential role in the particle system as well). The second difficulty is the fact that Kuramoto-Sakaguchi equation tends to concentrate the density, which produces exponential growth of the global $L^{p}$ norms for $p>1$. The third difficulty, related to the second one, is that a large family of equilibria with mass in the opposite hemisphere of the order parameter appears in which the entropy production vanishes.\\

\noindent In the particle system \eqref{Kmodel}, the potential function $V$ plays the role of the entropy. Consequently, since the particle system is a gradient flow (see \eqref{gflow}), we have that 
\[
\frac{d}{dt}V(\Theta(t))=-|\nabla_{N}V(\Theta(t))|_{N}^{2}.
\]
Thus, we can see from this expression that when the particle system slope $|\nabla_{N}V(\Theta(t))|_{N}^{2},$ is large, then the potential function $V(\Theta(t))$ should decrease locally. 
To quantify the rate of increase of the slope, the starting point is the Hessian operator \eqref{dHes} of the energy functional for the particle system. Such an expression implies that $D_{N}^{2}V(\Theta(t))$ is bounded from above (as a quadratic form) by $Kr(\Theta(t))$, that is, 
\[\langle D_{N}^{2}V(\Theta(t))v,v\rangle_{N}\leq Kr(\Theta(t))|v|_{N}^{2},
\]
for any $(v_1,\ldots,v_N)$ in $\mathbb{R}^N$, which implies the differential  inequality
\[-2Kr(\Theta(t))|\nabla_{N}V(\Theta(t))|_{N}^{2}\leq\frac{d}{dt}|\nabla_{N}V(\Theta(t))|_{N}^{2}\leq2K|\nabla_{N}V(\Theta(t))|^{2},\
\]
along solutions of the Kuramoto model \eqref{Kmodel}. Notice that by \eqref{Potential_energy-2}
\[
\frac{V(\Theta(t))}{K}\hspace{1em}\text{and}\hspace{1em}1-r^{2}(t),
\]
are related up to lower-order terms that can be neglected thanks to condition \eqref{main_cali}. Similarly, considering the time derivative of the above quantities, we have that the following two expressions
\[\frac{|\nabla_{N}V(\Theta(t))|_{N}^{2}}{K}\hspace{1em}\text{and}\hspace{1em}\frac{dr^{2}}{dt}(t),\]
should also differ by a lower-order term that, again, can be controlled using \eqref{main_cali}. This justifies that, in the large coupling strength regime, we indistinctly call $\frac{d R^2}{dt}$ and $\vert\nabla_{N}V(\Theta(t))\vert_{N}^{2}$ the dissipation.\\

\noindent In the continuous case, those objects were extended to the setting of the Kuramoto-Sakaguchi equation \eqref{E-KS} with identical oscillators using the Riemannian structure introduced by F. Otto for the space of probability measures (see \cite[Appendix A]{HKMP-19}). However, in the non-identical case the Kuramoto-Sakaguchi equation \eqref{E-KS} is not a Wasserstein gradient flow and this presents an obstacle to try to use the above objects. By analogy, let us define the continuum analog of the particles' slope \eqref{dGrad} given, by, 
\begin{equation}\label{defdis}
\mathcal{I}[f]:=\int_{\mathbb{T}\times\mathbb{R}}\left(\omega-KR\sin(\theta-\phi)\right)^{2}f\,d\theta\,d\omega.
\end{equation}
We shall again call this quantity the dissipation. Indeed, notice that taking derivatives in \eqref{Rdef}, one clearly obtains the following dynamics of the order parameters
\begin{align}\label{para}
\begin{aligned}
\dot{R}=-\int_{\mathbb{T}\times \mathbb{R}}\sin(\theta-\phi)(\omega-KR\sin(\theta-\phi))f\,d\theta\,d\omega,\\
\dot{\phi}=\frac{1}{R}\int_{\mathbb{T}\times \mathbb{R}}\cos(\theta-\phi)(\omega-KR\sin(\theta-\phi))f\,d\theta\,d\omega.
\end{aligned}
\end{align}
Using it, we will show, in Lemma \ref{dis-dR}, that dissipation and time derivative of the order parameter are again related up to lower-order terms that can be controlled by condition \eqref{main_cali}, i.e.,
\[
\mathcal{I}[f_{t}]-W^{2}\leq K\frac{d}{dt}(R^{2})\leq3\,\mathcal{I}[f_{t}]+W^{2}.\]
Indeed, in Corollary \ref{C-dissipation-bounds} we show that we can again control the growth of the dissipation in the continuous description in a similar way, namely,
\[-2KR\mathcal{I}[f]\leq\frac{d}{dt}\mathcal{I}[f]\leq2K\mathcal{I}[f].\ 
\]

\noindent In Section 2.3, we will describe how this relationship along with the principle of entropy production, can be used to provide a universal lower bound of $R(t)$ of the form $\lambda R_{0},$ for some $\lambda$ in $(0,1)$. In fact, we will show that by making $K$ sufficiently large we can make $\lambda$ as close to one as needed.\\

\noindent \subsection{Displacement concavity and entropy production}

\noindent Before entering into the details of the entropy production principle, we set some necessary notation. We define a dynamic neighborhood of the order parameter $\phi$ and its antipode as follows.

\begin{definition}
Given an angle $\alpha$ in $(0,\frac{\pi}{2})$, we denote by $L_{\alpha}^{+}(t)$ the interval (arc) in $\mathbb{T}$ that is centered around $\phi(t),$ and has a diameter $\pi-2\alpha$, that is,
\[L^+_\alpha(t)=\left(\phi(t)-\frac{\pi}{2}+\alpha,\phi(t)+\frac{\pi}{2}-\alpha\right).\]
Similarly, we denote by $L_{\alpha}^{-}(t)$ the interval (arc) in $\mathbb{T}$ of the same diameter that is centered around the antipode $\phi(t)+\pi$, that is,
\[L^-_\alpha(t)=\left(\phi(t)+\frac{\pi}{2}+\alpha,\phi(t)+\frac{3\pi}{2}-\alpha\right).\]
In this way, $L_{\alpha}^{+}(t)\cup L_{\alpha}^{-}(t)$ is a neighborhood of the average phase and its antipode.\\
\end{definition}

\noindent Also, here and throughout the rest of the paper, given a measurable set $B\subseteq\mathbb{T}$ we define
\[
\rho_{t}(B)=\int_{B}\rho(t,\theta)\hspace{1mm}d\theta,\ 
\]
and more generally, we will, let
\[
\rho(A(t))=\int_{A}\rho(t,\theta)\hspace{1mm}d\theta,\  
\]
for any time-dependent family of measurable sets $t\rightarrow A_{t}.$\\

\noindent Now we describe the entropy production principle in our context. Roughly speaking, it will quantify the following fact:\\

\textit{If at some time $t$ the system is far from the family of equilibria with positive order parameter, then the order parameter will increase a lot in the next few instants of time.}\\

\noindent To make it rigorous, let us come back to the dissipation functional \eqref{defdis}. As for the particle system \eqref{dGrad}, notice that $\mathcal{I}[f]$ vanishes if, and only if, $f$ is an equilibrium. Hence, $\mathcal{I}[f]$ can be thought of a natural measure of how close a given $f$ is to the family of equilibria \eqref{equilibria}. Notice that such expression of equilibria \eqref{equilibria} guarantees that, by our assumption \eqref{main_cali} on $\frac{W}{K}$, all the possible equilibria in our analysis have phase support confined to small arcs centered around $\phi$ and its antipode $\phi+\pi$. Since the diameter of the neighborhood can be made arbitrarily small due to hypothesis \eqref{main_cali}, then we can fix any small enough value of $\alpha$ for the size of the neighborhood $L^+_\alpha(t)\cup L^-_\alpha(t)$. For simplicity, we will set $\alpha=\pi/6$ all along the paper.\\

\noindent The entropy production principle then shows that, in the large coupling strength regime, if entropy production is small (i.e., the time derivative of the order parameter is small), then most of the mass of the system lies in the neighborhood $L^+_\alpha(t)\cup L^-_\alpha(t)$ of $\phi(t)$ and its antipode. Specifically, in the proof of Proposition \ref{Rdecrease}, we will quantify such assertion as follows
\begin{equation}\label{lateral}\rho(\mathbb{T}\backslash L_{\alpha}^{+}(t)\cup L_{\alpha}^{-}(t))\leq\frac{1}{KR^{2}\cos^{2}\alpha}\frac{d}{dt}R^{2}+\frac{W^{2}}{K^{2}R^{2}\cos^{2}\alpha}.
\end{equation}
In other words, \eqref{lateral} suggests that when $f$ is sufficiently far from the family of equilibria \eqref{equilibria} (i.e. it has enough mass outside the time-dependent neighborhood $L^+_\alpha(t)\cup L^-_\alpha(t)$), then the dissipation $ \ensuremath{\mathcal{I}[f]}$ is large. Consequently, the time derivative of the order parameter is large in this case as well, and this produces an entropy production of the system. \\

\noindent In the Lemma below, we quantify the corresponding gain in the order parameter.

\begin{lemma}\textbf{(Semiconcavity and entropy production)}\label{Rgain}
Assume that $f_{0}$ is contained in $C^{1}(\mathbb{T}\times\mathbb{R})$ and that $g$ is compactly supported in $[-W,W]$. Consider the unique global-in-time classical solution $f=f(t,\theta,\omega)$ to \eqref{E-KS}. Let $\alpha=\pi/6$, $t_{0}$ be a positive time, and $\lambda$ be contained in $(0,1)$. Additionally, suppose that
\[\sqrt{2}R_{0}\geq R(t_{0})>\lambda R_{0}\hspace{1em}\text{and\ensuremath{\hspace{1em}\dot{R}(t_{0})\geq\frac{K}{4}\cos^{2}\alpha\lambda^{3}R_{0}^{3}}}.
\]
Then, there exists a universal constant $C$ such that if
\begin{equation}\label{Gain_cali}\frac{W}{K}\leq C\lambda^{2}R_{0}^{2},
\end{equation}
then,
\begin{equation}\label{E-semiconcavity-gain}R^{2}(t_{0}+d)-R^{2}(t_{0})\geq\frac{1}{40}\lambda^{4}R_{0}^{3}.
\end{equation}
Moreover, we can select $d$ in such a way that
\[d\leq\frac{1}{3KR_{0}}\log10,\ \ 
\]
and
\[R\leq\frac{3}{2}R_{0}\hspace{1em}\text{in}\hspace{1em}[t_{0},t_{0}+d].\ \ \ 
\]
\end{lemma}

\subsection{Small dissipation regime and lower bounds in the order parameter}
When the dissipation is large, the above entropy production principle quantifies the gain of the order parameter in the next few instants of time. Regarding the reverse regime with small dissipation, Proposition \ref{Rdecrease} in Section 3 will show that when $\dot{R}$ is below a critical threshold, we achieve the following differential inequality
\begin{equation}\label{rlambdaineq}
\frac{d}{dt}R^{2}>\frac{K}{2}\bigg(-R^{3}+[\lambda R_{0}+\frac{3}{5}(1-\lambda)R_{0}]R^{2}-\frac{3}{5}(1-\lambda)\lambda^{2}R_{0}^{3}\bigg),
\end{equation}
which hold in any time interval $[t_{1},t_{2}]$ such that
\[\dot{R}(t)\leq\frac{K}{4}\cos^{2}\alpha\lambda^{3}R_{0}^{3}\hspace{1em}\text{in}\hspace{1em}[t_{1},t_{2}].\ 
\]
The estimate \eqref{lateral} will be crucial to derive such a proposition.   Additionally, note that the right-hand side of \eqref{rlambdaineq} vanishes when $R=\lambda R_{0}$. In Corollary \ref{gain_vs_loss}, we will combine this inequality with the above entropy production in Lemma \ref{Rgain} to quantify a universal lower bound $R(t)\geq\lambda R_0$ of the order parameter.

\subsection{Instability of the antipodal equilibria}
\noindent The main obstacle to use the above entropy production estimate to show the convergence to the global equilibrium is the fact that it does not exclude the possibility that $\dot{R}$ may vanish or alternate signs over long periods. To overcome such difficulty we need to quantify the instability of the antipodal equilibrium, that roughly speaking states the following:\\

\textit{If the system is eventually close enough to a critical point and such a critical point has mass in the opposite hemisphere of the order parameter, then the system would depart from such equilibria and mass will leave the opposite hemisphere exponentially fast.}\\

\noindent To quantify this instability, let us first introduce  some necessary notation. We consider a smooth regularization of the characteristic function of $L^-_\alpha(t)$ as follows
\[\chi_{\alpha,\delta_{0}}^{-}(\theta)=\xi_{\alpha,\delta_{0}}(\theta-\phi-\pi),\]
where $\delta_{0}>0$ is a small fixed parameter and $\xi_{\alpha,\delta_{0}}$ is a smooth regularization of the characteristic function of $[-(\frac{\pi}{2}-\alpha),(\frac{\pi}{2}-\alpha)]$, namely,
\begin{equation}\label{chi_def}
\xi_{\alpha,\delta_{0}}(r):=\left\{ \begin{array}{ll}
{\displaystyle 1,} & \mbox{if }\ \vert r\vert\leq\frac{\pi}{2}-\alpha,\\
{\displaystyle \frac{1}{1+\exp\left(\frac{2\vert r\vert-(\pi-2\alpha+\delta_{0})}{(\frac{\pi}{2}-\alpha+\delta_{0}-\vert r\vert)(\vert r\vert-\frac{\pi}{2}+\alpha)}\right)},} & \mbox{if }\ \frac{\pi}{2}-\alpha\leq\vert r\vert\leq\frac{\pi}{2}-\alpha+\delta_{0},\\
{\displaystyle 0,} & \mbox{if }\ \vert r\vert\geq\frac{\pi}{2}-\alpha+\delta_{0}.
\end{array}\right.\ \ 
\end{equation}
As for $\alpha$, we can take $\delta_{0}$ as small as desired. For notational simplicity we will set \[\xi_{\alpha}:=\xi_{\alpha,1/2}\hspace{1em}\text{and}\hspace{1em}\chi_{\alpha}^{-}:=\chi_{\alpha,1/2}^{-}.\] Additionally, we will use the notation
\[f_{t}^{2}(B)=\int_{A}f^{2}(t,\theta,\omega)\,d\theta\,d\omega,\ 
\]
for any measurable set $B\subseteq\mathbb{T}$ and, more generally,
\[f^{2}(\varphi)=\int\xi(t,\theta,\omega)f^{2}\hspace{1mm}d\theta d\omega,\ \ 
\]
for any function $\varphi:\mathbb{R}^{+}\times\mathbb{T}\times\mathbb{R}\rightarrow\mathbb{R}$. Bearing all the above notation in mind, the main inequality quantifying the instability of equilibria with antipodal mass reads as follows
\[\frac{d}{dt}f^{2}(\chi_{\alpha}^{-}(t))\leq-KR\sin\alpha f^{2}(\chi_{\alpha}^{-}(t)))+4Kf_{t}^{2}\big(\mathbb{T}\big)\bigg[\frac{W}{K}+\sqrt{\frac{2\dot{R}}{KR}+\frac{1}{R^{2}}\frac{W^{2}}{K^{2}}}-R\cos\alpha\bigg]^{+}.\ 
\]
Although this inequality is a variant of an estimate previously introduced in \cite{HKMP-19}, we prove it in Proposition \ref{l2imeq} because it fits better the approach in this paper.\\

\noindent Notice that when the system is close enough to an equilibrium so that the dissipation is below a critical threshold, the second term of this inequality vanishes and, indeed, it establishes the instability of equilibria with antipodal mass. However, when one tries to use such inequality to quantify the convergence rates, but the dissipation is not sufficiently small, one sees that the term $f_{t}^{2}\big(\mathbb{T}\big)$ represents an obstacle. Specifically, it stands to reason that one can produce examples in which $f_{t}^{2}\big(\mathbb{T}\big)$ grows exponentially fast because the Kuramoto-Sakaguchi equation concentrates mass. We solve this difficulty by adopting a Lagrangian viewpoint in which we analyze norms of the solution along sets evolving according to the continuity equation. That is the content of the next subsection.

\subsection{Sliding norms} The key ingredient that allows us to relate the different functionals appearing in our estimates is the notion of sliding norms along the flow of the continuity equation. For this purpose, let $\mathbb{X}_{t_{0},t}(\theta,\omega)=(\Theta_{t_0,t}(\theta,\omega),\omega)$ denote the forward flow map, that is,
\begin{equation*}
\begin{cases}
\displaystyle\frac{d}{dt}\mathbb{X}_{t_{0},t}(\theta,\omega) =\big(v[f],0\big),\\
\displaystyle\mathbb{X}_{t_{0},t_{0}}(\theta,\omega)=(\theta,\omega),
\end{cases}
\end{equation*}
associated to the continuity equation \eqref{E-KS} for any $t,t_{0}\geq0.$\\

\noindent For any measurable set $A\subseteq\mathbb{T}\mathbb{\times R}$, we will denote the image $\mathbb{X}_{t_{0},t}(A)$ by $A_{t_{0},t}$,  For simplicity, when considering a time-dependent set $A(t)$, we will use the notation $A(t_{0})_{t}$ to denote $A(t_0)_{t_0,t}$. Additionally, given a measurable set $B\subseteq\mathbb{T}$, we will use $B_{t_{0},t}$ to denote the projection of $(B\times[-W,W])_{t_{0},t}$ into $\mathbb{T}$. Again, if $B(t)$ is a time-dependent set in $\mathbb{T},$ we will use $B(t_{0})_{t}$ to denote the projection of $(B(t_{0})\times[-W,W])_{t_{0},t}$ into $\mathbb{T}$.\\ 

\noindent Now, we are a position to state our sliding norm estimate which is given by
\[\frac{d}{dt}f^{2}(A_{t_{0},t})\leq KR\bigg(\sup_{(\theta,\omega)\in A_{t_{0},t}}\cos(\theta-\phi(t))\bigg)f^{2}(A_{t_{0},t}),\ \ 
\]
and holds for any measurable set $A\subseteq \mathbb{T}\times\mathbb{R}$. We prove such inequality in Lemma \ref{slidel2}. To use this inequality effectively, one must obtain a control on the dynamics of sets evolving according to the characteristic flow, both in the large and small dissipation regime. We perform this analysis in Section 4.1.

\subsection{The system} 

All the above-mentioned bounds lead to a system of coupled differential inequalities and functional inequalities. For convenience, let us recast it explicitly here:
\begin{equation}\label{GL2}
\frac{d}{dt}f^{2}(A_{t_{0},t})\leq KR\bigg(\sup_{(\theta,\omega)\in A_{t_{0},t}}\cos(\theta-\phi(t)\bigg)f^{2}(A_{t_{0},t}),
\end{equation}
\begin{equation}\label{Dis}
-2KR\mathcal{I}[f]\leq\frac{d}{dt}\mathcal{I}[f]\leq2K\mathcal{I}[f],
\end{equation}
\begin{equation}\label{DisRbis}
\mathcal{I}[f_{t}]-W^{2}\leq K\frac{d}{dt}R^{2}\leq3\,\mathcal{I}[f_{t}]+W^{2},
\end{equation}
\begin{equation}\label{Inst}\frac{d}{dt}f^{2}(\chi_{\alpha}^{-}(t))\leq-KR\sin\alpha f^{2}(\chi_{\alpha}^{-}(t)))+4Kf_{t}^{2}\big(\mathbb{T}\big)\bigg[\frac{W}{K}+\sqrt{\frac{2\dot{R}}{KR}+\frac{1}{R^{2}}\frac{W^{2}}{K^{2}}}-R\cos\alpha\bigg]^{+},
\end{equation}
\begin{equation}\label{Wup}\frac{d}{dt}R^{2}>K\bigg(-R^{3}+[\lambda R_{0}+\frac{3}{5}(1-\lambda)R_{0}]R^{2}-\frac{3}{5}(1-\lambda)\lambda^{2}R_{0}^{3}\bigg),
\end{equation}
where the first inequality holds for any measurable set $A\subseteq\mathbb{T}\times\mathbb{R}$, the last inequality holds in any interval $[t_{1},t_{2}]$ satisfying the hypotheses of Proposition \ref{Rdecrease}, and all of the other inequalities above holds for every $t$ in $[0,\infty)$.\\

\noindent The goal of such a system is to derive an explicit bound on the time $T_{0}$ in Theorem \ref{mainresult}. To achieve this, we use two main components. On the one hand, we study the dynamics of sets along the characteristic flow in Section 4.1. On the other hand, we recover the approach developed by L. Desvillettes and C. Villani in \cite{DV-05} in our setting. Such an argument is described in detail in Section 5 and it consists of performing a subdivision into time intervals subordinated to different scales of values of the order parameter. Such intervals are classified into intervals where the dissipation is above and below a certain threshold. If the dissipation is large on an interval, we use the lower bound \eqref{Dis} in the form of our entropy production estimate to quantify the increase of the order parameter. Conversely, if the dissipation is small, we use \eqref{Inst} to quantify the departure of the system from the family of equilibria with antipodal mass. To do this effectively, we communicate information between the different regimes using inequality \eqref{GL2} and our analysis on the dynamics of sets from Section 4.1.\\

\noindent As a result of the above analysis, we obtain the following corollary:

\begin{corollary} \label{r_*bounda}
Let $f_{0}$ be contained in $C^{1}(\mathbb{T}\times\mathbb{R})$ and let $g$ be compactly supported in $[-W,W]$. Consider the unique global-in-time classical solution $f=f(t,\theta,\omega)$ to \eqref{E-KS} and let $\beta=\pi/3.$ Then, there exists a universal constant $C$ such that if
\[
\frac{W}{K}\leq CR_{0}^{3},
\]
then we can find a time $T_{0}$ with the property that 
\[T_{0}\lesssim\frac{1}{KR_{0}^{2}}\log\bigg(1+W^{1/2}||f_{0}||_{2}+\frac{1}{R_{0}}\bigg),
\]
and
\begin{equation}\label{convex_regime}R(t)\geq\frac{3}{5}\hspace{1em}\text{and\hspace{1em}}\rho\big(\mathbb{T}\mathbb{\backslash}L_{\beta}^{+}(t)\big)\leq e^{-\frac{1}{20}K(t-T_{0})},
\end{equation}
for every $t$ in $[T_{0},\infty).$
\end{corollary}

\noindent Such a Corollary is the starting point of the last part of our strategy.

\subsection{Local displacement convexity and Talagrand type inequalities}
At the particle level, we see that the Hessian operator \eqref{dHes} is positive definite in the subspace orthogonal to rotation whenever the oscillators are strictly contained one a suitable interval. As mentioned in Section 1, the classical theory of gradient flows allows deriving convergence rates towards equilibrium when the energy is strictly convex. Thus, once the mass enters exponentially fast to the region of convexity after $T_{0}$, one may hope to recover such a convergence result for our system. Indeed, inspired by the arguments in \cite{OV-00} on their proof of the logarithmic Sobolev and Talagrand inequalities, we derive analogous inequalities that yield the exponential convergence result and uniqueness of the global equilibrium. Since our system is not a Wasserstein gradient flow, we derive such inequalities for a fibered transportation distance $W_{2,g}$, independently introduced in \cite{M-17-thesis} and \cite{P-19-arxiv}, which is well adapted to the nonlinear problem. The proof of such inequalities is the content of the next section.

\section{Functional  inequalities and a
fibered Wasserstein distance}

{\color{black}

\noindent As discussed before, the proof of Theorem \ref{mainresult} will be split into two distinguished parts that capture two qualitatively different features of the dynamics of Kuramoto-Sakaguchi equation \eqref{E-KS}. Firstly, recall that from many preceding works (see e.g., \cite{BCM-15,CCHKK-14,HKMP-19}) it is apparent that the entropy functional of the equation does not satisfy the necessary convexity properties for the classical theory of gradient flows to work and show convergence towards the global equilibrium. Thus, we need to prove, using different tools, that the dynamics of the equation itself drives the system towards an appropriate ``convexity area'' exponentially fast after some quantified time $T_0>0$. This is the content of Corollary \ref{r_*bounda} where such a convexity area is described by a dynamic neighborhood of the order parameter $\phi$.\\

\noindent The proof of such result is postponed to forthcoming sections and becomes the cornerstone of this paper. We devote this part to study the other main feature of the dynamics. Specifically, we show that although the system is not a Wasserstein gradient flow, the generalized dissipation functional that has been introduced in \eqref{defdis} satisfies an appropriate Hessian-type inequality after the solution has entered into the concentration regime quantified in Corollary \ref{r_*bounda}. The final step is inspired in \cite{OV-00} about the derivation of the logarithmic Sobolev and Talagrand inequalities for gradient flows in then Wasserstein space. Indeed, we shall show that despite the fact that our system is not a Wasserstein gradient flow due to the presence of heterogeneities introduced by $\omega$, some dissipation-transportation inequality still can be achieved for an adequate distance on the space of probability measures. Such inequality along with the exponential decay of the dissipation guarantee the exponential convergence to the global equilibrium in Theorem \ref{mainresult}.\\

\noindent To start, we first study the dynamics of the dissipation functional \eqref{defdis} along the flow of the Kuramoto-Sakaguchi equation.

\begin{theorem}\label{T-dissipation-evolution}
Assume that $f_0$ is contained in $C^1(\mathbb{T}\times \mathbb{R})$ and $g$ is compactly supported in $[-W,W]$. Consider the unique global-in-time classical solution $f=f(t,\theta,\omega)$ to \eqref{E-KS}. Then,
\begin{align*}
\frac{d}{dt}\mathcal{I}[f]&=-K\int_{\mathbb{T}^2\times \mathbb{R}^2}\left((\omega-KR\sin(\theta-\phi))-(\omega'-KR\sin(\theta'-\phi))\right)^2\\
&\hspace{1.5cm}\times \cos(\theta-\theta')f(t,\theta,\omega)f(t,\theta',\omega')\,d\theta\,d\theta'\,d\omega\,d\omega'.
\end{align*}
\end{theorem}

\begin{proof}
Taking derivatives yields the Wasserstein two terms
$$\frac{d}{dt}\mathcal{I}[f]=I_1+I_2,$$
where each of them takes the form
\begin{align*}
I_1&:=2\int_{\mathbb{T}\times \mathbb{R}}(\omega-KR\sin(\theta-\phi))(-K\dot{R}\sin(\theta-\phi)+KR\cos(\theta-\phi)\dot{\phi})f\,d\theta\,d\omega,\\
I_2&:=\int_{\mathbb{T}\times \mathbb{R}}(\omega-KR\sin(\theta-\phi))^2\partial_t f\,d\theta\,d\omega.
\end{align*}
Let us use \eqref{para} and substitute the formulas for $\dot{R}$ and $\dot{\phi}$ in each term. By doing this, we get that
\begin{align}\label{E-1}
\begin{aligned}
I_1&=2K\int_{\mathbb{T}^2\times \mathbb{R}^2}(\omega-KR\sin(\theta-\phi))(\omega'-KR\sin(\theta'-\phi))\\
&\hspace{1.5cm}\times (\sin(\theta-\phi)\sin(\theta'-\phi)-\cos(\theta-\phi)\cos(\theta'-\phi))f(t,\theta,\omega)f(t,\theta',\omega')\,d\theta\,d\theta'\,d\omega\,d\omega'\\
&=2K\int_{\mathbb{T}^2\times \mathbb{R}^2}(\omega-KR\sin(\theta-\phi))(\omega'-KR\sin(\theta'-\phi))\\
&\hspace{3cm}\times \cos(\theta-\theta')f(t,\theta,\omega)f(t,\theta',\omega')\,d\theta\,d\theta'\,d\omega\,d\omega',
\end{aligned}
\end{align}
and
\begin{align*}
I_2&=\int_{\mathbb{T}\times \mathbb{R}}\partial_\theta\left[(\omega-KR\sin(\theta-\phi))^2\right](\omega-KR\sin(\theta-\phi))f\,d\theta\,d\omega\\
&=-2K\int_{\mathbb{T}\times \mathbb{R}}(\omega-KR\sin(\theta-\phi))^2R\cos(\theta-\phi)f\,d\theta\,d\omega,
\end{align*}
where we have used the Kuramoto-Sakaguchi equation \eqref{E-KS} and integration by parts. Notice that by definition of the order parameter \eqref{Rdef}, we obtain
\begin{equation}\label{E-R-trick-sym}
R\cos(\theta-\phi)=\int_{\mathbb{T}\times \mathbb{R}}\cos(\theta-\theta')f(t,\theta',\omega')\,d\theta'\,d\omega'.
\end{equation}
Using such identity in the above formula for $I_2$ implies
\begin{equation}\label{E-2}
I_2=-2K\int_{\mathbb{T}^2\times \mathbb{R}^2}(\omega-KR\sin(\theta-\phi))^2\cos(\theta-\theta')f(t,\theta,\omega)f(t,\theta',\omega')\,d\theta\,d\theta'\,d\omega\,d\omega'.
\end{equation}
Let us now change variables $(\theta,\omega)$ with $(\theta',\omega')$ in \eqref{E-2} and take the mean value of both expressions for $I_2$. Since the cosine is an even function, we equivalently write
\begin{align}\label{E-3}
\begin{aligned}
&I_2=-K\int_{\mathbb{T}^2\times \mathbb{R}^2}((\omega-KR\sin(\theta-\phi))^2+(\omega'-KR\sin(\theta'-\phi))^2)\\
&\hspace{1.5cm}\times\cos(\theta-\theta')f(t,\theta,\omega)f(t,\theta',\omega')\,d\theta\,d\theta'\,d\omega\,d\omega'.
\end{aligned}
\end{align}
Finally, putting \eqref{E-1} and \eqref{E-3} together and completing the square yield the desired result.
\end{proof}

\noindent As a consequence of the previous theorem, we obtain the following quantitative behavior of the dissipation.

\begin{corollary}\label{C-dissipation-bounds}
Assume that $f_0$ is contained in $C^1(\mathbb{T}\times \mathbb{R})$ and $g$ is compactly supported in $[-W,W]$. 
Consider the unique global-in-time classical solution $f=f(t,\theta,\omega)$ to \eqref{E-KS}. Then,
\begin{equation}\label{Disb}
-2KR\,\mathcal{I}[f]\leq \frac{d}{dt}\mathcal{I}[f]\leq 2K\mathcal{I}[f],  
\end{equation}
for all $t\geq 0$. In particular,
\[\displaystyle\mathcal{I}[f](t_0)e^{-2K\int_{t_0}^t R(s)\,ds}\leq \mathcal{I}[f](t)\leq \mathcal{I}[f](t_0)e^{2K(t-t_0)},\]
for all $t\geq t_0\geq 0.$
\end{corollary}

\begin{proof}
Note that the second chain of inequalities follows from by integration on \eqref{Disb} with respect to time. Then, we focus on the proof of \eqref{Disb}, that we divide in two steps associated with the upper and lower bound respectively.

\noindent $\bullet$ \textit{Step 1}: Upper bound.

\noindent Using Theorem \ref{T-dissipation-evolution} and bounding $\cos(\theta-\theta')$ by $1$, we achieve the following upper bound for the derivative of the dissipation functional along $f$:
\begin{align*}
\frac{d}{dt}\mathcal{I}[f]&\leq\int_{\mathbb{T}^2\times \mathbb{R}^2}\left((\omega-KR\sin(\theta-\phi)-(\omega'-KR\sin(\theta'-\phi)))\right)^2\\
&\hspace{1cm}\times f(t,\theta,\omega)f(t,\theta',\omega')\,d\theta\,d\theta'\,d\omega\,d\omega'\\
&=2K\int_{\mathbb{T}\times \mathbb{R}}(\omega-KR\sin(\theta-\phi))^2f\,d\theta\,d\omega\\
&\ -2K\left(\int_{\mathbb{T}\times \mathbb{R}}(\omega-KR\sin(\theta-\phi))f\,d\theta\,d\omega\right)^2.
\end{align*}
Using the definition \eqref{Rdef} of $R$ and $\phi$ along with the assumption \eqref{E-omega-centered}, we clearly obtain that the second term vanishes and we conclude the upper bound.

\noindent $\bullet$ \textit{Step 2}: Lower bound.

\noindent Again, we shall use Theorem \ref{T-dissipation-evolution} and expand the square to obtain
\begin{align*}
\frac{d}{dt}\mathcal{I}[f]&=-2K\int_{\mathbb{T}^2\times \mathbb{R}^2}(\omega-KR\sin(\theta-\phi))^2\cos(\theta-\theta')f(t,\theta,\omega)f(t,\theta',\omega')\,d\theta\,d\theta'\,d\omega\,d\omega'\\
&\hspace{0.4cm}+2K\int_{\mathbb{T}^2\times \mathbb{R}^2}(\omega-KR\sin(\theta-\phi))(\omega'-KR\sin(\theta'-\phi))\cos(\theta-\theta')\\
&\hspace{7.5cm}\times f(t,\theta,\omega)f(t,\theta',\omega')\,d\theta\,d\theta'\,d\omega\,d\omega'\\
&=-2KR\int_{\mathbb{T}\times \mathbb{R}}(\omega-KR\sin(\theta-\phi))^2\cos(\theta-\phi)f\,d\theta\,d\omega\\
&\hspace{0.4cm}+2K\left\vert\int_{\mathbb{T}\times \mathbb{R}}(\omega-KR\sin(\theta-\phi)) e^{i(\theta-\phi)}f\,d\theta\,d\omega\right\vert^2\\
&\geq -2KR\int_{\mathbb{T}\times \mathbb{R}}(\omega-KR\sin(\theta-\phi))^2f\,d\theta\,d\omega,
\end{align*}
where in the second identity we have used \eqref{E-R-trick-sym} while in the last inequality we have bounded $\cos(\theta-\theta')$ by $1$ and we have neglected the non-negative term. Hence, the desired result follows.
\end{proof}

\subsection{A fibered Wasserstein distance and relation to dissipation}
In this section, we introduce a Wasserstein-type distance in the product space $\mathbb{T}\times \mathbb{R}$ that will play an essential role in the aforementioned dissipation-transportation inequality. This metric is constructed through a gluing procedure of the standard quadratic Wasserstein distance in $\mathbb{T}$ between conditional probabilities at any fiber $\omega\in \mathbb{R}$. Since it behaves in a fiber-wise way, we call it the fibered quadratic Wasserstein distance. See also \cite{M-17-thesis} and \cite{P-19-arxiv}, where it was introduced independently by both authors. For the reader convenience, we recall it here and introduce some of the main properties that will be used throughout the paper.

{\color{black}
\begin{definition}[{\bf Fibered quadratic Wasserstein distance}]\label{D-fibered-wasserstein}
Consider any probability measure $g\in \mathbb{P}(\mathbb{R})$ and let us define the closed subset of those probability measures $\mathbb{T}\times \mathbb{R}$ whose $\omega$-marginal agrees with $g$, i.e.,
\[\mathbb{P}_g(\mathbb{T}\times \mathbb{R}):=\{\mu\in \mathbb{P}(\mathbb{T}\times \mathbb{R}):\,(\pi_\theta)_{\#}\mu=g\}.\]
We define the fibered quadratic Wasserstein distance on $\mathbb{P}_g(\mathbb{T}\times \mathbb{R})$ as follows
\begin{equation}\label{E-fibered-Wasserstein}
W_{2,g}(\mu,\nu):=\left(\int_{\mathbb{R}}W_2(\mu(\cdot\vert\omega),\nu(\cdot\vert\omega))^2\,d_\omega g\right)^{1/2},
\end{equation}
for any $\mu,\nu\in \mathbb{P}_g(\mathbb{T}\times \mathbb{R})$. Here, we denote family of  conditional probabilities (or disintegrations) of $\mu$ with respect to the fiber $\omega\in \mathbb{R}$ as follows
\[\omega\in \mathbb{R}\longmapsto \mu(\cdot\vert\omega)\in \mathbb{P}(\mathbb{T}),\]
that is a Borel- measurable function defined by the following formula
\[\int_{\mathbb{T}\times \mathbb{R}}\varphi(\theta,\omega)\,d_{(\theta,\omega)}\mu=\int_{\mathbb{R}}\left(\int_{\mathbb{T}}\varphi(\theta,\omega)\,d_\theta \mu(\cdot\vert\omega)\right)\,d_\omega g,\]
for any test function $\varphi\in C_b(\mathbb{T}\times \mathbb{R})$.
\end{definition}
}

Like for the classical quadratic Wasserstein distance, this distance also admits an equivalent Benamou--Brenier representation (see \cite{BB-00}), that can be obtained by gluing the corresponding representations at any fiber. 


\begin{proposition}\label{P-fibered-distance-HJ}
Consider $g\in \mathbb{P}(\mathbb{R})$ and let $f^1,f^2\in \mathbb{P}_g(\mathbb{T}\times \mathbb{R})$. For $g$-a.e. value of $\omega\in \mathbb{R}$, let us consider some Wasserstein geodesic $\tau\in [0,1]\longrightarrow h_\tau(\cdot\vert \omega)\in \mathbb{P}(\mathbb{T})$ that joins the conditional probabilities with respect to $\omega$, that is 
\[h_{\tau=0}(\theta)=f^1(\cdot\vert \omega)\ \mbox{ and }\ h_{\tau=1}(\theta)=f^2(\cdot\vert\omega).\]
This is an absolutely continuous familywith respect to the Wasserstein distance on $\mathbb{T}$ and it has an associated family of potentials $\tau\in [0,T]\longrightarrow\psi_\tau(\cdot,\omega)$ so that
\begin{equation}\label{E-fibered-geodesic-HJ}
\left\{\begin{array}{l}
\displaystyle \frac{\partial}{d\tau}h_\tau(\cdot\vert \omega)+\divop_\theta\left(\nabla_\theta\psi_\tau(\cdot,\omega)h_\tau(\cdot\vert\omega)\right)=0,\\
\displaystyle \frac{\partial}{\partial \tau}\psi_\tau(\cdot,\omega)+\frac{1}{2}\vert \nabla_\theta\psi_\tau(\cdot,\omega)\vert^2=0,\ \psi_{\tau=0}(\cdot,\omega)=\psi_0(\cdot,\omega),
\end{array}\right.
\end{equation}
for some $\frac{d^2}{2}$-concave function $-\psi_0$ with respect to $\theta$, in the distributional/viscosity sense. Then, the following identity holds true
\begin{equation}\label{E-fibered-Wasserstein-HJ}
W_{2,g}(f^1,f^2)^2=\int_0^1\int_{\mathbb{T}\times \mathbb{R}}\vert \nabla_\theta\psi_\tau\vert^2\,d h_\tau\,d\tau,
\end{equation}
where we denote $h_\tau$ to the measure that can be recovered from the conditional probabilities $h_\tau(\cdot\vert\omega)$ with marginal $g$, that is, for any test function $\varphi\in C_b(\mathbb{T}\times \mathbb{R})$ the disintegration formula holds
$$\int_{\mathbb{T}\times\mathbb{R}}\varphi(\theta,\omega)\,dh_\tau=\int_\mathbb{R}\left(\int_{\mathbb{T}}\varphi(\theta,\omega)\,d_\theta h_\tau(\cdot\vert \omega)\right)\,d_\omega g.$$
\end{proposition}

\noindent Since the proof is a simple gluing procedure applied to the classical result for the quadratic Wasserstein distance, we skip it. The interested reader may want to get further details in the textbooks \cite{AGS-08,BB-00} and, \cite[Chapter 13]{V-09}.

\begin{remark}
The second equation in \eqref{E-fibered-geodesic-HJ} is called the Hamilton--Jacobi equation and using it, we observe that \eqref{E-fibered-Wasserstein-HJ} can be restated as follows
\begin{equation}\label{E-fibered-Wasserstein-HJ-2}
W_{2,g}(f^1,f^2)^2=\int_{\mathbb{T}\times \mathbb{R}}\vert \nabla_\theta\psi_\tau\vert^2\ h_\tau\,d\theta\,d\omega,
\end{equation}
for every $\tau\in [0,1]$. This suggests that the such Wasserstein geodesics have constant speed.
\end{remark}

{\color{black}
\noindent An interesting fact is that this new fibered quadratic Wasserstein distance and the classical quadratic Wasserstein distances in $\mathbb{P}_2(\mathbb{T}\times \mathbb{R})$ are appropriately ordered. Before we state the relation, let us remark the following  fact.

\begin{remark}\label{R-W2-notappropriate}
The classical quadratic Wasserstein distance $W_2$ in $\mathbb{P}_2(\mathbb{T}\times \mathbb{R})$ is defined as the transportation cost associated with the standard Riemannian distance in the product space $\mathbb{T}\times \mathbb{R}$. That is, $W_2$ is defined by
\[W_2(\mu^N_0,f_0)=\left(\inf_{\gamma\in \Pi(\mu^N_0,f_0)}\int_{\mathbb{T}^2\times \mathbb{R}^2}(d(\theta,\theta')^2+(\omega-\omega')^2)\,d\gamma\right)^{1/2},\]
for any $\mu,\nu\in \mathbb{P}_2(\mathbb{T}\times \mathbb{R})$. Here, $d(\theta,\theta^{\prime})$ denotes the canonical Riemannian distance in between any two point $\theta$ and $\theta^{\prime}$ in $\mathbb{T}.$\\

\noindent For our purposes, such distance is not appropriate as it is not dimensionally correct. Indeed, $\theta$ and $\omega$ have different physical units and considering $W_2$ causes problems to derive asymptotic behavior of solutions.
\end{remark}

\noindent The above remark suggests considering the following correction of the classical quadratic Wasserstein distance in $\mathbb{P}_2(\mathbb{T}\times \mathbb{R})$.

\begin{definition}[{\bf Scaled quadratic Wasserstein distance}]\label{D-scaled-wasserstein}
Let us consider the scaled Riemannian distance on the product space $\mathbb{T}\times \mathbb{R}$, i.e.,
\[d_K((\theta,\omega),(\theta',\omega'))=\left(d(\theta,\theta')^2+\frac{(\omega-\omega')^2}{K^2}\right)^{\frac{1}{2}}.\]
We define the scaled quadratic Wasserstein distance on $\mathbb{P}_2(\mathbb{T}\times \mathbb{R})$ by the transportation costs associated with the above scaled Riemannian distance, that is,
\[SW_2(\mu^N_0,f_0)=\left(\inf_{\gamma\in \Pi(\mu,\nu)}\int_{\mathbb{T}^2\times \mathbb{R}^2}\left(d(\theta,\theta')^2+\frac{(\omega-\omega')^2}{K^2}\right)\,d\gamma\right)^{1/2},\]
for any $\mu,\nu\in \mathbb{P}_2(\mathbb{T}\times \mathbb{R})$.
\end{definition}

\noindent We are now ready to state the relation between $SW_2$ and $W_2$.

\begin{proposition}\label{P-order-fibered-distance}
Consider $g\in \mathbb{P}_2(\mathbb{T})$. Then we obtain
$$SW_2(\mu,\nu)\leq W_{2,g}(\mu,\nu),$$
for any $\mu,\nu\in \mathbb{P}_g(\mathbb{T}\times \mathbb{R})$. In particular, we have that
\[
W_2(\mu,\nu)\leq W_{2,g}(\mu,\nu). 
\]

\end{proposition}

\begin{proof}
Consider for $g$-a.e. $\omega\in \mathbb{R}$ the optimal coupling $\gamma_{0,\omega}\in \Pi(\mu(\cdot\vert\omega),\nu(\cdot\vert\omega))$ between the conditional probabilities $\mu(\cdot\vert\omega)$ and $\nu(\cdot\vert\omega)$. Then, we can construct the probability measure $\gamma\in \mathbb{P}(\mathbb{T}^2\times \mathbb{R}^2)$ given by
\begin{equation}\label{E-Wg-1}
\gamma:=\gamma_{0,\omega}(\theta,\theta')\otimes \delta_\omega(\omega')\otimes g(\omega).
\end{equation}
Let us see first that it defines a transference plan, that is, $\gamma\in \Pi(\mu,\nu)$. To such end, consider any test function $\varphi\in C_b(\mathbb{T}\times \mathbb{R})$ and note that
\begin{align*}
\int_{\mathbb{T}\times \mathbb{R}}\varphi\,d_{(\theta,\omega)}(\pi_{(\theta,\omega)\,\#}\gamma)&=\int_{\mathbb{T}^2\times \mathbb{R}^2}\varphi(\theta,\omega)\,d_{(\theta,\theta')}\gamma_{0,\omega}\,d_{\omega'}(\delta_\omega)\,d_{\omega}g\\
&=\int_{\mathbb{T}^2\times \mathbb{R}}\varphi(\theta,\omega)\,d_{(\theta,\theta')}\gamma_{0,\omega}\,d_\omega g=\int_{\mathbb{T}\times \mathbb{R}}\varphi(\theta,\omega)\,d_\theta(\pi_{\theta\,\#}\gamma_{0,\omega})\,d_{\omega}g\\
&=\int_{\mathbb{T}\times \mathbb{R}}\varphi(\theta,\omega)\,d_\theta \mu(\cdot\vert\omega)\,d_\omega g=\int_{\mathbb{T}\times \mathbb{R}}\varphi\,d_{(\theta,\omega)}\mu.
\end{align*}
Then, $\pi_{(\theta,\omega)\,\#}\gamma=\mu$. Similarly, note that
\begin{align*}
\int_{\mathbb{T}\times \mathbb{R}}\varphi\,d_{(\theta',\omega')}(\pi_{(\theta',\omega')\,\#}\gamma)&=\int_{\mathbb{T}^2\times \mathbb{R}^2}\varphi(\theta',\omega')\,d_{(\theta,\theta')}\gamma_{0,\omega}\,d_{\omega'}(\delta_\omega)\,d_{\omega}g\\
&=\int_{\mathbb{T}^2\times \mathbb{R}}\varphi(\theta',\omega)\,d_{(\theta,\theta')}\gamma_{0,\omega}\,d_\omega g=\int_{\mathbb{T}\times \mathbb{R}}\varphi(\theta',\omega)\,d_{\theta'}(\pi_{\theta'\,\#}\gamma_{0,\omega})\,d_{\omega}g\\
&=\int_{\mathbb{T}\times \mathbb{R}}\varphi(\theta',\omega)\,d_{\theta'} \nu(\cdot\vert\omega)\,d_\omega g=\int_{\mathbb{T}\times \mathbb{R}}\varphi\,d_{(\theta',\omega')}\nu.
\end{align*}
Then we also recover $\pi_{(\theta',\omega')\,\#}\gamma=\nu$. Also note that by definition
\begin{multline*}
W_{2,g}(\mu,\nu)^2=\int_{\mathbb{R}\times \mathbb{T}^2}d(\theta,\theta')^2\,d_{(\theta,\theta')}\gamma_{0,\omega}\,d_\omega g=\int_{\mathbb{T}^2\times \mathbb{R}^2}d(\theta,\theta')^2\,d_{((\theta,\omega),(\theta',\omega'))}\gamma\\
=\int_{\mathbb{T}^2\times \mathbb{R}^2}d_K((\theta,\omega),(\theta',\omega'))\,d_{((\theta,\omega),(\theta',\omega'))}\gamma\geq SW_2(\mu,\nu)^2,
\end{multline*}
where the extra term that has been added in the second line vanishes because of the presence of $\delta_{\omega}(\omega')$ in \eqref{E-Wg-1}
\end{proof}

\noindent Indeed, the scaled and fibered Wasserstein distances are strictly ordered.

\begin{remark}
Consider the empirical measures
$$\mu:=\frac{1}{2}\left(\delta_{(\theta_1,\omega_1)}+\delta_{(\theta_2,\omega_2)}\right)\ \mbox{ and }\ \nu:=\frac{1}{2}\left(\delta_{(\theta_2,\omega_1)}+\delta_{(\theta_1,\omega_2)}\right),$$
for some $\theta_1,\theta_2\in \mathbb{T}$ and $\omega_1,\omega_2\in \mathbb{R}$ Notice that
$$\pi_{\omega\,\#}\mu=\pi_{\omega\,\#}\nu=\frac{1}{2}(\delta_{\omega_1}+\delta_{\omega_2})=:g,$$
thus, $\mu,\nu\in \mathbb{P}_g(\mathbb{T}\times \mathbb{R})$. Finally, for $\varepsilon_\theta:=d(\theta_1,\theta_2)$ and $\varepsilon_\omega:=\vert \omega_1-\omega_2\vert$ it is clear that
$$
W_{2,g}(\mu,\nu)^2=\varepsilon_\theta^2\ \mbox{ and }\ SW_{2}(\mu,\nu)^2=\frac{1}{K^2}\min\{\varepsilon_\theta^2,\varepsilon_\omega^2\}.
$$
Consequently, we obtain that
$$
\begin{array}{ll}
SW_2(\mu,\nu)<W_{2,g}(\mu,\nu), & \mbox{ if }\ \frac{\varepsilon_\omega}{K}<\varepsilon_\theta,\\
SW_2(\mu,\nu)=W_{2,g}(\mu,\nu), & \mbox{ if }\ \frac{\varepsilon_\omega}{K}\geq \varepsilon_\theta.
\end{array}
$$
\end{remark}
}

\noindent We are now ready to state the main relation between this fibered transportation distance \eqref{E-fibered-Wasserstein} and the dissipation functional \eqref{defdis}.

\begin{lemma}\label{T-Wasserstein-dissipation}
Assume that $f_0$ is contained in $C^1(\mathbb{T}\times \mathbb{R})$ and $g$ is compactly supported in $[-W,W]$. Consider the unique global-in-time classical solution $f=f(t,\theta,\omega)$ to \eqref{E-KS}. Then,
\[
\frac{d}{ds}\frac{1}{2}W_{2,g}(f_t,f_s)^2\leq \mathcal{I}[f]^{\frac{1}{2}}W_{2,g}(f_t,f_s),
\]
for every $t\geq 0$ and almost every $s\geq 0$.
\end{lemma}

\noindent A similar result was explored in \cite[Theorem 4.4]{P-19-arxiv}. There, the author used the definition of $W_{2,g}$ in \eqref{E-fibered-Wasserstein} for general measures that may enjoy atoms eventually. In this result, we sketch a simpler proof that used the representation formula of the derivative of Wasserstein distance for absolutely continuous measures, see \cite[Theorem 8.4.6]{AGS-08}, \cite[Theorem 23.9]{V-09}.

\begin{proof}[Proof of Lemma \ref{T-Wasserstein-dissipation}]
Since $f$ satisfies the Kuramoto-Sakaguchi equation \eqref{E-KS}, then each conditional probability with respect to $\omega\in \mathbb{T}$ verifies the following continuity equation
\[\frac{\partial}{\partial t} f(\theta\vert \omega)+\divop_\theta((\omega-KR\sin(\theta-\phi))e^{i\theta}f(\theta \vert\omega))=0,\]
for all $t\geq 0$ and $\theta\in \mathbb{T}$. That is, the disintegrations themselves are driven by the following tangent transport field
$$\theta\in \mathbb{T}\longmapsto v_t^\omega(\theta):=(\omega-KR\sin(\theta-\phi))e^{i\theta}.$$
Since $f$ is smooth, it is clear that the family $s\in [0,+\infty)\longmapsto f_s(\cdot\vert \omega)$ is locally absolutely continuous with respect to the quadratic Wasserstein distance on $\mathbb{T}$. This clearly guarantees that the following function is also locally absolutely continuous 
\[s\in [0,+\infty)\longrightarrow W_2(f_t(\cdot\vert \omega),f_s(\cdot\vert\omega))^2,\]
for every $\omega\in \supp g$, see \cite[Theorem 8.4.6]{AGS-08} or \cite[Theorem 23.9]{V-09}. In particular, we can take derivatives almost everywhere and obtain the formula
\begin{equation}\label{E-derivative-W-1}
\frac{d}{ds}\frac{1}{2}W_2(f_t(\cdot\vert\omega),f_s(\cdot\vert\omega))^2=-\int_{\mathbb{T}}\left<v_s^\omega(\theta),\nabla\psi_{\tau=0}^{s,t}(\theta,\omega)\right> f_s(\theta\vert \omega)\,d\theta,
\end{equation}
for almost every $t\geq 0$, where the family $\tau\in [0,1]\longmapsto (h_\tau^{s,t},\psi_\tau^{s,t})$ has been chose according to \eqref{E-fibered-geodesic-HJ} so that it represents a Wasserstein geodesic joining the conditional probabilities of $f_s$ to those of $f_t$. By the dominated convergence theorem, we can then show that the following function is also absolutely continuous
\[s\in [0,+\infty)\longrightarrow W_{2,g}(f_t,f_s)^2.\]
Integrating by parts and using \eqref{E-derivative-W-1} we obtain that
\begin{align}\label{E-derivative-W-2}
\begin{aligned}
\frac{d}{ds}\frac{1}{2}W_{2,g}(f_t,f_s)^2&=-\int_{\mathbb{T}\times\mathbb{R}}\left<v_s^\omega(\theta),\nabla\psi_{\tau=0}^{s,t}(\theta,\omega)\right> f_s(\theta\vert \omega)g(\omega)\,d\theta\,d\omega\\
&=-\int_{\mathbb{T}\times\mathbb{R}}\left<v_s^\omega(\theta),\nabla\psi_{\tau=0}^{s,t}(\theta,\omega)\right> f_s(\theta,\omega)\,d\theta\,d\omega.
\end{aligned}
\end{align}
Using the Cauchy--Schwarz inequality in \eqref{E-derivative-W-2} along with the definition of the dissipation function \eqref{defdis} and the representation of the fibered quadratic Wasserstein distance in Proposition \ref{P-fibered-distance-HJ} we obtain that 
\[
\frac{d}{ds}\frac{1}{2}W_{2,g}(f_t,f_s)^2\leq \mathcal{I}[f]^{\frac{1}{2}}W_{2,g}(f_t,f_s),
\]
for almost every $s\geq 0.$ Hence, the desired result follows.
\end{proof}

\noindent As a direct consequence of the above Lemma, we obtain the following dissipation-transportation inequality.

\begin{corollary}\label{C-Wasserstein-distance}
Assume that $f_0$ is contained in $C^1(\mathbb{T}\times \mathbb{R})$ and $g$ is compactly supported in $[-W,W]$. Consider the unique global-in-time classical solution $f=f(t,\theta,\omega)$ to \eqref{E-KS}. Then,
$$W_{2,g}(f_t,f_s)\leq \int_t^s\mathcal{I}[f_\tau]^{1/2}\,d\tau,\ \mbox{ for all }\ s\geq t.$$
\end{corollary}

\subsection{Convergence and uniqueness of the global equilibria}
In this section, we shall show the claimed result about convergence to the global equilibria. Before we proceed with the proof, let us first show that such equilibrium is unique up to phase rotations. That result is not new and was first proved in \cite{CCHKK-14} via a strict contractivity estimate in such region of convexity for an appropriate Wasserstein distance $\widetilde{W}_p$ in $\mathbb{P}_2([0,2\pi)\times \mathbb{R})$. Notice that the geometry of $\mathbb{T}$ has been disregarded in $\widetilde{W}_p$. Indeed, the distance $\widetilde{W}_2$ is strictly larger $W_{2,g}$ because the geometry of the $\mathbb{T}$ reduces the transportation cost of mass between phases separated by distances larger that $\pi$ (when viewed in the real line). We show that the uniqueness result is also true using this new fibered distance and we leave the full study of similar strict contractivityof $W_{2,g}$ to future works.

\begin{proposition}\label{unique}
Let $f_{\infty}$ and $f_{\infty}^{\prime}$ be stationary measure-valued solutions to \eqref{E-KS} and assume that they have the same distribution $g$ of natural frequencies and, $\diam (\supp_\theta f_\infty)$ and $\diam(\supp_\theta f_\infty')$ are less than $\pi/2.$ Then, they agree up to phase rotations, that is, there exists a constant $c\in \mathbb{R}$ such that
 \[
f_{\infty}^{\prime}(\theta,\omega)=f_{\infty}(\theta-c,\omega).
\]
\end{proposition}

\begin{proof}
For any $c\in \mathbb{R}$ we consider the rotation operator in the variable $\theta$
\[\mathcal{T}_c [f_\infty'](\theta,\omega):=f_\infty'(\theta-c,\omega),\]
and define the following optimization problem
\begin{equation}\label{E-rotation-same-center-mass}
\min_{c\in \mathbb{R}}W_{2,g}(f_\infty,\mathcal{T}_c[f_\infty'])^2.
\end{equation}
Such minimum of \eqref{E-rotation-same-center-mass} exists from straightforward arguments and will be achieved at some $c=c_0\in \mathbb{R}$. Without loss of generality, let us assume that $c_0=0$. Indeed, otherwise we can replace $f_\infty'$ with $\mathcal{T}_{c_0}[f_\infty']$ and it does not change thesis of this result. On the one hand, let us consider the following continuity equation
\begin{equation}\label{E-continuity-eq-rotation}
\left\{
\begin{array}{l}
\frac{\partial}{\partial s} f'_s+\divop_\theta(e^{i\theta} f'_s)=0,\\
f'_{s=0}=f_\infty',
\end{array}
\right.
\end{equation}
whose solution clearly describes the above family of phase shifts, namely, $f'_s=\mathcal{T}_s [f_\infty']$. Since $W_{2,g}(f_\infty,f_\infty')$ minimizes the problem \eqref{E-rotation-same-center-mass}, then we obtain a critical value at $c=0$, i.e.,
\begin{equation}\label{E-rotation-same-center-mass-2}
\left.\frac{d}{ds}\right\vert_{s=0}W_{2,g}(f_\infty,f'_s)^2=0.
\end{equation}
Let us write down condition \eqref{E-rotation-same-center-mass-2} more explicitly. Indeed, consider a Wasserstein geodesic that joins the conditional probability $f_\infty'(\cdot\vert\omega)$ to $f_s'(\cdot\vert\omega)$ and represent it through a family 
\begin{equation}\label{E-family-HJ-uniqueness}
\tau\in [0,T]\longrightarrow (h_\tau^s,\psi_\tau^s)\ \mbox{ with }\ \begin{array}{c}
h_{\tau=0}^s(\cdot\vert \omega)=f_\infty'(\cdot\vert\omega),\\
h_{\tau=1}^s(\cdot\vert\omega)=f_s'(\cdot\vert \omega),\end{array}
\end{equation}
as in \eqref{E-fibered-geodesic-HJ} in Proposition \ref{P-fibered-distance-HJ}. Here, although \ref{E-fibered-geodesic-HJ} holds only on the viscosity/distributional such fact can be handled by nowadays standard regularization arguments, we refer the reader to \cite[Chapter 13]{V-09}. (In particular our dissipation functional $\mathcal{I}[f]$ is continuous with respect to $W_{2,g}$ which makes it well behaved with respect to regularizations).\\ 

\noindent Now observe that, by construction $f_s'(\cdot\vert\omega),$ verifies the continuity equation \eqref{E-continuity-eq-rotation} that is driven by the trivial tangent transport field $\theta\in \mathbb{T}\longrightarrow e^{i\theta}$. Then, the same ideas in the proof of Lemma \ref{T-Wasserstein-dissipation} (see \cite[Theorem 8.4.6]{AGS-08} or \cite[Theorem 23.9]{V-09}), we obtain
\[
\left.\frac{d}{ds}\right\vert_{s=0}\frac{1}{2}W_2(f_\infty(\cdot\vert\omega),f_s'(\cdot\vert \omega))^2=\int_\mathbb{T}\left<e^{i\theta},\nabla_\theta\psi_{\tau=1}^{s=0}(\theta,\omega)\right>\,d_\theta f_\infty'(\cdot\vert\omega).
\]
for almost every $s\geq 0$. Taking integrals in $\omega$ against $g$ and using \eqref{E-rotation-same-center-mass-2} we obtain
$$
\int_\mathbb{T\times \mathbb{R}}\left<e^{i\theta},\nabla_\theta\psi_{\tau=1}^{s=0}\right>\,d_{(\theta,\omega)} f_\infty'=0.
$$
Indeed, using the equations for $h_\tau^{s=0}$ and $\varphi_\tau^{s=0}$ in \eqref{E-fibered-geodesic-HJ}, it is clear that the above implies 
\begin{equation}\label{E-rotation-same-center-mass-3}
\int_\mathbb{T\times \mathbb{R}}\left<e^{i\theta},\nabla_\theta\psi_{\tau}^{s=0}\right>\,d_{(\theta,\omega)} h_\tau^{s=0}=0,
\end{equation}
for every $\tau\in [0,1]$. On the other hand, by hypothesis $f_\infty$ and $f_\infty'$ verify the (stationary) Kuramoto-Sakaguchi equation \eqref{E-KS}, that is,
\begin{align*}
\frac{\partial}{\partial t} f_\infty+\text{div}_\theta((\omega-KR_\infty\sin(\theta-\phi_\infty))e^{i\theta}f_\infty)&=0,\\
\frac{\partial}{\partial t} f_\infty'+\text{div}_\theta((\omega-KR_\infty'\sin(\theta-\phi_\infty'))e^{i\theta}f_\infty')&=0.
\end{align*}
Since the solutions are stationary, then we can again use the same ideas as before to arrive at the identity
\begin{align*}
0=\frac{d}{dt}\frac{1}{2}W_2(f_\infty(\cdot\vert \omega),f_\infty'(\cdot \vert \omega))^2&=\int_\mathbb{T}\left<(\omega-KR_\infty'\sin(\theta-\phi_\infty'))e^{i\theta},\nabla_\theta\psi_{\tau=1}^{s=0}(\cdot,\omega)\right>\,d_\theta f_\infty'(\cdot\vert \omega)\\
&-\int_\mathbb{T}\left<(\omega-KR_\infty\sin(\theta-\phi_\infty))e^{i\theta},\nabla_\theta\psi_{\tau=0}^{s=0}(\cdot,\omega)\right>\,d_\theta f_\infty(\cdot\vert \omega),
\end{align*}
Here on we shall omit the superscripts $s=0$ of $h_\tau^{s=0}$ and $\psi_\tau^{s=0}$ for simplicity, as it is clear from the context. Then, integrating against $g$ and using the fundamental theorem of calculus in $\tau$ yields
\begin{equation}\label{E-unique-1}
\int_0^1\frac{d}{d\tau}\int_{\mathbb{T}\times \mathbb{R}}\left<(\omega-KR_\tau\sin(\theta-\phi_\tau))e^{i\theta},\nabla_\theta\varphi_\tau\right>\,d_{(\theta,\omega)}h_\tau\,d\tau=0,
\end{equation}
where $R_\tau$ and $\phi_\tau$ are order parameters associated with the displacement interpolation $h_\tau$. Let us now expand the derivative in \eqref{E-unique-1} and use the Hamilton--Jacobi equation for $\psi_\tau$ and the continuity equation for $h_\tau$ in \eqref{E-fibered-geodesic-HJ} (see \cite[Chapter 13]{V-09}). Then we obtain that
$$A+B+C=0,$$
where each term reads
\begin{align*}
A&:=\int_0^1\int_{\mathbb{T}\times \mathbb{R}} \left<\nabla_\theta\left(-\frac{1}{2}\left\vert \nabla_\theta\psi_\tau\right\vert^2\right),(\omega-KR_\tau\sin(\theta-\phi_\tau)e^{i\theta}\right>\,d_{(\theta,\omega)}h_\tau\,d\tau,\\
B&:=\int_0^1\int_{\mathbb{T}\times \mathbb{R}}\left<\frac{d}{d\tau}\left[\omega-KR_\tau\sin(\theta-\phi_\tau)\right]e^{i\theta},\nabla_\theta\psi_\tau\right>\,d_{(\theta,\omega)}h_\tau\,d\tau,\\
C&:=\int_0^1\int_{\mathbb{T}\times \mathbb{R}}\left<\nabla_\theta\left<\nabla_\theta\psi_\tau, (\omega-KR_\tau\sin(\theta-\phi_\tau))e^{i\theta}\right>,\nabla_\theta \psi_\tau\right>\,d_{(\theta,\omega)}h_\tau\,d\tau.
\end{align*}
On the one hand, taking the sum of $A$ and $C$ we can simplify into
\begin{align}
A+C&=-K\int_0^1\int_{\mathbb{T}\times \mathbb{R}}R_\tau\cos(\theta-\phi_\tau)\left\vert \nabla_\theta\psi_\tau\right\vert^2\,d_{(\theta,\omega)}h_\tau\,d\tau\nonumber\\
&=-K\int_0^1\int_{\mathbb{T}\times \mathbb{R}}\int_{\mathbb{T}\times \mathbb{R}}\cos(\theta-\theta')\left\vert \nabla_\theta\psi_\tau\right\vert^2\,d_{(\theta,\omega)}h_\tau\,d_{(\theta',\omega')}h_\tau\,d\tau\nonumber\\
&=-\frac{K}{2}\int_0^1\int_{\mathbb{T}\times \mathbb{R}}\int_{\mathbb{T}\times \mathbb{R}}\cos(\theta-\theta')\left(\left\vert \nabla_\theta\psi_\tau(\theta,\omega)\right\vert^2+\left\vert \nabla_\theta\psi_\tau(\theta',\omega')\right\vert^2\right)\,d_{(\theta,\omega)}h_\tau\,d_{(\theta',\omega')}h_\tau\,d\tau,\label{E-unique-2}
\end{align}
where in the second line we have used the properties of the order parameters $R_\tau$ and $\phi_\tau$ of the interpolation $h_\tau$, namely
\begin{align*}
R_\tau&=\int_{\mathbb{T}\times \mathbb{R}}\cos(\theta'-\phi_\tau)\,d_{(\theta',\omega')}h_\tau,\\
0&=\int_{\mathbb{T}\times \mathbb{R}}\sin(\theta'-\phi_\tau)\,d_{(\theta',\omega')}h_\tau.
\end{align*}
and in the third line we have used a clear symmetrization argument. Let us now differentiate with respect to $\tau$ and use the continuity equation for $h_\tau$ to obtain the formulas
\begin{align*}
\frac{dR_\tau}{d\tau}&=-\int_{\mathbb{T}\times \mathbb{R}}\sin(\theta'-\phi_\tau)\left<e^{i\theta'},\nabla_{\theta}\psi_\tau(\theta',\omega')\right>\,d_{(\theta',\omega')}h_\tau,\\
R_\tau\frac{d\phi_\tau}{d\tau}&=\int_{\mathbb{T}\times \mathbb{R}}\cos(\theta'-\phi_\tau)\left<e^{i\theta'},\nabla_{\theta}\psi_\tau(\theta',\omega')\right>\,d_{(\theta',\omega')}h_\tau.
\end{align*} 
Then, the term $B$ can be written as follows
\begin{align}
B&=\int_0^1\int_{\mathbb{T}\times \mathbb{R}}\left<e^{i\theta},\nabla_\theta\psi_\tau\right>\frac{d}{d\tau}\left(-K\frac{dR_\tau}{d\tau}\sin(\theta-\phi_\tau)+KR_\tau\frac{d\phi_\tau}{d\tau}\cos(\theta-\phi_\tau)\right)\,d_{(\theta,\omega)}h_\tau\,d\tau\nonumber\\
&=K\int_0^1\int_{\mathbb{T}\times \mathbb{R}}\int_{\mathbb{T}\times \mathbb{R}}\cos(\theta-\theta')\left<e^{i\theta},\nabla_{\theta}\psi_\tau(\theta,\omega)\right>\left<e^{i\theta'},\nabla_{\theta}\psi_\tau(\theta',\omega')\right>\,d_{(\theta,\omega)}h_\tau\,d_{(\theta',\omega')}h_\tau\,d\tau\label{E-unique-3}
\end{align}
Putting the formulas\eqref{E-unique-2} and \eqref{E-unique-3} into \eqref{E-unique-1} entails
\begin{multline}\label{E-unique-4}
0=-\frac{K}{2}\int_0^1\int_{\mathbb{T}\times \mathbb{R}}\int_{\mathbb{T}\times \mathbb{R}}\cos(\theta-\theta')\left(\left<e^{i\theta},\nabla_{\theta}\psi_\tau(\theta,\omega)\right>-\left<e^{i\theta'},\nabla_{\theta}\psi_\tau(\theta',\omega')\right>\right)^2\\
\times d_{(\theta,\omega)}h_\tau\,d_{(\theta',\omega')}h_\tau\,d\tau.
\end{multline}
Since there exists $0<\delta<\pi/2$ such that 
$$\text{diam}(\text{supp}_\theta f_\infty)<\delta\ \mbox{ and }\ \text{diam}(\text{supp}_\theta f_\infty')<\delta.$$
The same is true for the interpolations $h_\tau$ and, consequently. Indeed, this is a consequence of the monotone rearrangement property of the 1-dimensional transport on each fiber.  Hence, we can take upper bounds in \eqref{E-unique-4} and obtain that
\begin{align*}
0&\leq -\frac{K}{2}\cos \delta\int_0^1\int_{\mathbb{T}\times \mathbb{R}}\int_{\mathbb{T}\times \mathbb{R}}\left(\left<e^{i\theta},\nabla_{\theta}\psi_\tau(\theta,\omega)\right>-\left<e^{i\theta'},\nabla_{\theta}\psi_\tau(\theta',\omega')\right>\right)^2\,d_{(\theta,\omega)}h_\tau\,d_{(\theta',\omega')}h_\tau\,d\tau\\
&=-K\cos \delta\int_0^1\int_{\mathbb{T}\times \mathbb{R}}\vert \nabla_\theta\psi_\tau\vert^2\,d_{(\theta,\omega)}h_\tau\,d\tau+K\cos\delta\int_0^1\left(\int_\mathbb{T\times \mathbb{R}}\left<e^{i\theta},\nabla_\theta\psi_\tau\right>\,d_{(\theta,\omega)} h_\tau\right)^2\,d\tau.
\end{align*}
Notice that the condition \eqref{E-rotation-same-center-mass-3} allows neglecting the second term. Also, notice that the cosine has positive sign and hence,
$$\nabla_\theta \psi_\tau^{s=0}=0, \mbox{ for }d\tau\otimes h_\tau^{s=0}\mbox{-a.e. }\ (\tau,\theta,\omega)\in [0,1]\times \mathbb{T}\times \mathbb{R}.$$
In particular, the continuity equation for $h_\tau^{s=0}$ implies that
$$f_\infty=h_\tau^{s=0}=f_\infty', \mbox{ for all }\tau\in [0,1],$$
thus ending the proof.
\end{proof}

%
%
%

\noindent We now come back to the proof of Theorem \ref{mainresult}. First, we show that once the concentration regime in Corollary \ref{r_*bounda} takes place, Theorem \ref{T-dissipation-evolution} guaranteed that the dissipation decays exponentially fast.

\begin{corollary}\label{C-exponential-dissipation}
Assume that $f_0$ is contained in $C^1(\mathbb{T}\times \mathbb{R})$ and $g$ is compactly supported in $[-W,W]$ and centered (i.e., \eqref{E-omega-centered}). Consider the unique global-in-time classical solution $f=f(t,\theta,\omega)$ to \eqref{E-KS}. Then, the following holds true
\[\frac{d\mathcal{I}[f]}{dt}\leq-2K\cos(\beta)\,\mathcal{I}[f]+24K (W+K)^2\rho_t(\mathbb{T}\setminus L^+_\beta(t)),\]
for every $t\geq 0$.
\end{corollary}

\begin{proof}
Set $\beta=\frac{\pi}{3}$ and use Theorem \ref{T-dissipation-evolution} to split the derivative of the dissipation functional into two parts as follows
$$\frac{d\mathcal{I}[f]}{dt}=I_1+I_2,$$
where each factor reads
\begin{align*}
I_1&=-K\int_{L^+_\beta(t)\times L^+_\beta(t)\times \mathbb{R}\times \mathbb{R}}\left((\omega-KR\sin(\theta-\phi)-(\omega'-KR\sin(\theta'-\phi)))\right)^2\\
&\hspace{2cm}\times\cos(\theta-\theta')f(t,\theta,\omega)f(t,\theta',\omega')\,d\theta\,d\theta'\,d\omega\,d\omega',\\
I_2&=-K\int_{((\mathbb{T}\times \mathbb{T})\setminus(L^+_\beta(t)\times L^+_\beta(t)))\times \mathbb{R}\times \mathbb{R}}\left((\omega-KR\sin(\theta-\phi)-(\omega'-KR\sin(\theta'-\phi)))\right)^2\\
&\hspace{2cm}\times\cos(\theta-\theta')f(t,\theta,\omega)f(t,\theta',\omega')\,d\theta\,d\theta'\,d\omega\,d\omega'.
\end{align*}
On the one hand, it is clear that
\begin{align}\label{E-4}
\begin{aligned}
I_1&\leq-K\cos(\beta)\int_{L^+_\beta(t)\times L^+_\beta(t)\times \mathbb{R}\times \mathbb{R}}\left((\omega-KR\sin(\theta-\phi)-(\omega'-KR\sin(\theta'-\phi)))\right)^2\\
&\hspace{6cm}\times f(t,\theta,\omega)f(t,\theta',\omega')\,d\theta\,d\theta'\,d\omega\,d\omega'\\
&=-K\cos(\beta)\,\int_{\mathbb{T}^2\times \mathbb{R}^2}\left((\omega-KR\sin(\theta-\phi)-(\omega'-KR\sin(\theta'-\phi)))\right)^2\\
&\hspace{6cm}\times f(t,\theta,\omega)f(t,\theta',\omega')\,d\theta\,d\theta'\,d\omega\,d\omega'\\
&\hspace{0.4cm}+ K\cos(\beta)\int_{((\mathbb{T}\times \mathbb{T})\setminus(L^+_\beta(t)\times L^+_\beta(t)))\times \mathbb{R}\times \mathbb{R}}\left((\omega-KR\sin(\theta-\phi)-(\omega'-KR\sin(\theta'-\phi)))\right)^2\\
&\hspace{6cm}\times f(t,\theta,\omega)f(t,\theta',\omega')\,d\theta\,d\theta'\,d\omega\,d\omega'\\
&=:I_{11}+I_{12},
\end{aligned}
\end{align}
where in the second identity we have added and subtracted the second term in order to complete an integral in $\mathbb{T}^2\times \mathbb{R}^2$. Indeed, notice that doing so and using \eqref{E-omega-centered} we get
\begin{align*}
I_{11}&=-K\cos(\beta)\int_{\mathbb{T}^2\times \mathbb{R}^2}\left((\omega-KR\sin(\theta-\phi)-(\omega'-KR\sin(\theta'-\phi)))\right)^2\\
&\hspace{5cm}\times f(t,\theta,\omega)f(t,\theta',\omega')\,d\theta\,d\theta'\,d\omega\,d\omega'\\
&=-2K\cos(\beta)\int_{\mathbb{T}\times \mathbb{R}}(\omega-KR\sin(\theta-\phi))^2f\,d\theta\,d\omega=-2K\cos(\beta)\mathcal{I}[f].
\end{align*}
Here, we have used the cancellation of the crossed term after we expand the square appearing in the first factor. Let us call $I_3=I_{12}+I_2$ and notice that
\begin{multline*}
I_3\leq 2K\int_{((\mathbb{T}\times \mathbb{T})\setminus(L^+_\beta(t)\times L^+_\beta(t)))\times \mathbb{R}\times \mathbb{R}}\left((\omega-KR\sin(\theta-\phi)-(\omega'-KR\sin(\theta'-\phi)))\right)^2\\
\times f(t,\theta,\omega)f(t,\theta',\omega')\,d\theta\,d\theta'\,d\omega\,d\omega'.
\end{multline*}
In other words, we achieved the estimate
\begin{equation}\label{E-disspation-conc-1}
\frac{d\mathcal{I}[f]}{dt}\leq-2K\cos(\beta)\, \mathcal{I}[f]+I_3.
\end{equation}
Our last goal is to estimate the remainder $I_3$. Define the following time-dependent sets
\begin{align*}
A_1&:=L^+_\beta(t)\times (\mathbb{T}\setminus L^+_\beta(t))\times \mathbb{R}\times \mathbb{R},\\
A_2&:=(\mathbb{T}\setminus L^+_\beta(t))\times L^+_\beta(t)\times \mathbb{R}\times \mathbb{R},\\
A_3&:=(\mathbb{T}\setminus L_\beta^+(t))\times (\mathbb{T}\setminus L^+_\beta(t))\times \mathbb{T}\times \mathbb{R}.
\end{align*}

\noindent Since we have that $((\mathbb{T}\times \mathbb{T})\setminus(L^+_\beta(t)\times L^+_\beta(t)))\times \mathbb{R}\times \mathbb{R}=A_1\cup A_2\cup A_3$, then we can split $I_3$ as follows
\[I_3\leq I_{31}+I_{32}+I_{33},\]
where each integral takes the following form
\begin{align*}\begin{aligned}I_{3i}:=2K\int_{A_{i}}\big((\omega&-KR\sin(\theta-\phi)\\
 & -(\omega'-KR\sin(\theta'-\phi))\big)f(t,\theta,\omega)f(t,\theta',\omega')\,d\theta\,d\theta'\,d\omega\,d\omega',\\
\end{aligned}
\end{align*}
for every $i=1,2$. Changing variables we observe that $I_{31}=I_{32}$. Then we can focus on estimating $I_{31}$ and $I_{33}$ only. Notice that the integrand can be bounded as follows
$$\left((\omega-KR\sin(\theta-\phi))-(\omega'-KR\sin(\theta'-\phi))\right)^2 \leq 4(W+K)^2.$$
Then, we obtain
\[I_{31}(t)\leq 8K(W+K)^2\rho_t(\mathbb{T}\setminus L^+_\beta(t)),\]
for every $t\geq 0$. Exactly the same argument allows estimating $I_{33}$ and obtaining an identical bound. Putting everything together into \eqref{E-disspation-conc-1} finishes the proof.
\end{proof}

\noindent Now, we can apply Gronwall's lemma in order to derive the desired quantitative estimate on the decay rate of the dissipation.

\begin{corollary}\label{C-exponential-dissipation-2}
Assume that $f_0$ is contained in $C^1(\mathbb{T}\times \mathbb{R})$ and $g$ is compactly supported in $[-W,W]$ and centered (i.e., \eqref{E-omega-centered}). Consider the unique global-in-time classical solution $f=f(t,\theta,\omega)$ to \eqref{E-KS}. Then, there is a universal constant $C$ such that if
\[\frac{W}{K}\leq CR_0^3,\]
then there exists a time $T_0$ with the property that
\[T_0\lesssim\frac{1}{KR_0^2}\log\left(1+\frac{1}{R_0}+W^{1/2}\Vert f_0\Vert_2\right),\]
and
\[\mathcal{I}[f_t]\lesssim K^2e^{-\frac{1}{20}K(t-T_0)}.\]
\end{corollary}

\begin{proof}
Let us adjust $C$ small enough so that we meet the hypotheses of Corollary \ref{r_*bounda}. Then, there exists such a time $T_0$ so that
\[\rho_{t}(\mathbb{T}\setminus L_{\alpha}^{+}(t))\leq Me^{-\frac{1}{20}K(t-T_{0})},\]
for every $t\geq T_0$ and some universal constant $M$. This along with Corollary \ref{C-exponential-dissipation} implies
\[\frac{d}{dt}\mathcal{I}[f]\leq-2K\cos(\beta)\mathcal{I}[f]+24K(W+K)^{2}Me^{-\frac{1}{20}K(t-T_{0})},\]
for any $t\geq T_0$. Integrating the inequality, we obtain that
\begin{align*}\mathcal{I}[f_{t}] & \leq\mathcal{I}[f_{T_{0}}]e^{-2K\cos(\beta)(t-T_{0})}\\
 & \ +\frac{24K(W+K)^{2}M}{2K\cos(\beta)-\frac{1}{20}K}\left(e^{-\frac{K}{20}(t-T_{0})}-e^{-2K\cos(\beta)(t-T_{0})}\right),\\
 & \lesssim(W+K)^{2}e^{-\frac{K}{20}(t-T_{0})}\lesssim K^{2}e^{-\frac{K}{20}(t-T_{0})},
\end{align*}
where in the second inequality we have used that
\[\mathcal{I}[f_{T_0}]\leq (W+K)^2,\]
by the definition \eqref{defdis} and in the second inequality we have used the hypothesis on $\frac{W}{K}$.
\end{proof}

\noindent Using the transportation-dissipation inequality in Corollary \ref{C-Wasserstein-distance} and the above exponential decay of the dissipation in Corollary \ref{C-exponential-dissipation-2} we obtain the following result.

\begin{corollary}\label{C-exponential-cauchy-condition}
Assume that the hypotheses in Corollary \ref{C-exponential-dissipation-2} hold true. Then,
\[W_{2,g}(f_{t},f_{s})\lesssim e^{-\frac{1}{40}K(t-T_{0})}-e^{-\frac{1}{40}K(s-T_{0})},\]
for every $s\geq t\geq T_0$.
\end{corollary}

\noindent We are now ready to conclude the proof of the main theorem of this paper.

\begin{proof}[Proof of Theorem \ref{mainresult}]
~\\

\noindent $\bullet$ \textit{Step 1} Convergence.

\noindent By the above Corollary \ref{C-exponential-cauchy-condition}, the net $(f_t)_{t\geq 0}$ verifies the Cauchy condition in the metric space $(\mathbb{P}_g(\mathbb{T}\times \mathbb{R}),W_{2,g})$. Notice that it is a complete metric space. Consequently, there exists some probability measure $f_\infty\in \mathbb{P}_g(\mathbb{T}\times \mathbb{R})$ such that $W_{2,g}(f_t,f_\infty)\rightarrow 0$ as $t\rightarrow\infty$. Taking limits in the inequality in Corollary \ref{C-exponential-cauchy-condition} as $s\rightarrow\infty$ yields
\begin{equation}\label{fibered_conv} W_{2,g}(f_{t},f_{\infty})\lesssim e^{-\frac{1}{40}K(t-T_{0})}, 
\end{equation}
for every $t\geq T_0$ and using the order relation in Proposition \ref{P-order-fibered-distance} between the standard quadratic Wasserstein distance and the fibered quadratic Wasserstein distance concludes the exponential convergence in Theorem \ref{mainresult}. \\

\noindent $\bullet$ \textit{Step 2} Uniqueness of the equilibrium.

\noindent Notice that, in particular, $f_\infty$ is an equilibrium of the Kuramoto-Sakaguchi equation \eqref{E-KS} and the asymptotic concentration estimate in Corollary \ref{r_*bounda} guarantees that
\[\diam(\supp_\theta f_\infty)\leq \beta=\frac{\pi}{3}<\frac{\pi}{2}.\]
Hence, by Proposition \ref{unique} it is unique up to phase shifts.  
\end{proof}
}

\subsection{Semiconcavity, entropy production estimate, and lower bounds in the order parameter}

The main objective of this part is the proof of the entropy production estimate Lemma \ref{Rgain}. As a byproduct in Corollary \ref{gain_vs_loss} we will obtain a universal lower bound on the order parameter. Before we begin the proof of the entropy production estimate, we will need a relationship between the time derivative of the order parameter and the dissipation  \eqref{DisRbis}. That is the content of the following lemma.

\begin{lemma}\label{dis-dR}

Assume that $f_{0}$ is contained in $C^{1}(\mathbb{T}\times\mathbb{R})$ and that $g$ is compactly supported in $[-W,W].$ Then, the inequality
\begin{equation}\label{DisR}
\mathcal{I}[f_{t}]-W^{2}\leq K\frac{d}{dt}(R^{2})\leq3\,\mathcal{I}[f_{t}]+W^{2},
\end{equation}
holds.
\end{lemma}
\begin{proof} By \eqref{para} we have that
\begin{align*}\begin{aligned}\frac{1}{2}\frac{d}{dt}KR^{2} & =-\int KR\sin(\theta-\phi)(\omega-KR\sin(\theta-\phi))fd\theta d\omega\\
 & =\mathcal{I}[f]-\int\omega(\omega-KR\sin(\theta-\phi))fd\theta d\omega.\\
\end{aligned}
\end{align*}
Consequently, by young's inequality, we obtain that 
\[
\frac{1}{2}\frac{d}{dt}KR^{2}\leq\mathcal{I}[f]+\frac{1}{2}\int(\omega-KR\sin(\theta-\phi))^{2}fd\theta d\omega+\frac{1}{2}\int\omega^{2}fd\theta d\omega,
\]
and
\[
\frac{1}{2}\frac{d}{dt}KR^{2}\geq\mathcal{I}[f]-\frac{1}{2}\int(\omega-KR\sin(\theta-\phi))^{2}fd\theta d\omega-\frac{1}{2}\int\omega^{2}fd\theta d\omega.
\]
Hence, the desired result follows.
\end{proof}
\noindent {\color{black} Now we are ready to prove our entropy production estimate.\\ 

\noindent \textbf{Proof of Lemma \ref{Rgain}}\\

\noindent Without loss of generality, we can assume that 
\begin{equation}\label{G2R}R<\frac{3}{2} R_{0}\hspace{1em}\text{in}\hspace{1em}\bigg[t_{0},t_{0}+\frac{1}{3KR_{0}}\log10\bigg].
\end{equation}
Otherwise, if this condition fails for some $s$ in the above interval, then we set $d=s-t_{0}$ and \eqref{E-semiconcavity-gain} would follow.
Thanks to the inequalities \eqref{Disb} and \eqref{DisR}, we arrive at the following estimate
\begin{align*}\begin{aligned}\frac{dR^{2}}{dt} & \geq\frac{\mathcal{I}[f_{t}]}{K}-\frac{W^{2}}{K}\geq\frac{\mathcal{I}[f_{t_{0}}]e^{-3KR_{0}(t-t_{0})}}{K}-\frac{W^{2}}{K}\\
 & \geq\frac{1}{K}\left(\frac{K}{3}\frac{dR^{2}}{dt}\bigg|_{t=t_{0}}-\frac{W^{2}}{3}\right)e^{-3KR_{0}(t-t_{0})}-\frac{W^{2}}{K}\\
 & =\frac{2}{3}R(t_{0})\dot{R}(t_{0})e^{-3KR_{0}(t-t_{0})}-\frac{4W^{2}}{3K}\\
 & \geq\frac{K}{6}\cos^{2}\alpha\lambda^{3}R_{0}^{3}R(t_{0})e^{-3KR_{0}(t-t_{0})}-\frac{4W^{2}}{3K}.
\end{aligned}
\end{align*}
Let us integrate the above inequality on the interval $[t_{0},t_{0}+d]$ for some $d$ in $[0,\frac{1}{3KR_{0}}\log 10),$ which we will choose appropriately after the calculations below. By doing this and using \eqref{G2R}, we deduce that 
\[R^{2}(t_{0}+d)-R^{2}(t_{0})\geq\frac{1}{18}\cos^{2}\alpha\lambda^{4}R_{0}^{3}\bigg[1-e^{-3KR_{0}d}\bigg]-\frac{4}{3}\frac{W^{2}}{K}d. 
\]
Thus, by choosing $d=\frac{1}{3KR_{0}}\log10,$ we obtain that
\[R^{2}(t_{0}+d)-R^{2}(t_{0})\geq\frac{1}{20}\cos^{2}\alpha\lambda^{4}R_{0}^{3}-\frac{4}{9}\frac{W^{2}}{K^{2}R_{0}}\log 10.
\]
Consequently, by selecting $C$ appropriately in \eqref{Gain_cali} we conclude that
\[R^{2}(t_{0}+d)-R^{2}(t_{0})\geq\frac{1}{21}\cos^{2}\alpha\lambda^{4}R_{0}^{3}.\ 
\]
Hence, since $\alpha=\pi/6$ the desired result follows.
$\square.$\\

\noindent Before showing the lower bound in the order parameter, we will need control in its angular velocity in the small dissipation regime. We achieve this in the following lemma

{\color{black}
\begin{lemma}\label{dRdphi}
Assume that $f_{0}$ is contained in $C^{1}(\mathbb{T}\times\mathbb{R})$ and that $g$ is compactly supported in $[-W,W]$. Consider the unique global-in-time classical solution $f=f(t,\theta,\omega)$ to \eqref{E-KS}. Then, we have that
\[|\dot{\phi}|\leq\frac{1}{R}\sqrt{K\frac{d}{dt}R^{2}+W^{2}}.
\]
\begin{proof}
By \eqref{para}, and Jensen inequality, we have that
\begin{align*}
\begin{aligned}R|\dot{\phi}| & \leq\int|\cos(\theta-\phi)(\omega-KR\sin(\theta-\phi))|f\hspace{1mm}d\theta d\omega\\
 & \leq\text{\ensuremath{\int|(\omega-KR\sin(\theta-\phi))|f\hspace{1mm}d\theta d\omega}}\\
 & \leq\text{\ensuremath{\bigg(}\ensuremath{\ensuremath{\int}|(\ensuremath{\omega}-KR\ensuremath{\sin}(\ensuremath{\theta}-\ensuremath{\phi}))\ensuremath{|^{2}}fd\ensuremath{\theta}d\ensuremath{\omega}}}\bigg)^{\frac{1}{2}}\\
 & = I^{\frac{1}{2}}\\
 & \leq \sqrt{K\frac{d}{dt}R^{2}+W^{2}},
\end{aligned}
\end{align*}
where in the last inequality, we have used $\eqref{DisR}.$
Thus, the desired result follows.
\end{proof}
\end{lemma}

\noindent We will derive a global lower bound on the order parameter as an application of the  entropy production estimate \eqref{Rgain}. To achieve this, we consider the following lemma, which controls the rate at which the order parameter can decrease. }

{\color{black}
\begin{lemma}\textbf{(Rate of decrease and mass monotonicity)}\label{Rdecrease}
{\color{black}Let $\lambda$ be contained in $(2/3,1),$}   assume that  $f_{0}$ is contained $C^{1}(\mathbb{T}\times\mathbb{R})$ and that $g$ is compactly supported in $[-W,W]$. Consider the unique global-in-time classical solution $f=f(t,\theta,\omega)$ to \eqref{E-KS}. Additionally, let $\gamma$ be a positive number in $(\pi/6,\pi/2),$ and let $\alpha$ be as specified in Section 2. Then, we have that
\begin{equation}\label{decRa}\frac{d}{dt}R^{2}\geq\frac{KR^{2}\cos{}^{2}\gamma}{2}\bigg(1-\frac{2W^{2}}{K^{2}R^{2}\cos^{2}\gamma}-\frac{R}{\sin\gamma}-\frac{1+\sin\gamma}{\sin\gamma}f(\chi_{\alpha}^{-})\bigg),
\end{equation}
and
\begin{equation}\label{Inv}\frac{d}{dt}f(\chi_{\alpha}^{-})\leq 4K\bigg[\frac{W}{K}+\sqrt{\frac{\dot{2R}}{KR}+\frac{1}{R^{2}}\frac{W^{2}}{K^{2}}}-R\cos\alpha\bigg]^{+},
\end{equation}
for all $t\geq0.$\\

\noindent Moreover, suppose that $\dot{R}(t_{0})\leq0,$ $R(t_{0})\geq R_{0},$
\[\dot{R}\leq\frac{K\cos^{2}\alpha\lambda^{3}R_{0}^{3}}{4}\hspace{1em}\text{in}\hspace{1em}[t_{0},t_{0}+d]\ \text{and\hspace{1em}\ensuremath{\cos^{2}\gamma}=\ensuremath{\frac{1-\lambda}{5}R_{0}},}
\]
for some non-negative numbers $d$ and $t_{0}.$ Then, there exist a universal constant $C$ such that if 
\begin{equation}\label{deccon}\frac{W}{K}\leq C(1-\lambda)\lambda^{2}R_{0}^{2},
\end{equation}
then, 
\begin{equation}\label{decRb}\frac{d}{dt}R^{2}>\frac{K\cos^{2}\gamma}{2\sin\gamma}\bigg(-R^{3}+[\lambda R_{0}+\frac{3}{5}(1-\lambda)R_{0}]R^{2}-\frac{3}{5}(1-\lambda)\lambda^{2}R_{0}^{3}\bigg),
\end{equation} 
in $[t_{0},t_{0}+d].$ Consequently,
\[R\geq\lambda R_{0} \hspace{1em}\text{in\hspace{1em}}[t_{0},t_{0}+d).\ \ \ 
\]
\end{lemma}
\begin{proof}
We divide the proof into the following steps:\\
\noindent $\bullet$ \textit{Step 1}: Derivation of estimate \eqref{Inv}.\\
Recall that $\chi_{\alpha}^{-}(\theta)=\xi_{\alpha}(\theta-\phi-\pi)$,  with $\xi_{\alpha}$ as defined in \eqref{chi_def}. Then, by direct computation, we have that
\begin{align*}\begin{aligned}\frac{d}{dt}f(\chi_{\alpha}^{-}) & =\frac{d}{dt}\int_{\mathbb{T}\times\mathbb{R}}\xi_{\alpha}(\theta-\phi-\pi)f\hspace{1mm}d\theta d\omega\\
 & =\int_{\mathbb{T}\times\mathbb{R}}\xi_{\alpha}^{\prime}(\theta-\phi-\pi)[\omega-KR\sin(\theta-\phi)-\dot{\phi}]f\hspace{1mm}d\theta d\omega\\
 & \leq f(|\xi_{\alpha}^{\prime}|)[W+\vert\dot{\phi}\vert]+KR\int_{\mathbb{T}\times\mathbb{R}}\xi_{\alpha}'(\theta-\phi-\pi)\sin(\theta-\phi-\pi)f\,\hspace{1mm}d\theta d\omega\\
 & \leq f(|\xi_{\alpha}^{\prime}|)[W+|\dot{\phi}|-KR\cos\alpha]\\
 & \leq f(|\xi_{\alpha}^{\prime}|)\bigg[W+\frac{1}{R}\sqrt{2KR\frac{d}{dt}R+W^{2}}-KR\cos\alpha\bigg].
\end{aligned}
\end{align*}
Notice that in the last inequality we have used Lemma \ref{dRdphi} in order to estimate $\vert\dot{\phi}\vert$ and the only thing that remains to show is the bound of the second term in the third line. Firstly, the support of $\xi_{\alpha}'(\theta-\phi-\pi)$ consists of $S^{+}\cup S^{-}$ where each set stands for
\[S^{+}:=\left[\phi+\frac{3\pi}{2}-\alpha,\phi+\frac{3\pi}{2}-\alpha+\frac{1}{2}\right]\hspace{1em}\text{and}\hspace{1em}S^{-}:=\left[\phi+\frac{\pi}{2}+\alpha-\frac{1}{2},\phi+\frac{\pi}{2}+\alpha\right].
\]
Since $\xi_{\alpha}'(\theta-\phi-\pi)$ is non-increasing in $S^{+}$ and non-decreasing in $S^{-}$,  we then obtain
\[\begin{array}{l}
\theta\in S^{+}\ \Longrightarrow\ \xi_{\alpha}'(\theta-\phi-\pi)\leq0\ \mbox{ and }\ \sin(\theta-\phi-\pi)\geq\cos\alpha,\\
\theta\in S^{-}\ \Longrightarrow\ \xi_{\alpha}'(\theta-\phi-\pi)\geq0\ \mbox{ and }\ \sin(\theta-\phi-\pi)\leq-\cos\alpha.
\end{array}\ \ 
\]
Consequently,
\[\xi_{\alpha}'(\theta-\phi-\pi)\sin(\theta-\phi-\pi)\leq-\vert\xi_{\alpha}'(\theta-\phi-\pi)\vert\cos\alpha,\ \ \ 
\]
for all $\theta\in S^{+}\cup S^{-}$,  thus yielding the aforementioned bound
Hence, \eqref{Inv} follows. \\

\noindent $\bullet$ \textit{Step 2:} Derivation of estimate \eqref{decRa}.\\
By the first equation in \eqref{para}, we obtain the following lower bound on $\dot{R}$ 
\begin{align*}
\begin{aligned}\frac{K}{2}\frac{d}{dt}R^{2} & =-\int_{\mathbb{T}\times\mathbb{R}}KR\sin(\theta-\phi)(\omega-KR\sin(\theta-\phi))f\,\hspace{1mm}d\theta\,d\omega\\
 & \geq\int_{\mathbb{T}\times\mathbb{R}}(KR\sin(\theta-\phi))^{2}f\,\hspace{1mm}d\theta\,d\omega-\int\omega(KR\sin(\theta-\phi))f\,\hspace{1mm}d\theta\,d\omega\\
 & \geq\frac{1}{2}\int_{\mathbb{T}\times\mathbb{R}}(KR\sin(\theta-\phi))^{2}f\,\hspace{1mm}d\theta\,d\omega-\frac{W^{2}}{2}\\
 & \geq\frac{1}{2}K^{2}R^{2}\cos^{2}\gamma f(\mathbb{T}\setminus(L_{\gamma}^{+}(t)\cup L_{\gamma}^{-}(t))-\frac{W^{2}}{2}.
\end{aligned}
\end{align*}
Then, we obtain
\begin{equation}\label{dRlat}f(\mathbb{T}\setminus(L_{\gamma}^{+}(t)\cup L_{\gamma}^{-}(t))\leq\frac{1}{KR^{2}\cos^{2}\gamma}\frac{d}{dt}R^{2}+\frac{W^{2}}{K^{2}R^{2}\cos^{2}\gamma}.
\end{equation}
Additionally, using a similar argument on \eqref{Rdef}, where we split the integral into the sectors $L_{\gamma}^{+}$,  $L_{\gamma}^{-}$ and $\mathbb{T}\setminus(L_{\gamma}^{+}\cup L_{\gamma}^{-})$,  allows getting the lower bound
\begin{align*}
\begin{aligned}R & \geq\sin\gamma\,f(L_{\gamma}^{+})-\sin\gamma\,f(\mathbb{T}\setminus(L_{\gamma}^{+}\cup L_{\gamma}^{-}))-f(L_{\gamma}^{-})\\
 & =\sin\gamma\left(1-f(L_{\gamma}^{-})-f(\mathbb{T}\setminus(L_{\gamma}^{+}\cup L_{\gamma}^{-}))\right)-\sin\gamma\,f(\mathbb{T}\setminus(L_{\gamma}^{+}\cup L_{\gamma}^{-}))-f(L_{\gamma}^{-})\\
 & =\sin\gamma-2\sin\gamma\,f(\mathbb{T}\setminus(L_{\gamma}^{+}\cup L_{\gamma}^{-}))-(1+\sin\gamma)f(L_{\gamma}^{-})\\
 & \geq\sin\gamma-2\sin\gamma\bigg(\frac{1}{KR^{2}\cos^{2}\gamma}\frac{d}{dt}R^{2}+\frac{W^{2}}{K^{2}R^{2}\cos^{2}\gamma}\bigg)-(1+\sin\gamma)f(L_{\gamma}^{-}).
\end{aligned}
\end{align*}
Here, we have used the estimate \eqref{dRlat} in the last inequality. Then, \eqref{decRa} follows.\\
\noindent $\bullet$ \textit{Step 3}: Upper bound on $f(L_{\gamma}^{-})$.\\
Let us first achieve a lower bound of $f(L_{\gamma}^{+})$. To such end, we use a similar procedure and reverse the inequalities that we have considered in the preceding step. Specifically, notice that a similar split in \eqref{Rdef} allows obtaining
\begin{align*}
\begin{aligned}R & \leq f(L_{\gamma}^{+})+\sin\gamma f(\mathbb{T}\setminus(L_{\gamma}^{+}\cup L_{\gamma}^{-}))-\sin\gamma f(L_{\gamma}^{-})\\
 & =f(L_{\gamma}^{+})+\sin\gamma f(\mathbb{T}\setminus(L_{\gamma}^{+}\cup L_{\gamma}^{-}))-\sin\gamma(1-f(L_{\gamma}^{+})-f(\mathbb{T}\setminus(L_{\gamma}^{+}\cup L_{\gamma}^{-})))\\
 & =(1+\sin\gamma)f(L_{\gamma}^{+})+2\sin\gamma f(\mathbb{T}\setminus(L_{\gamma}^{+}\cup L_{\gamma}^{-}))-\sin\gamma.
\end{aligned}
\end{align*}
In particular, we obtain the lower bound
\[f(L_{\gamma}^{+})\geq\frac{R}{1+\sin\gamma}-\frac{2\sin\gamma}{1+\sin\gamma}f(\mathbb{T}\setminus(L_{\gamma}^{+}\cup L_{\gamma}^{-}))+\frac{\sin\gamma}{1+\sin\gamma}.\ \ \ 
\]
Hence, we obtain the upper bound
\begin{align}
\begin{aligned}\label{dRoposite}f(L_{\gamma}^{-}) & =1-f(L_{\gamma}^{+})-f(\mathbb{T}\setminus(L_{\gamma}^{+}\cup L_{\gamma}^{-}))\\
 & \leq1-\frac{\sin\gamma}{1+\sin\gamma}-\frac{R}{1+\sin\gamma}-\frac{1-\sin\gamma}{1+\sin\gamma}f(\mathbb{T}\setminus(L_{\gamma}^{+}\cup L_{\gamma}^{-}))\\
 & \leq\frac{1}{1+\sin\gamma}-\frac{R}{1+\sin\gamma}.
\end{aligned}
\end{align}
Notice that since $\dot{R}(t_{0})\leq0$ we can select $C$ appropriately in \ref{deccon} to guarantee that {\color{black}
\begin{equation}\label{field1}\frac{W}{K}+\sqrt{\frac{2\dot{R}(t_{0})}{KR(t_{0})}+\frac{1}{R(t_{0})^{2}}\frac{W^{2}}{K^{2}}}-R(t_{0})\cos\alpha\leq\frac{W}{K}+\sqrt{\frac{1}{ R_{0}^{2}}\frac{W^{2}}{K^{2}}}-\lambda R_{0}\cos\alpha<0.
\end{equation}}
Then, estimate $\eqref{Inv}$ implies
\[\left.\frac{d}{dt}\right\vert _{t=t_{0}}f(\chi_{\alpha}^{-})(t)\leq0.\ \ 
\]
By continuity, and, inequalities \eqref{Inv} and \eqref{field1}, $f(\chi_{\alpha}^{-})(t)$ remains non increasing along $[t_{0},t_{0}+\delta]$ for small enough $\delta>0$. Hence, we obtain that
\begin{align}
\begin{aligned}\label{dRoposite-2}f(L_{\gamma}^{-})(t) & \leq f(\chi_{\alpha}^{-})(t)\\
 & \leq f(\chi_{\alpha}^{-})(t_{0})\\
 & \leq f(L_{\gamma}^{-})(t_{0})+f(\mathbb{T}\setminus(L_{\gamma}^{+}\cup L_{\gamma}^{-}))(t_{0})\\
 & \leq\frac{1}{1+\sin\gamma}-\frac{R(t_{0})}{1+\sin\gamma}+\frac{W^{2}}{K^{2}\cos^{2}\gamma R^{2}(t_{0})},
\end{aligned}
\end{align}
for all $t$ in $[t_{0},t_{0}+\delta]$. Here, we have used the estimates \eqref{dRlat} and \eqref{dRoposite} along with the hypothesis $\dot{R}(t_{0})\leq0$.\\
\noindent $\bullet$ \textit{Step 4:} Derivation of \eqref{decRb} and lower bound of $R$ in $[t_{0},t_{0}+\delta]$.\\
Putting the last estimate \eqref{dRoposite-2} and \eqref{decRa} together, we obtain the differential inequality
\begin{align}
\begin{aligned}\label{decR-auxiliary}\frac{dR^{2}}{dt} & \geq\frac{KR^{2}\cos^{2}\gamma}{2\sin\gamma}\left[R(t_{0})-R-(1-\sin\gamma)-\frac{2\sin\gamma W^{2}}{K^{2}\cos^{2}\gamma R^{2}}-\frac{1+\sin\gamma W^{2}}{K^{2}\cos^{2}\gamma R^{2}(t_{0})}\right]\\
 & >\frac{K\cos^{2}\gamma}{2\sin\gamma}\left[-R^{3}+b(t_{0})R^{2}-c(t_{0})\right],
\end{aligned}
\end{align}
for all $t$ in $[t_{0},t_{0}+\delta]$. Here, the coefficients read
\begin{align*}
\begin{aligned}b(t_{0}) & :=R(t_{0})-\cos^{2}\gamma-\frac{2W^{2}}{K^{2}\cos^{2}\gamma R^{2}(t_{0})},\\
c(t_{0}) & :=\frac{2W^{2}}{K^{2}\cos^{2}\gamma}.
\end{aligned}
\end{align*}
Notice that in the last inequality in \eqref{decR-auxiliary} we have used
\[1-\sin\gamma<\cos^{2}\gamma,\hspace{1em}\sin\gamma<1,\hspace{1em}\mbox{ and }\hspace{1em}1+\sin\gamma<2.\]
By making $C$ smaller if necessary in \eqref{deccon} we can guarantee that
\begin{align*}\begin{aligned}b(t_{0}) & =R(t_{0})-\cos^{2}\gamma-\frac{2W^{2}}{K^{2}\cos^{2}\gamma R^{2}(t_{0})}\\
 & \geq R_{0}-\frac{(1-\lambda)}{5}R_{0}-10\frac{W^{2}}{K^{2}}\frac{1}{R_{0}^{3}(1-\lambda)}\\
 & \geq R(t_{0})-\frac{2(1-\lambda)}{5}R_{0}\\
 & =\lambda R_{0}+(1-\lambda)R_{0}-\frac{2(1-\lambda)}{5}R_{0}\\
 & =\lambda R_{0}+\frac{3}{5}(1-\lambda)R_{0}.
\end{aligned}
\end{align*}
Arguing in a similar way and making $C$ smaller if necessary in \eqref{deccon}, we can guarantee that 
\begin{align*}\begin{aligned}c(t_{0}) & :=\frac{2W^{2}}{K^{2}\cos^{2}\gamma}\\
 & =\bigg(\frac{W}{K}\bigg)^{2}\frac{10}{(1-\lambda)R_{0}}\\
 & \leq\frac{3}{5}(1-\lambda)\lambda^{2}R_{0}^{3}.
\end{aligned}
\end{align*}
Consequently, we have that 
\[\frac{d}{dt}R^{2}>\frac{K\cos^{2}\gamma}{2\sin\gamma}\bigg[-R^{3}+\big[\lambda R_{0}+\frac{3}{5}(1-\lambda)R_{0}\big]R^{2}-\frac{3}{5}(1-\lambda)\lambda^{2}R_{0}^{3}\bigg].\]
in $[t_{0},t_{0}+\delta].$ Since $\lambda R_{0},$ is the biggest root of the polynomial
\[p(r)=-r^{3}+[\lambda R_{0}+\frac{3}{5}(1-\lambda)R_{0}]r^{2}-\frac{3}{5}(1-\lambda)\lambda^{2}R_{0}^{3},\ \ 
\]
we obtain desire lower bound $R\geq\lambda R_{0}$ in $ [t_{0},t_{0}+\delta]$ by an elementary continuity method argument (we can see that $\lambda R_{0}$ is the biggest root of $p$ from the inequality $p(0)<0$ and the fact that $\lambda$ being contained in $(2/3,1)$ implies that $p^{\prime}(\lambda R_{0})<0)).$\\

\noindent $\bullet$ \textit{Step 5}: Propagation  of \eqref{decRb} and the lower bound of $R$ in $[t_{0},t_{0}+d]$.\\
The main idea is supported by a continuity method. We proceed by contradiction. Specifically, define the time
\[t_{*}:=\inf\left\{ t\in(t_{0}+\delta,t_{0}+d]:\,\frac{d}{dt}R^{2}<\frac{K\cos^{2}\gamma}{2\sin\gamma}p(R)\right\},
\]
and assume that $t^{*}<t_{0}+d$. Notice that, by definition, it implies
\[\frac{d}{dt}R^{2}\geq\frac{K\cos^{2}\gamma}{2\sin\gamma}p(R),\ \mbox{ for all }\ t\in[t_{0},t_{*}].
\]
In particular, by the same ideas in Step 4, we have that
\[R(t)\geq\lambda R_{0},\ \mbox{ for all }\ t\in[t_{0},t_{*}].
\]
By  \eqref{Inv} and the fact that
\[
\dot{R}\leq\frac{K\cos^{2}\alpha\lambda^{3}R_{0}^{3}}{4}\hspace{1em}\text{in}\hspace{1em}[t_{0},t_{0}+d],
\]
 making $C$ smaller in \eqref{deccon} if necessary, we can guarantee that, 
\begin{align}\begin{aligned}\label{local_field}\frac{W}{K}+\sqrt{\frac{2\dot{R}(t)}{KR(t)}+\frac{1}{R(t)^{2}}\frac{W^{2}}{K^{2}}} & -R(t)\cos\alpha\leq\frac{W}{K}\\
 & +\sqrt{\frac{\lambda^{2}R_{0}^{2}\cos^{2}\alpha}{2}+\frac{1}{\lambda^{2}R_{0}^{2}}\frac{W^{2}}{K^{2}}}-\lambda R_{0}\cos\alpha<0,\\
\end{aligned}
\end{align}
for all $t$ in $[t_{0},t_{*}]$. In particular the, by \eqref{Inv}   and continuity we have that $f(\chi_{\alpha}^{-})$ is non increasing in $[t_{0},t_{*}+\delta_{*}]$ and some small enough $\delta_{*}>0$. Hence, we can repeat the train of thoughts in \textit{Step 4} to extend the upper bound of $f(\chi_{\gamma}^{-})(t)$ in \eqref{dRoposite-2} to the larger interval $[t_{0},t_{*}+\delta_{*}]$. Again, the same ideas as in\\
\textit{Step 4} imply that
\[\frac{d}{dt}R^{2}>\frac{K\cos^{2}\gamma}{2\sin\gamma}p(R),\ \mbox{ for all }\ t\in[t_{0},t_{*}+\delta_{*}],
\]
and it contradicts the definition of $t_{*}$.
\end{proof}
}
\noindent We close this section by showing that we can obtain a universal lower bound on the order parameter. That is the objective of the following corollary.
{\color{black}
\begin{corollary}\label{gain_vs_loss}
Suppose that $1-\lambda$ is contained in $(0,R_{0}/120)$. Then,  there exists a universal constant $C$ such that if
\begin{equation}\label{Gain_vs_loss_cali}\frac{W}{K}<C\lambda^{2}(1-\lambda)R_{0}^{2},
\end{equation} 
then, we have that 
\[R\geq\lambda R_{0},\ \ 
\]
for every $t$ in $[0,\infty).$ \begin{proof}
We begin by choosing $C$ small enough so that it can be taken simultaneously as the corresponding universal constants in Lemma \ref{Rgain} and \ref{Rdecrease}.\\
\noindent We claim that either one of the following two conditions holds:
\begin{itemize}
	\item[$(i)$ ] We have that $\dot{R}<K/4\lambda^{3}R_{0}^{3}\cos^{2}\alpha$ in $[0,\infty).$ \\
	
\item[$(ii) $ ] There exist a time $t^{*}$ and an increasing and strictly positive universal function $h,$ satisfying that $R\geq\lambda R_{0}$ in $[0,t^{*}]$ and $R(t^{*})^{2}\geq R_{0}^{2}+h(R_{0}).$ 
\end{itemize}
We divide the proof of the corollary into two steps. The second of which is the proof of the claim.\\

\noindent $\bullet$ \textit{Step 1}: We show how the claim implies the Corollary.\\

\noindent To see this, we use the following iterative argument based on the fact that $R$ is bounded and the system is autonomous. If condition $(ii)$ of the claim holds, we use the fact that the system is autonomous in time to translate the initial condition of the system to be the configuration at $t^{*}.$ Since by assumption the value of the order parameter at $t^{*}$ is bigger than $R_{0}$ we are free to apply the claim again with the same value of $C$ to the corresponding shifted initial condition. We can do this iteratively as many time as needed provided that condition $(ii)$ still holds after the time translation.\\

 \noindent To conclude this step, note that since $R$ is bounded and the function $h$ is positive, increasing, and universal condition $(ii)$ can hold consecutively after each time translation only a finite number of times. Hence,  after finitely many time shifts, condition $(i)$ will hold. Finally, once condition $(i)$ holds, the global lower bound follows by applying Lemma \ref{Rdecrease}. \\ 
 
\noindent $\bullet$ \textit{Step 2}: We show the claim.\\

\noindent For this purpose suppose that $(i)$ does not hold, that is the set 
\[\bigg\{ t\geq0\hspace{1mm}:\hspace{1mm}\dot{R}(t)\geq K\frac{\lambda^{3}R_{0}^{3}\cos^{2}\alpha}{4}\bigg\},\ 
\]
is not empty. To show that $(ii)$ holds in this case, let us consider the smallest time $t_{1}$ such that $\dot{R}(t_{1})\geq K/4\lambda^{3}R_{0}^{3}\cos^{2}\alpha.$ Now, let $t_{2}$ denote the biggest time $t_{2},$ bigger or equal to $t_{1},$ such that $\ensuremath{\dot{R}\geq K/4\lambda^{3}R_{0}^{3}\cos^{2}\alpha}$ in $[t_{1},t_{2}].$ Notice that the existence of $t_{2}$ follows by the boundedness of $R.$ \\
\noindent Now, observe that, by definition of $t_{1}$ Lemma \ref{Rdecrease} implies that $R\geq\lambda R_{0}$ in $[0,t_{1}].$ Moreover, by construction 
\[\ensuremath{\dot{R}\geq\frac{K}{4}\lambda^{3}R_{0}^{3}\cos^{2}\alpha\hspace{1em}\text{in}\hspace{1em}[t_{1},t_{2}]}.\ 
\]
Consequently, $R\geq R(t_{1})\geq\lambda R_{0}$ in $[t_{1},t_{2}].$ Now, we consider two cases:\\

\noindent \textit{Case 1:} $R(t_{2})\leq\sqrt{2}R_{0}.$\\

\noindent In this case, observe that Lemma \ref{Rgain} implies that we can find a constant $d$ such that
\[R^{2}(t_{2}+d)-R^{2}(t_{2})=\frac{\lambda^{4}}{40}R_{0}^{3}.\ 
\]
Consequently, by our assumptions on $\lambda,$ we have that 
\begin{align*}\begin{aligned}R(t_{2}+d) & ^{2}=R(t_{2})^{2}+\frac{\lambda^{4}}{40}R_{0}^{3}\\
 & \geq\lambda^{2}R_{0}^{2}+\frac{\lambda^{4}}{40}R_{0}^{3}\\
 & \geq R_{0}^{2}+\frac{\lambda^{4}}{40}R_{0}^{3}-(1-\lambda^{2})R_{0}^{2}\\
 & >R_{0}^{2}+\bigg(\frac{5}{240}R_{0}-2(1-\lambda)\bigg)R_{0}^{2}\\
 & >R_{0}^{2}+\frac{1}{240}R_{0}^{3}.
\end{aligned}
\end{align*}
Here, on the third line, we have used the fact that $\lambda^{4}>9/10.$ \\

\noindent Thus, the desired result follows by setting $t^{*}=t_{2}+d$ and
\[h(r):=\frac{r^{3}}{240}. 
\]
\noindent \textit{Case 2:} $R(t_{2})>\sqrt{2}R_{0}.$\\

\noindent In this case, we obtain that $R(t_{2})^{2}-R_{0}^{2}>R_{0}^{2}>R_{0}^{3}/240.$ Hence, the desired result holds for $t^{*}=t_{2}.$
\end{proof}
\end{corollary}}

\section{ Instability of antipodal equilibria and Sliding norms }

\noindent We now start implementing the program outlined in sections 2.3 and 2.4. To do this, we first derive inequalities \eqref{GL2}  and \eqref{Inst}.
{\color{black}
\begin{proposition}\textbf{(Instability of antipodal equilibria)}\label{l2imeq}
Assume that $f_{0}$ is contained in  $C^{1}(\mathbb{T}\times\mathbb{R})$ and $g$ is compactly supported in $[-W,W]$. Consider the unique global-in-time classical solution $f=f(t,\theta,\omega)$ to \eqref{E-KS} and let $\alpha$ be as specified in Section 2. Then, we have that
\[\frac{d}{dt}f^{2}(\chi_{\alpha}^{-}(t))\leq-KR\sin\alpha f^{2}(\chi_{\alpha}^{-}(t)))+4Kf^{2}\big(\mathbb{T}\big)\bigg[\frac{W}{K}+\sqrt{\frac{2\dot{R}}{KR}+\frac{1}{R^{2}}\frac{W^{2}}{K^{2}}}-R\cos\alpha\bigg]^{+},
\]
and
\begin{equation}\label{total_l2}
\frac{d}{dt}f^{2}\big(\mathbb{T}\big)\leq KR f^{2}\big(\mathbb{T}\big).
\end{equation}
Moreover, with the hypothesis  \eqref{deccon} and notation from Proposition \ref{Rdecrease}, if  $[t_{1},t_{2}]$ is  a time interval such that 
\begin{equation}\label{local_dotr}
\dot{R}\leq K\frac{\lambda^{3}R_{0}^{3}\cos^{2}\alpha}{4}\hspace{1em}\hspace{1em}\text{in}\hspace{1em}[t_{1},t_{2}],
\end{equation}
then, we have that 
\begin{equation}\label{antipodal_l2}
\frac{d}{dt}f^{2}(L_{\alpha}^{-}(t))\leq-K\lambda R_{0}\sin\alpha f^{2}(L_{\alpha}^{-}(t))) \hspace{1em}\hspace{1em}\text{in}\hspace{1em}[t_{1},t_{2}].
\end{equation}
\begin{proof}
We begin with the first inequality in the Proposition. Arguing as in Step 1 of the proof of Lemma \ref{Rdecrease}  we obtain that
\begin{align}\begin{aligned}\label{test}\frac{d}{dt}\int\chi_{\alpha}^{-}(\theta-\phi+\pi)f^{2}\hspace{1mm}d\theta d\omega & =\int\dot{\phi}\chi_{\alpha}^{-\prime}(\theta-\phi+\pi)f^{2}\hspace{1mm}d\theta d\omega\\
 & \hspace{1em}+2\int\chi_{\alpha}^{-}(\theta-\phi+\pi)f\partial_{t}f\hspace{1mm}d\theta d\omega\\
 & =\int[\dot{\phi}+2\omega-2KR\sin(\theta-\phi)]\chi_{\alpha}^{-\prime}(\theta-\phi+\pi)f^{2}\hspace{1mm}d\theta d\omega\\
 & \hspace{1em}+2\int\chi_{\alpha}^{-}(\theta-\phi+\pi)[\omega-KR\sin(\theta-\phi)]f\partial_{\theta}f\hspace{1mm}d\theta d\omega\\
 & \leq\int[\dot{\phi}+\omega-KR\sin(\theta-\phi)]\chi_{\alpha}^{-\prime}(\theta-\phi+\pi)f^{2}\hspace{1mm}d\theta d\omega\\
 & \hspace{1em}-\int\chi_{\alpha}^{-}(\theta-\phi+\pi)KR\sin\alpha f^{2}\hspace{1mm}d\theta d\omega.
\end{aligned}
\end{align}
The first inequality in the proposition follows from Lemma \ref{dRdphi} and the same arguments as in Step 1 from Proposition \ref{Rdecrease}. Inequality \eqref{total_l2} follows from similar arguments to those of \eqref{test} by replacing $\chi_{\alpha}$ with the constant function that is equal to one in $\mathbb{T}.$
Finally, to derive inequality \eqref{antipodal_l2}, recalling the notation introduced in Section 2.3, replacing $\chi_{\alpha}^{-}$ with $\chi_{\alpha,\varepsilon}^{-}$ in \eqref{test} and arguing as in  Step 1 from Proposition \ref{Rdecrease} we get that
\[
\frac{d}{dt}f^{2}(\chi_{\alpha,\varepsilon}^{-}(t))\leq-KR\sin\alpha f^{2}(\chi_{\alpha,\varepsilon}^{-}(t)))+KC_{\varepsilon,\alpha}f^{2}\big(\mathbb{T}\big)\bigg[\frac{W}{K}+\sqrt{\frac{2\dot{R}}{KR}+\frac{1}{R^{2}}\frac{W^{2}}{K^{2}}}-R\cos\alpha\bigg]^{+}. 
\]
Now, we observe that as in \eqref{local_field}, we can see that the second term of the above inequality  vanishes on the interval $[t_{1},t_{2}].$ Consequently, such a term is independent of $\varepsilon$ and thus \eqref{antipodal_l2}  follows by letting $\varepsilon\rightarrow0.$

\end{proof}
\end{proposition}}

\noindent A form of the above Lemma was one of the main tools used to derive the main result in \cite{HKMP-19}.  However, to obtain our convergence rates, we work with a sliding version of the $L^{2}$ norm.  Such sliding norms allow us to propagate the above estimate analog the flow of the continuity equation. This technique turns out to be one of the crucial components in our arguments in Section 5.
{\color{black}
\begin{lemma}\textbf{(Sliding norms)}\label{slidel2}
Assume that $f_{0}$ is contained in $C^{1}(\mathbb{T}\times\mathbb{R})$ and $g$ is compactly supported. Consider the unique global-in-time classical solution $f=f(t,\theta,\omega)$ to \eqref{E-KS}. Then, for any measurable  set $A$ we have that 
\[\frac{d}{dt}f^{2}(A_{t_{0},t})\leq KR\bigg(\sup_{(\theta,\omega)\in A_{t_{0},t}}\cos(\theta-\phi(t))\bigg)f^{2}(A_{t_{0},t}).\]
\end{lemma}
\begin{proof}
By the change of variable theorem, we have that 
\begin{align*}\begin{aligned}\frac{d}{dt}\frac{1}{2}\int_{A_{t_{0},t}}f^{2}\hspace{1mm}d\theta d\omega & =\frac{d}{dt}\bigg|_{t=t_{0}}\frac{1}{2}\int_{A}f_{t}^{2}(\Theta_{t_{0},t}(\theta,\omega),\omega)\partial_{\theta}\Theta_{t_{0},t}\hspace{1mm}d\theta d\omega\\
 & =\int_{A}f_{t}(\Theta_{t_{0},t}(\theta,\omega),\omega)[\partial_{t}f(\Theta_{t_{0},t}(\theta,\omega),\omega)\\
 & +\dot{\Theta}_{t_{0},t}(\theta,\omega)\partial_{\theta}f(\Theta_{t_{0},t_{0}}(\theta,\omega),\omega)]\partial_{\theta}\Theta_{t_{0},t}\hspace{1mm}d\theta d\omega\\
 & \hspace{1em}-\frac{1}{2}KR\int_{A}\cos(\Theta_{t}(\theta,\omega)-\phi)\partial_{\theta}\Theta_{t_{0},t}f^{2}\hspace{1mm}d\theta d\omega\\
 & =\int_{A}f_{t}(\Theta_{t_{0},t}(\theta,\omega),\omega)\big[-\partial_{\theta}(\omega f-KR\sin(\Theta_{t_{0},t}(\theta,\omega)-\phi)f)\\
 & \hspace{1em}+(\omega-KR\sin(\Theta_{t_{0},t}(\theta,\omega)-\phi)\partial_{\theta}f(\Theta_{t_{0},t}(\theta,\omega),\omega)]\partial_{\theta}\Theta_{t_{0},t}\hspace{1mm}d\theta d\omega\\
 & \hspace{1em}+\frac{1}{2}KR\int_{A}\cos(\Theta_{t_{0},t}(\theta,\omega)-\phi)f^{2}\partial_{\theta}\Theta_{t_{0},t}\hspace{1mm}d\theta d\omega\\
 & =\frac{1}{2}KR\int_{A}\cos(\Theta_{t_{0},t}(\theta,\omega)-\phi)f_{t}^{2}(\theta,\omega)\partial_{\theta}\Theta_{t_{0},t}\hspace{1mm}d\theta d\omega.
\end{aligned}
\end{align*}
where for $t$ and each $\omega,$ $\ensuremath{\partial_{\theta}\Theta_{t_{0},t}}(\cdot,\omega)$ denotes the Jacobian of the map $\theta\rightarrow \Theta_{t_{0},t}(\theta,\omega).$ Hence, the desired result follows.
\end{proof}}
\noindent To make full use of the above control, we need to understand the dynamics of the Lagrangian flow associated with the continuity equation. That is the objective of the next part.
{\color{black}
\subsection{Attractors}
In this section, we will show the emergence of  time-dependent sets that will act as attractors along the characteristic flow. Such sets, in combination with our analysis on sliding norms in the previous section, will allow us to propagate information between the different parts of the system.\\

\noindent Before showing the emergence of attractor sets, we state the following Lemma, which we will repeatedly use throughout the rest of the paper. Additionally, in this part, we will use the notation introduced in Section 2.3.

\begin{lemma}[\textbf{Emergence of invariant sets}]\label{G_attractor} 
Assume that $f_{0}$ is contained in $C^{1}(\mathbb{T}\times\mathbb{R})$,  $g$ has compact support in $[-W,W],$ and consider the unique global-in-time classical solution to \eqref{E-KS} $f=f(t,\theta,\omega)$. 
Let $t_{0}\geq0$ be an initial time in  $[0,\infty)$ and $L\subset\mathbb{T}$ be an interval. Now, assume that, initially we have that 
\[\rho_{t_{0}}(L)\geq m,\hspace{1em}\text{and}\hspace{1em}p=\inf_{\theta,\theta^{\prime}\in L}\cos(\theta-\theta^{\prime}),
\]
for some positive numbers $m$ and $p$ in $(0,1).$ Additionally, suppose that 
\begin{equation}\label{Gcali}mp-(1-m)\geq\sigma\hspace{1em}\text{and}\hspace{1em}\frac{W^{2}}{K^{2}}\leq\frac{(1-p)\sigma^{2}}{4},
\end{equation}
for some $\sigma>0.$ \\
Then, if we set 
\[\underline{P}(t)=\inf_{\theta,\theta^{\prime}\in L_{t_{0},t}}\cos(\theta-\theta^{\prime}),\ 
\] the following bounds hold true
\begin{equation}\label{Gmtm1}\rho(L_{t_{0},t})\geq m,
\end{equation}
\begin{equation}\label{GE-attractor-2}\inf_{\theta\in L_{t_{0},t}}R\cos(\theta-\phi)\geq m\underline{P}-(1-m),
\end{equation}
and
\begin{equation}\label{GE-attractor-4}1-\underline{P}(t)\leq\max\bigg((1-p)e^{-\frac{K\sigma}{4}(t-t_{0})},\frac{4}{\sigma^{2}}\frac{W^{2}}{K^{2}}\bigg),
\end{equation}
for every $t$ in $[t_{0},\infty)$
\end{lemma}
\begin{proof}
The proof of \eqref{GE-attractor-4} is based on a continuity method argument that holds under the condition \eqref{Gcali}. Such an argument is based on inequalities \eqref{Gmtm1}, \eqref{GE-attractor-2}, and 
\begin{equation}\label{GE-attractor-3}\frac{dP}{dt}\geq2K\sqrt{1-P^{2}}\left[R\left(\inf_{\theta\in L_{t_{0},t}}\cos(\theta-\phi)\right)\sqrt{\frac{1-P}{2}}-\frac{W}{K}\right],\hspace{1em}\forall t\geq t_{0},
\end{equation}
which hold when 
\begin{equation}\label{Ge_pdef}
P=\cos(\Theta_{s,t}(\theta,\omega)-\Theta_{s,t}(\theta^{\prime},\omega^{\prime})),
\end{equation}
for any $s\geq t_{0}$ such that $t\geq s,$ and any couple of points $(\theta,\omega)$ and $(\theta^{\prime},\omega^{\prime})$ contained in $L_{t_{0},s}\times[-W,W].$ \\
We will first proof inequality \eqref{GE-attractor-4} first and then prove the remaining inequalities afterward. Indeed, let us define $t^{\prime}$ as the supremum of the set of times $t^{*}\geq t_{0}$ such that inequality \eqref{GE-attractor-4} holds, for every $t$ in $[t_{0},t^{*}].$ We begin by noting that, by continuity
\[
1-\underline{P}(t^{\prime})=\max\bigg((1-p)e^{-\frac{K\sigma}{4}(t^{\prime}-t_{0})},\frac{4}{\sigma^{2}}\frac{W^{2}}{K^{2}}\bigg).
\]
Now, we must prove that there exists $\delta>0$ such that \eqref{GE-attractor-4} holds in $[t_{0},t^{\prime}+\delta].$ More precisely, our goal is to show that there exists a uniform time $\delta>0$ such that for any pair of characteristics starting at $L_{t_{0},t^{\prime}}\times[-W,W]$ we have that the corresponding $P$ (given by \ref{Ge_pdef}) satisfies that $1-P$ is bounded by the right-hand side of \eqref{GE-attractor-4} in $[t^{\prime},t^{\prime}+\delta]$. \\
To do this, let $s=t^{\prime}$ in the definition of $P.$ Now observe that by \eqref{GE-attractor-2} and \eqref{GE-attractor-3}, when $t=t^{\prime}$, we have that
\begin{align}\begin{aligned}\label{GdP}\frac{dP}{dt}\bigg|_{t=t^{\prime}} & \geq2K\sqrt{1-P^{2}}\left[R\left(\inf_{\theta\in L_{t_{0},t^{\prime}}}\cos(\theta-\phi)\right)\sqrt{\frac{1-P}{2}}-\frac{W}{K}\right]\\
 & \geq2K\sqrt{1-P^{2}}\left[[m\underline{P}-(1-m)]\sqrt{\frac{1-P}{2}}-\frac{W}{K}\right]\\
 & \geq2K\sqrt{1+P}\bigg[\frac{\sqrt{2}}{2}\sigma(1-P)-\frac{W}{K}\sqrt{1-P}\bigg].
\end{aligned}
\end{align}
Here, all the time-dependent expressions are evaluated at $t=t^{\prime}.$ Additionally, in the last inequality, we have used our assumption that \eqref{GE-attractor-4} holds on the interval $[t_{0},t^{\prime}],$ which together with \eqref{Gcali} implies the uniform lower bound $p\leq\underline{P}$. Now, let $(\theta,\omega)$ and $(\theta^{\prime},\omega^{\prime})$ be any couple of points contained in $L_{t_{0},t^{\prime}}\times[-W,W]$ such that the corresponding $P$ satisfies that 
\begin{equation}\label{Ge-pairs}
1-P(t^{\prime})=1-\underline{P}(t^{\prime})=\max\bigg((1-p)e^{-\frac{K\sigma}{4}(t^{\prime}-t_{0})},\frac{4}{\sigma^{2}}\frac{W^{2}}{K^{2}}\bigg).
\end{equation}
Note that since $L$ is compact, then $L_{t_{0},t^{\prime}}\times[-W,W]$ is compact as well. Thus,  the set of such pairs $\ensuremath{(\theta,\omega)}$ and $\ensuremath{(\theta^{\prime},\omega^{\prime})}$ in $L_{t_{0},t^{\prime}}$ whose corresponding $P$ (obtained via \eqref{Ge_pdef})  satisfies \eqref{Ge-pairs} is a compact set as well. We shall denote such a set by $\mathcal{P}\subset\ensuremath{L_{t_{0},t^{\prime}}\times[-W,W]\times L_{t_{0},t^{\prime}}\times[-W,W].}$
To continue our proof observe that by using the assumption \eqref{Ge-pairs}, we get that 
\[\frac{\sigma\sqrt{1-P(t^{\prime})}}{2}\geq\frac{W}{K},
\]
for any couple of characteristics in $\mathcal{P}$ and, consequently, by \eqref{GdP} we obtain that
\begin{align*}\begin{aligned}\frac{d}{dt}\bigg|_{t=t^{\prime}}(1 & -P)\\
 & \leq-2K\sqrt{1+P}\bigg[\frac{\sqrt{2}}{2}\sigma(1-P)-\frac{\sigma(1-P(t^{\prime}))}{2}\bigg]\\
 & \leq-\frac{2}{5}K\sigma((1-\underline{P}(t^{\prime})).\\
 & <\begin{cases}
-\frac{2}{5}K\sigma(1-p)e^{-\frac{K\sigma}{4}(t^{\prime}-t_{0})} & \text{if}\hspace{1em}\frac{4}{\sigma^{2}}\frac{W^{2}}{K^{2}}<1-\underline{P}(t^{\prime}),\\
0 & \text{if}\hspace{1em}\frac{4}{\sigma^{2}}\frac{W^{2}}{K^{2}}=1-\underline{P}(t^{\prime}),
\end{cases}
\end{aligned}
\end{align*}
Since the right-hand side of the above inequality is uniform in the set of pairs in $\mathcal{P}$ and the set  $\mathcal{P}$ is compact, we can find $\varepsilon>0$ such that if $\mathcal{P}_{\varepsilon}$ is an $\varepsilon$-neighborhood of $\mathcal{P},$ then we have that
\begin{align}\begin{aligned}\label{GEstrict}\frac{d}{dt}\bigg|_{t=t^{\prime}}(1- & \cos(\Theta_{t^{\prime},t}(\theta,\omega)-\Theta_{t^{\prime},t}(\theta^{\prime},\omega^{\prime})))\\
 & <-\frac{1}{3}K\sigma((1-\underline{P}(t^{\prime}))\\
 & <\begin{cases}
-\frac{K}{4}\sigma(1-p)e^{-\frac{K\sigma}{4}(t^{\prime}-t_{0})}, & \text{if}\hspace{1em}\frac{4}{\sigma^{2}}\frac{W^{2}}{K^{2}}<1-\underline{P}(t^{\prime}),\\
0 & \text{if}\hspace{1em}\frac{4}{\sigma^{2}}\frac{W^{2}}{K^{2}}=1-\underline{P}(t^{\prime}),
\end{cases}
\end{aligned}
\end{align}
for any $(\theta,\omega)$,$(\theta^{\prime},\omega^{\prime})$ in $\mathcal{P}_{\varepsilon}$. This implies the existence of $\delta$ and thus concludes the continuity method argument. Indeed, for characteristics with initial data in $\mathcal{P}_{\varepsilon}$ the existence of the  time interval  $[t^{\prime},t^{\prime}+\delta),$ follows by the fact that the inequality in \eqref{GEstrict} is strict  and uniform in $\mathcal{P}_{\varepsilon}$. Similarly, for characteristics in $\ensuremath{(L_{t_{0},t^{\prime}}\times[-W,W]\times L_{t_{0},t^{\prime}}\times[-W,W])\backslash\mathcal{P}_{\varepsilon}},$ the existence of the uniform time $\delta$ follows by the fact that the characteristics have uniformly bounded speed and $\varepsilon$ provides a uniform separation distance.  
\\

\noindent (Indeed, by continuity and compactness, we can find a uniform time neighborhood of $t^{\prime},$ in which the infimum for $\underline{P}$ is attained in $\mathcal{P}_{\varepsilon/2},$ and we have already shown the existence of $\delta$ in such a case.)\\

\noindent Hence, to complete the proof of the lemma it suffices to derive inequalities \eqref{Gmtm1}, \eqref{GE-attractor-2}, and \eqref{GE-attractor-3}. We achieve this in the following steps:\\
\noindent $\bullet$ \textit{Step 1}: Proof of inequalities \eqref{Gmtm1} and \eqref{GE-attractor-2}.\\
Inequality \eqref{Gmtm1} follows  from the fact that the continuity equation preserves the mass of sets along the characteristic flow. To derive inequality \eqref{GE-attractor-2}, we observe that
\begin{align}\begin{aligned}\label{GAlower}\inf_{\theta\in L_{t_{0},t}}R\cos(\theta-\phi) & =\inf_{\theta\in L_{t_{0},t}}\langle e^{i\theta},\int e^{i\theta^{\prime}}fd\omega^{\prime}\hspace{1mm}d\theta^{\prime}\rangle\\
 & \geq\inf_{\theta\in L_{t_{0},t}}\int\cos(\theta-\theta^{^{\prime}})f^{\prime}\hspace{1mm}d\theta^{\prime}d\omega^{\prime}\\
 & \geq\inf_{\theta\in L_{t_{0},t}}\bigg[\int_{(L\times[-W,W])_{t_{0},t}}\cos(\theta-\theta^{^{\prime}})f^{\prime}\hspace{1mm}d\theta^{\prime}d\omega^{\prime}\\
 & +\int_{\mathbb{T}\times\mathbb{R}\backslash(L\times[-W,W])_{t_{0},t}}\cos(\theta-\theta^{^{\prime}})f^{\prime}\hspace{1mm}d\theta^{\prime}d\omega^{\prime}\bigg]\\
 & \geq m\underline{P}-(1-m).
\end{aligned}
\end{align}
This completes Step 1. \\
\noindent $\bullet$ \textit{Step 2}: Proof of inequality \eqref{GE-attractor-3}.
To obtain \eqref{GE-attractor-3} let us fix $t$ in $[t_{0},\infty),$ and let $(\theta,\omega)$ and   $(\theta',\omega')$ be contained in $L\times[-W,W].$
Additionally, let us set
\[\Theta(s):=\Theta_{t_{0},s}(\theta,\omega),\ \text{and}\ \Theta'(s):=\Theta_{t_{0},s}(\theta',\omega').\]
Then,
\begin{align}\begin{aligned}\label{costheta12}
\frac{d}{ds}\bigg|_{s=t} & \cos(\Theta-\Theta')=-\sin(\Theta-\Theta')(\dot{\Theta}-\dot{\Theta}')\\
 & =-\sin(\Theta-\Theta')\left((\omega-\omega')-KR(\sin(\Theta-\phi)-\sin(\Theta'-\phi))\right)\\
 & =-\sin(\Theta-\Theta')\left[(\omega-\omega')-2KR\cos\left(\frac{\Theta+\Theta'}{2}-\phi\right)\sin\left(\frac{\Theta-\Theta'}{2}\right)\right]\\
 & =-2\cos\left(\frac{\Theta-\Theta'}{2}\right)\bigg[(\omega-\omega')\sin\left(\frac{\Theta-\Theta'}{2}\right)\\
 & \hspace{11em}-2KR\cos\left(\frac{\Theta+\Theta'}{2}-\phi\right)\sin^{2}\left(\frac{\Theta-\Theta'}{2}\right)\bigg]\\
 & \geq4KR\cos\left(\frac{\Theta-\Theta'}{2}\right)\bigg[\cos\left(\frac{\Theta+\Theta'}{2}-\phi\right)\frac{1-\cos\left(\Theta-\Theta'\right)}{2}\\
 & \hspace{20em}-\frac{W}{KR}\sqrt{\frac{1-\cos(\Theta-\Theta')}{2}}\bigg],
\end{aligned}
\end{align}
where we have used several standard trigonometric formulas. Now, notice that
\[\frac{\Theta_{t_{0},t}(\theta,\omega)+\Theta_{t_{0},t}(\theta',\omega')}{2}\hspace{1em}\text{is contaiend in \ensuremath{\hspace{1em}}}L_{t_{0},t},\ 
\]
since it is a convex combination of two points in $L_{t_{0},t}$. Thus, when $s=t_{0}$ \eqref{GE-attractor-3} follows by standard trigonometric identities. In the case when $s$ is contained in $[t_{0},t]$ we can easily derive \eqref{GE-attractor-3} by the same argument and the semigroup property of the characteristic flow.
\end{proof}

\noindent As a first application of the above lemma, we quantify below the first time in which the system forms an attractor.
\begin{lemma}[\textbf{First invariant set}]\label{t-1}
Assume that $f_{0}$ is contained in $C^{1}(\mathbb{T}\times\mathbb{R})$ and $g$ has compact support in $[-W,W].$ Consider the unique global-in-time classical solution to \eqref{E-KS} $f=f(t,\theta,\omega)$ and let us set an angle $0<\gamma<\frac{\pi}{2}$ so that
\begin{equation}\label{tm1cosdef} \cos^{2}\gamma=\frac{1}{30}R_{0}.
\end{equation}
Then, we can find a universal constant $C$ such that if 
\begin{equation}\label{tm1cali}\frac{W}{K}\leq CR_{0}^{2},
\end{equation}
then there exists a positive time 
$T_{-1}$ satisfying that
\begin{equation}\label{tm1}T_{-1}\lesssim\frac{1}{KR_{0}^{3}},
\end{equation}
and the bounds 
\begin{equation}\label{mtm1}\rho(L_{\gamma}^{+}(T_{-1})_{t})\geq\frac{1+\frac{4}{5}R_{0}}{2},
\end{equation}
\begin{equation}\label{costm1}\inf_{\theta\in L_{\gamma}^{+}(T_{-1})_{t}}R\cos(\theta-\phi)\geq\frac{3}{5}R_{0},
\end{equation}
and
\begin{equation}\label{diamtm1}\inf_{\theta,\theta^{\prime}\in L_{\gamma}^{+}(T_{-1})_{t}}\cos(\theta-\theta^{\prime})\geq1-\frac{1}{15}R_{0},
\end{equation}
hold true for every $t$ in $[T_{-1},\infty)$.
\end{lemma}
\begin{proof}
Define the time
\begin{equation}\label{E-t-attractor}T_{-1}=\inf\left\{ t\geq0:\,\dot{R}\leq\frac{KR_{0}^{3}}{4\cdot30^{2}}\right\},
\end{equation}
and note that by construction, \eqref{tm1} follows by the fact that $R$ is bounded by 1 and the fundamental theorem of calculus.\\

\noindent The proof of the remaining parts of the Lemma will follow directly from an application of Lemma \ref{G_attractor} by setting $L=L_{\gamma}^{+}(T_{-1})$ and $t_{0}=T_{-1}$. To verify the corresponding hypotheses, first, we begin by controlling the mass in $L_{\gamma}^{+}(T_{-1}).$ Indeed, by the decomposition of the integral \eqref{Rdef} that defines $R$ into three parts $L_{\gamma}^{+}$, $L_{\gamma}^{-},$ and $\mathbb{T}\setminus(L_{\gamma}^{+}\cup L_{\gamma}^{-}),$ we obtain the inequality:
\begin{equation}\label{R-split-upperbound}R\leq(1+\sin\gamma)\rho(L_{\gamma}^{+})-\sin\gamma+2\sin\gamma\rho(\mathbb{T}\setminus(L_{\gamma}^{+}\cup L_{\gamma}^{-})).
\end{equation}
Consequently, using \eqref{dRlat} to control $\rho(\mathbb{T}\setminus(L_{\gamma}^{+}\cup L_{\gamma}^{-}))$, we deduce that
\begin{align*}\begin{aligned}\rho(L_{\gamma}^{+}) & \geq\frac{R+\sin\gamma}{1+\sin\gamma}-\frac{2\sin\gamma}{1+\sin\gamma}\rho(\mathbb{T}\setminus(L_{\gamma}^{+}\cup L_{\gamma}^{-}))\\
 & \geq\frac{1}{1+\sin\gamma}\bigg[R+\sin\gamma-2\bigg(\frac{1}{KR^{2}\cos^{2}\gamma}\frac{d}{dt}R^{2}+\frac{W^{2}}{K^{2}R^{2}\cos^{2}\gamma}\bigg)\bigg]\\
 & =\frac{1}{1+\sin\gamma}\bigg[R+1+(\sin\gamma-1)-2\bigg(\frac{2\dot{R}}{KR\cos^{2}\gamma}+\frac{W^{2}}{K^{2}R^{2}\cos^{2}\gamma}\bigg)\bigg].
\end{aligned}
\end{align*}
for any $t\geq0.$ Then, evaluating the above expression at $t=T_{-1},$ using the fact that by construction $R(T_{-1})\geq R_{0}$, and selecting $C<1/30)$ in \eqref{tm1cali}, we deduce that 
\begin{align}\begin{aligned}\label{t-1mass}\rho(L_{\gamma}^{+}(T_{-1})) & \geq\frac{1}{2}\bigg[R_{0}+1+(\sin\gamma-1)-2\bigg(\frac{2\dot{R}(T_{-1})}{KR\cos^{2}\gamma}+\frac{W^{2}}{K^{2}R^{2}\cos^{2}\gamma}\bigg)\bigg]\\
 & \geq\frac{1}{2}\bigg[R_{0}+1-\frac{R_{0}}{30}-2\bigg(30\frac{2\dot{R}(T_{-1})}{KR_{0}^{2}}+\frac{30}{R_{0}}\frac{W^{2}}{K^{2}R_{0}^{2}}\bigg)\bigg]\\
 & \geq\frac{1}{2}\bigg(1+\frac{4}{5}R_{0}\bigg),
\end{aligned}
\end{align}
where we have used the fact that $1-\sin\gamma\leq1-\sin\gamma^{2}=\cos^{2}\gamma=\frac{R_{0}}{30}.$\\
Second, we estimate the infimum of the cosine of the difference of angles in $L_{\gamma}^{+}(T_{-1}),$ that is 
\begin{align}\begin{aligned}\label{t-1cos}\inf_{\theta,\theta^{\prime}\in L_{\gamma}^{+}(T_{-1})}\cos(\theta-\theta^{\prime}) & =\cos(\pi-2\gamma)\\
 & =\cos\bigg(2\big(\frac{\pi}{2}-\gamma\big)\bigg)\\
 & =2\cos^{2}\big(\frac{\pi}{2}-\gamma\big)-1\\
 & =2\sin^{2}\gamma-1\\
 & =1-\frac{1}{15}R_{0}.
\end{aligned}
\end{align}
Finally, considering Lemma \ref{G_attractor} with $m=\rho(L_{\gamma}^{+}(T_{-1}))$ and $p=\cos(\pi-2\gamma)$,  and using the bounds in \eqref{t-1mass} and \eqref{t-1cos}, we obtain
\begin{align*}\begin{aligned}mp-(1-m) & \geq\frac{1+\frac{4R_{0}}{5}}{2}\bigg(1-\frac{1}{15}R_{0}\bigg)+\bigg(\frac{\frac{4R_{0}}{5}-1}{2}\bigg)\\
 & \geq\frac{1}{2}\bigg(\frac{8}{5}R_{0}-\frac{1}{15}R_{0}-\frac{4}{75}R_{0}^{2}\bigg)\\
 & >\frac{3}{5}R_{0}.\\
\end{aligned}
\end{align*}
Thus, the desired result follows by applying Lemma \ref{G_attractor} with $\sigma=\frac{3}{5}R_{0}$ and noticing that the hypothesis in \eqref{Gcali} follows by the assumption \eqref{tm1cali} taking $C$ small enough.
\end{proof}
\noindent In the next corollary, we shall explain in which sense the sets whose formation we showed above have an attractive property. Before stating it we will need the following notation: \\

\begin{definition}\label{Ledef}
Given positive times $t_{0}\leq t_{1},$ we will define the new time-dependent interval in $[t_{1},\infty)$, which will be a dynamic neighborhood of $L_{\gamma}^{+}(t_{0})_{t_{1}},$ as follows. First, we define
\[\big(L_{\gamma}^{+}(t_{0})_{t_{1}}\big)_{\epsilon}=\big\{\theta\in\mathbb{T}\hspace{1mm}:\hspace{1mm}\inf_{\theta^{*}\in L_{\gamma}^{+}(t_{0})_{t_{1}}}\cos(\theta-\theta^{*})\geq1-\epsilon\big\},
\]
for any $\epsilon$ in $[R_{0}/15,1).$ Second, using the notation in subsection 2.3,  we will denote the $\theta$-projection of the image of $\big(L_{\gamma}^{+}(t_{0})_{t_{1}}\big)_{\epsilon}$ under the characteristic flow, that is $\Theta_{t_{1},t}\big(\big(L_{\gamma}^{+}(t_{0})_{t_{1}}\big)_{\epsilon}\times[-W,W]\big),$ by
\[\big(L_{\gamma}^{+}(t_{0})_{t_{1}}\big)_{\epsilon,t},
\]
for any $t>t_{1}.$
When $t_{0}$ is clear from the context, we will avoid referring to it in the above notation.\\
\end{definition}
\noindent Now, we are ready to state the corollary.
\begin{corollary}[\textbf{Emergence of attractor sets}]\label{Ga} Consider non-negative times $t\geq t_{1}\geq T_{-1}$ and  let $\epsilon=R_{0}/15.$
\noindent Then, there exists a universal constant $C$ such that if 
\begin{equation}\label{Gacali}
\frac{W}{K}<CR_{0}^{2},
\end{equation}
then
\begin{equation}\label{cormtm1}\rho\big((L_{\gamma}^{+}(T_{-1})_{t_{1}})_{\epsilon,t}\big)\geq\frac{1+\frac{4}{5}R_{0}}{2},
\end{equation}
\begin{equation}\label{corcostm1}\inf_{\theta\in(L_{\gamma}^{+}(T_{-1})_{t_{1}})_{\epsilon,t}}R\cos(\theta-\phi)\geq\frac{1}{2}R_{0},
\end{equation}
and
\begin{equation}\label{cordiamtm1}\inf_{\theta,\theta^{\prime}\in(L_{\gamma}^{+}(T_{-1})_{t_{1}})_{\epsilon,t}}\cos(\theta-\theta^{\prime})\geq1-\frac{1}{3}R_{0},
\end{equation}
hold true for every $t$ in $[t_{1},\infty)$.
\begin{proof}
We will show how to select $C$ appropriately at the end of the proof, for the moment, let us make it small enough so that we can use Lemma \ref{t-1}. The proof will follow directly from Lemma \ref{G_attractor} by setting $L:=L_{\gamma}^{+}(T_{-1})_{t_{1},\epsilon}.$ To verify the corresponding hypotheses; first, we begin by controlling the mass in $L.$ Indeed, by Lemma \ref{t-1} we have that 
\begin{equation}\label{Gamass}
\rho(L_{\gamma}^{+}(T_{-1})_{t_{1},\epsilon})\geq\rho\big(L_{\gamma}^{+}(T_{-1})_{t_{1}}\big)\geq\frac{1+\frac{4}{5}R_{0}}{2}.
\end{equation}
Second, we estimate the infimum over the cosine of the difference of angles in $L_{\gamma}^{+}(T_{-1})_{t_{1},\epsilon}$ for this purpose let $\text{\ensuremath{\bar{\theta}}}$ be contained in $L_{\gamma}^{+}(T_{-1})_{t_{1}}.$ Then, for any $\theta$ and $\theta^{\prime}$ in $\big(L_{\gamma}^{+}(T_{-1})_{t_{1}}\big)_{\epsilon},$ we have that 
\begin{align*}\begin{aligned}\cos\big(\theta-\theta^{\prime}\big) & =\cos\big(\theta-\bar{\theta}+\bar{\theta}-\theta^{\prime}\big)\\
 & =\cos\big(\theta-\bar{\theta}\big)\cos\big(\theta^{\prime}-\bar{\theta}\big)-\sin\big(\theta-\bar{\theta}\big)\sin(\theta^{\prime}-\bar{\theta})\\
 & \geq\bigg[1-\frac{1}{15}R_{0}\bigg]^{2}-\bigg[1-\bigg[1-\frac{1}{15}R_{0}\bigg]^{2}\bigg]\\
 & \geq2\bigg[1-\frac{1}{15}R_{0}\bigg]^{2}-1\\
 & \geq1-\frac{1}{3}R_{0}.
\end{aligned}
\end{align*}
Thus, since $\theta$ and $\theta^{\prime}$ were arbitrary, we deduce that 
\begin{equation}\label{Gacos}
\inf_{\theta,\theta^{\prime}\in L_{\gamma}^{+}(T_{-1})_{t_{1},\epsilon}}\cos(\theta-\theta^{\prime})\geq1-\frac{1}{3}R_{0}.
\end{equation}
Finally, considering $m=\frac{1+\frac{4}{5}R_{0}}{2}$ and $p=1-\frac{1}{3}R_{0}$ and using the bounds in \eqref{Gamass} and \eqref{Gacos}, we obtain that
\begin{align*}\begin{aligned}mp-(1-m) & \geq\frac{1+\frac{4R_{0}}{5}}{2}\bigg(1-\frac{1}{3}R_{0}\bigg)+\bigg(\frac{\frac{4R_{0}}{5}-1}{2}\bigg)\\
 & \geq\frac{1}{2}\bigg(\frac{8}{5}R_{0}-\frac{1}{3}R_{0}-\frac{4}{15}R_{0}^{2}\bigg)\\
 & >\frac{1}{2}\bigg(\frac{24-5-4}{15}\bigg)R_{0}\\
 & =\frac{1}{2}R_{0}.\\
\end{aligned}
\end{align*}
Therefore, the desired result follows by choosing $C$ appropriately in \eqref{Gacali} so that \eqref{Gcali} holds and applying Lemma \ref{G_attractor} with $\sigma=\frac{R_{0}}{2}.$
\end{proof}
\end{corollary}

\noindent In the next Lemma, we derive an estimate that we will use in Section 5. Such an estimate shows that if the entropy production vanishes over sufficiently long intervals of time, then $L^{2}$ norm of the solution in $\mathbb{T}\backslash\big(L_{\gamma}^{+}(T_{-1})_{t}\big)_{\epsilon},$ will begin to decrease exponentially.
\begin{lemma}\label{BBt1t2} Let $[t_{1},t_{2}]$ be a time interval in $[T_{-1},\infty), $ such that 
\[
\dot{R}\leq K\frac{\lambda^{3}R_{0}^{3}\cos^{2}\alpha}{4}\hspace{1em}\text{and}\hspace{1em}R<2R_{0},\hspace{1em}\text{in}\hspace{1em}[t_{1},t_{2}].
\]
with $\alpha$ as specified in Section 2. Assume that $\epsilon=R_{0}/15$ and $\lambda$ is contained in $(2/3,1).$ Then there exists a universal constant $C$ and some $\delta>0$ such that if 
\begin{equation}\label{BBt1t2cali}
\frac{W}{K}<C\lambda^{2}R_{0}^{2}(1-\lambda),\hspace{1em}\text{and}\hspace{1em}t_{2}-t_{1}\geq\delta
\end{equation}
then, we have that  
\begin{equation}\label{12decay}f^{2}(\mathbb{T}\backslash\big(L_{\gamma}^{+}(T_{-1})_{t}\big)_{\epsilon})\leq f^{2}(L_{\alpha}^{-}(t_{1}))e^{K(2\delta R_{0}-\frac{(t-t_{1}-\delta)R_{0}\sin\alpha}{2})}\hspace{1em}\text{in}\hspace{1em}[t_{1}+\delta,t_{2}].
\end{equation}
Moreover, we can choose $\delta$ so that
\begin{equation}\label{12delta}\delta\lesssim\frac{1}{K\lambda R_{0}\cos^{2}\alpha}+\frac{\sin\alpha}{K\lambda R_{0}}\log\frac{1}{R_{0}}.
\end{equation}
\end{lemma}
\begin{proof}
We will show how to select $C$ appropriately at the end of the proof, for the moment, let us make it small enough so that we can use Lemma \ref{Rdecrease} and Lemma \ref{t-1}. The proof is based on Lemma \ref{Rdecrease},  Proposition \ref{l2imeq}, Lemma \ref{t-1},  and the following differential inequalities:
\begin{align}\begin{aligned}\label{bb12ineq}\frac{d}{dt}\underline{P} & \geq K\lambda R_{0}\sqrt{1-\underline{P}^{2}}\bigg(\sqrt{1-\underline{P}^{2}}-\frac{4\cos\alpha}{5}\bigg)\hspace{1em} & \text{in}\text{\ensuremath{\hspace{1em}}}[t_{1},t_{2}]\cap\{|P|\leq\sin\alpha\},\\
\frac{d}{dt}(1-P) & \leq-\frac{1}{4}\sin\alpha K\lambda R_{0}(1-P)\hspace{1em} & \text{in}\hspace{1em}[s,t_{2}]\cap\{P\leq1-R_{0}/15\}.
\end{aligned}
\end{align}
Such inequalities hold when $\underline{P}=\cos(\underline{\Theta}_{r,t}(\theta,\omega)-\phi)$ for any $r$ and for any $\theta$ satisfying that $\cos(\theta-\phi(r))=-\sin\alpha$ in $[t_{1},t_{2}]$,  and when $P=\cos(\Theta_{r^{\prime},t}(\theta,\omega)-\Theta_{T_{-1},t}(\theta^{\prime},\omega^{\prime}))$ for any $r^{\prime}$ in $[t_{1},t_{2}]$ and any $\theta$ and $\theta^{\prime}$ such that $\cos(\theta-\phi(r^{\prime}))\geq\sin\alpha$ and $\theta^{\prime}$ is contained in $L_{\gamma}^{+}(T_{-1}).$ Here, $\omega$ and $\omega^{\prime}$ are contained in $[-W,W].$ \\

\noindent We claim such inequalities imply  that there exists $\delta>0$ satisfying \eqref{12delta} such that 
\[\mathbb{T}\backslash\big(L_{\gamma}^{+}(T_{-1})_{s}\big)_{\epsilon}\subset L_{\alpha}^{-}(s-\delta)_{s},
\]
for any $s$ in $[t_{1}+\delta,t_{2}].$ Here, we are using the notation introduced in Section 2.3 and in Definition \ref{Ledef}. We divide the proof into three steps, the second of which will be the proof of the claim:\\
    \noindent $\bullet$ \textit{Step 1}: We show that the claim implies \eqref{12decay}.\\
\noindent To achieve this let $s$ be contained in $[t_{1}+\delta,t_{2}].$ Then, using Lemma \ref{Rdecrease} and Proposition \ref{l2imeq}, on the interval $[t_{1},s-\delta]$ we obtain that 
\[f^{2}(L_{\alpha}^{-}(s-\delta))\leq f^{2}(L_{\alpha}^{-}(t_{1}))e^{-K(\frac{(s-\delta-t_{1})KR_{0}\sin\alpha}{2})}.
\]
Consequently,  once the claim is proved, the lemma would follow by the above inequality and Lemma \ref{slidel2}.\\
\noindent $\bullet$ \textit{Step 2}: We show how the inequalities in \eqref{bb12ineq}  imply the claim.\\
Consider a time $r$ contained in $[t_{1},t_{2}-\delta]$. Since we are assuming that $\underline{P}(r)=-\sin\alpha,$ the first inequality in \eqref{bb12ineq} implies that there exists $\underline{\delta}>0$ such that
\[\frac{d}{dt}\underline{P}\geq\frac{K\lambda R_{0}\cos^{2}\alpha}{5}\hspace{1em}\text{\ensuremath{\text{in}\hspace{1em}[r,r+\underline{\delta}]}}.
\]
Consequently, in particular we can find  $\underline{\delta}$ such that the above property holds, $\underline{P}(r+\underline{\delta})=\sin\alpha$ and 
\[\underline{\delta}\leq\frac{10\alpha}{K\lambda R_{0}\cos^{2}\alpha}.
\]
By the definition of $\underline{P}$ this implies that
\[\mathbb{T\backslash}L_{\alpha}^{+}(s)\subset L_{\alpha}^{-}(s-\underline{\delta})_{s},
\]
for any $s$ in $[t_{1}+\delta,t_{2}].$ To derive such implication, we have  set $s=r+\underline{\delta}.$ \\
\noindent Consequently, if we let $\theta$ be any element $\mathbb{T}\backslash L_{\alpha}^{-}(s-\underline{\delta})_{s}$ and we set $r^{\prime}=r+\underline{\delta}$ in the definition of $P$ then, by Lemma \ref{t-1} and construction, we have that $P(r^{\prime})>-1.$ Moreover, by integrating the second inequality in \eqref{bb12ineq} we have that we can find $\overline{\delta}>0$ such that $P(s+\underline{\delta}+\overline{\delta})\geq1-\frac{R_{0}}{15}$ and
\[\overline{\delta}\lesssim\frac{\sin\alpha}{K\lambda R_{0}}\log\frac{1}{R_{0}}.\]
Thus, by the construction of $P$ we obtain that
\[\mathbb{T}\backslash\big(L_{\gamma}^{+}(T_{-1})_{r+\underline{\delta}+\overline{\delta}}\big)_{\epsilon}\subset\mathbb{T}\backslash L_{\alpha}^{+}(r+\bar{\underline{\delta}})_{r+\underline{\delta}+\overline{\delta}}.\]
Consequently, the claim follows by selecting $s=r+\underline{\delta}+\overline{\delta}$ and $\delta=\underline{\delta}+\overline{\delta}.$ \\
\noindent $\bullet$ \textit{Step 3}: We  derive \eqref{bb12ineq}.\\
\noindent Let us denote:
\[\underline{\Theta}=\Theta_{r,t}(\theta,\omega),\hspace{1em}\Theta=\Theta_{r^{\prime},t}(\theta,\omega),\hspace{1em}\text{and},\hspace{1em}\Theta^{\prime}=\Theta_{T_{-1},t}(\theta^{\prime},\omega).
\]
To derive the first inequality, observe that 
 thanks to Lemma \ref{dRdphi} and our assumption on $\dot{R},$ we can select the constant in $\eqref{BBt1t2cali}$ appropriately so that we can guarantee that
\begin{align*}\begin{aligned}\frac{d}{dt}\cos\big(\underline{\Theta}-\phi) & =-\sin(\underline{\Theta}-\phi(t))(\dot{\underline{\Theta}}-\dot{\phi})\\
 & =-\sin(\underline{\Theta}-\phi)(\omega-KR\sin(\underline{\Theta}-\phi))-\dot{\phi})\\
 & \geq-|\sin(\underline{\Theta}-\phi)|\bigg[\frac{1}{R}\sqrt{K\frac{d}{dt}R^{2}+W^{2}}+W-KR|\sin(\underline{\Theta}-\phi)|\bigg)\\
 & \geq|\sin(\underline{\Theta}-\phi)|(KR|\sin(\Theta-\phi)|-\frac{4K\lambda R_{0}\cos\alpha}{5}).
\end{aligned}
\end{align*}
Here, in the third inequality, we have used Lemma \ref{dRdphi}. Consequently, $\underline{P}$ satisfies the inequality 
\[\frac{d}{dt}\underline{P}\geq K\lambda R_{0}\sqrt{1-\underline{P}^{2}}\bigg(\sqrt{1-\underline{P}^{2}}-\frac{4\cos\alpha}{5}\bigg).
\]
Thus, the first inequality in \eqref{bb12ineq} follows.\\
Finally, to derive the second inequality we use the same argument in the derivation of \eqref{GE-attractor-3} to obtain that
\[\frac{dP}{dt}\geq2K\sqrt{1-P^{2}}\left[R\cos\bigg(\frac{\Theta+\Theta^{\prime}}{2}-\phi\bigg)\sqrt{\frac{1-P}{2}}-\frac{W}{K}\right]\hspace{1em}\text{in}\hspace{1em}[t_{1},t_{2}],
\]
Now, using the same arguments as in the proof of inequality in \eqref{bb12ineq} and equation \eqref{costm1}, we obtain that 
\begin{align*}\begin{aligned}\cos\bigg(\frac{\Theta+\Theta^{\prime}}{2}-\phi\bigg) & =\cos\bigg(\frac{(\Theta-\phi)+(\Theta^{\prime}-\phi)}{2}\bigg)\\
 & \geq\frac{\cos(\Theta-\phi)+\cos(\Theta^{\prime}-\phi)}{2}\\
 & \geq\frac{\sin(\alpha)+\frac{3}{5}R_{0}}{2}\\
 & \geq\frac{\sin\alpha}{2}.
\end{aligned}
\end{align*}
Here, we have used the fact that the first inequality in \eqref{bb12ineq} implies that $\cos(\Theta-\phi)\geq\sin\alpha$ in $[r^{\prime},t_{2}].$ Thus, we deduce that, $\text{whenever, }1-P\geq R_{0}/15,$ we have that\begin{align*}\begin{aligned}\frac{dP}{dt} & \geq2K\sqrt{1-P^{2}}\left[\frac{R_{0}\lambda\sin\alpha}{2}\sqrt{\frac{1-P}{2}}-\frac{W}{K}\right]\\
 & \geq2K\sqrt{1+P}\bigg[\frac{\sqrt{2}}{4}R_{0}\lambda\sin\alpha(1-P)-\frac{W}{K}\sqrt{1-P}\bigg].
\end{aligned}
\end{align*}
{\color{black}
Consequently, by choosing $C$ appropriately in \eqref{BBt1t2cali} so that 
\[\frac{W}{K}\sqrt{1-P}<CR_{0}^{2}<\frac{R_{0}\lambda\sin\alpha}{20}(1-P)
\hspace{1em}\text{whenever}\hspace{1em}
1-P\geq R_{0}/15,\]
we can guarantee that }
\[\frac{d}{dt}P\geq\frac{K\lambda R_{0}}{4}(1-P),\hspace{1em}\text{whenever \ensuremath{\hspace{1em}P\leq1-\frac{R_{0}}{15}.}}\]
Hence, the desired result follows.
\end{proof}

\noindent We close this section with a lemma that will allow us to control the $L^{2}$ norm of the solution in $\mathbb{T}\backslash\big(L_{\gamma}^{+}(T_{-1})_{t}\big)_{\epsilon}$ in the intervals of high entropy production.

\begin{lemma}\label{Gt1t2}  Let $[t_{1},t_{2}]$ be a time interval contained in $[T_{-1},\infty)$ with the property that 
\[
R<2R_{0},\hspace{1em}\text{in}\hspace{1em}[t_{1},t_{2}].
\] Then, we have that 
\[f^{2}(L_{\alpha}^{-}(t))\leq f^{2}(\mathbb{T}\backslash\big(L_{\gamma}^{+}(T_{-1})_{t_{1}}\big)_{\epsilon})e^{2KR_{0}(t-t_{1})}\hspace{1em}\text{in}\hspace{1em}[t_{1},t_{2}].\]
\end{lemma}
\begin{proof}
This Lemma follows directly from Lemma \ref{slidel2} and Corollary \ref{Ga}.
\end{proof}}

\section{Average entropy production via differential  inequalities}

\noindent In this section, we analyze the system of inequalities presented in Section 2.5 and derived in sections 3 to 4. We shall demonstrate that this system implies the control on the time $T_{0}$ presented in Theorem \ref{mainresult}. We  begin by describing a subdivision of the interval $[0,T_{0}]$ inspired by the treatment of L. Desvillettes and C. Villani in \cite{DV-05}. \\

\noindent We first subordinate the subdivision to different scales of values of the order parameter. Then, we classify the intervals (of such subdivision) into intervals where dissipation is above and below a certain threshold. Such threshold depends on the scale of the order parameter.

\subsection{The subdivision}

Now, we give the precise construction of our subdivision. Before we enter into details, we shall introduce further notation that we will  use along this part.\\ 

{\color{black}
\noindent $\bullet$ \textit{The dyadic hierarchy}: Let us consider an auxiliary time partition into subintervals $[r_{k},r_{k+1})$ whose endpoints are enumerated in the sequence $\{r_{k}\}_{k\in\mathbb{N}}$. Such a partition will be used in this part and is set according to a dyadic behavior of the square of the order parameter $R^{2}$. Namely, such sequence provides the first times at which $R^{2}$ doubles its value. To such an end, let us set $R_{0}=R(0)$ and $r_{0}=0$. Additionally, assume that $R_{k}$ and $r_{k}$ are given for certain $k\in\mathbb{N}$ and let us define
\begin{equation}\label{E-Rk-subd}R_{k+1}^{2}=2R_{k}^{2}\ \mbox{ and }\ r_{k+1}:=\inf\{t\geq r_{k}:\,R^{2}(t)\geq2R_{k}^{2}=R_{k+1}^{2}\}.
\end{equation}

\noindent Since $R$ is bounded by $1$, then the sequence consists of finitely many terms
\[0=r_{0}<r_{1}<\cdots<r_{k_{*}}<r_{k_{*}+1}=\infty.\ 
\]
Here and throughout this section, we will assume that 
\begin{equation}\label{Hypo-subd}\frac{W}{K}\leq CR_{0}^{3}\ \mbox{ and }\ 
1-\lambda\leq\frac{\cos^{2}\alpha}{180}R_{0},
\end{equation}
\noindent with $C$ small enough so that all the results in sections 3.3 and 4 hold (note that our assumption in $\lambda$ implies the lower bound $\lambda > 179/180$ and thus we can suppress $\lambda$ from the previous constraints on the universal constant $C$).\\

\noindent Now, let us set
\begin{equation}\label{E-muk-subd}\mu_{k}:=\frac{\cos^{2}\alpha}{4}\,\lambda^{3}R_{k}^{3},\ d_{k}:=\frac{1}{3KR_{k}}\log10,\ \mbox{ and }\ \delta_{k}:=\frac{1}{KR_{k}}\log\left(\frac{1}{R_{k}}\right).
\end{equation}
Observe that \eqref{Hypo-subd} implies that $\frac{W}{K}\leq C\lambda^{2}(1-\lambda)R_{k}^{2},$ for any $k=0,\ldots,k_{*}$ with the same universal constant $C$. In particular, we can use Lemma \ref{Rdecrease} and obtain that
\begin{equation}\label{E-lowerbound-Rk}
R(t)\geq\lambda R_{k},\ \mbox{ for all\ \ensuremath{t}\ in \ [\ensuremath{r_{k}},\ensuremath{r_{k+1}}).}
\end{equation}\\

\noindent $\bullet$ \textit{Initial time of the subdivision}: Let us use Lemma \ref{t-1} to define the corresponding times of formation of attractors that is, we set
\begin{equation}\label{E-t-1k-attractors}T_{-1}^{k}:=\inf\left\{ t\geq r_{k}:\,\frac{dR}{dt}\leq KQR_{k}^{3}\right\},
\end{equation}
where $k=0,\ldots,k_{*}$ and $Q$ is chosen so that we meet condition \eqref{E-t-attractor} when one applies Lemma \ref{t-1} after translating the system in time. Here, for each $k,$ we select the time translation so that the configuration of the system at time $r_{k}$ is the new initial condition (recall that, by the definition of $r_{k},$ we can use Lemma \ref{t-1} with the same universal constant $C$. Then, we let 
\begin{equation}\label{E-t0-attractors}t_{0}:=\min\{T_{-1}^{k}:\,k=0,\ldots,k_{*}\},
\end{equation}
and
\[k_{0}:=\max\{k\in\mathbb{Z}_{0}^{+}:\,r_{k}\leq t_{0}\}.\ 
\] 
\noindent Notice that since $t_{0}$ is the first time in the subdivision, Lemma \ref{t-1} and Corollary  \ref{Ga} will apply at any later step. Thus, we will obtain a controlled behavior of the characteristic flow close to the attractor set $(L_{\gamma}^{+}(t_{0})_{t})_{\epsilon}$. Here, and throughout the rest of this section we will choose  $\gamma$ by the condition
\begin{equation}\label{gamma_def}
 \cos^{2}\gamma=\frac{1}{30}R_{k_{0}}.
\end{equation}
We have done so according to condition \eqref{tm1cosdef}.\\

\noindent $\bullet$ \textit{The subdivision}: Subordinated to the ``dyadic'' sequence $\{r_{k}\}_{k=0}^{k_{*}}$, we will construct the sequence of times $\{t_{l}\}_{l\in\mathbb{N}}$ describing the subdivision in the following way. We start at the time $t_{0}$ specified in Lemma \ref{t_0}. Assume that for some $l$ in $ \mathbb{N}$ the time $t_{l}$ is given and let us proceed with the construction of $t_{l+1}$. First, consider the only $k(l)$ in $\{0,\ldots,k_{*}\}$ such that $t_{l}$ is contained in $[r_{k(l)},r_{k(l)+1})$. Then, we will distinguish between two different situations:
\begin{enumerate}
\item If $\dot{R}(t_l)<K\mu_{k(l)}$, then we set
\begin{equation}\label{E-Bk-tl}
t_{l+1}:=\sup\{t\hspace{1mm}\in [t_{l},r_{k(l)+1}):\,\dot{R}(s)<K\mu_{k(l)},\ \forall\,s\in[t_{l},s)\}.
\end{equation}
\item If $\dot{R}(t_l)\geq K\mu_{k(l)}$, then we first compute
\begin{equation}\label{E-Gk-ttl}
\widetilde{t}_{l+1}:=\sup\{t\hspace{1mm}\in[t_{l},r_{k(l)+1}):\,\dot{R}(s)\geq K\mu_{k(l)},\ \forall\,s\in[t_{l},s)\},
\end{equation}
and set $t_{l+1}$ via the following correction: 
\begin{equation}\label{E-Gk-tl}
t_{l+1}=\begin{cases}
\widetilde{t}_{l+1}+d_{k(l)} & \text{if\ensuremath{\hspace{1em}}\ensuremath{\widetilde{t}_{l+1}}+\ensuremath{d_{k(l)}}\ensuremath{\ensuremath{\leq r_{k(l)+1}},}}\\
r_{k(l)+1} & \text{otherwise}.
\end{cases}
\end{equation}
\end{enumerate}

\noindent $\bullet$ \textit{The good and the bad sets}: We can think of the intervals $[t_{l},t_{l+1})$ obeying the above first item as \textit{bad sets} as they are subject to ``small'' slope of the order parameter. On the contrary, those sets obeying the second item can be thought of \textit{good sets}, as they involve ``large'' slope of the order parameter in comparison with the critical value $K\mu_{k(l)}$. The critical value itself depends on the size of $R_{k(l)}^{2}$ in the above dyadic hierarchy as depicted in \eqref{E-muk-subd}. For this reason, we shall collect all the indices $l$ of good and bad sets associated to the index $k$ of the dyadic hierarchy as follows.
\begin{align}\label{E-Gk-Bk-t1}
\begin{aligned}G_{k} & :=\{l\in\mathbb{Z}_{0}^{+}:\,t_{l}\in[r_{k},r_{k+1})\ \mbox{ and }\ \dot{R}(t_{l})\geq K\mu_{k}\},\\
B_{k} & :=\{l\in\mathbb{Z}_{0}^{+}:\,t_{l}\in[r_{k},r_{k+1})\ \mbox{ and }\ \dot{R}(t_{l})<K\mu_{k}\},
\end{aligned}
\end{align}
for every $k=0,\ldots,k_{*}$. Equivalently, we will say that $[t_{l},t_{l+1})$ is \textit{of type $G_k$} if $l\in G_{k}$ and it is \textit{of type $B_k$} if $l\in B_{k}$. For notational purposes, we will denote their sizes
\begin{align}
\begin{aligned}\label{E-gk-bk}g_{k} & :=\#G_{k},\\
b_{k} & :=\#B_{k},
\end{aligned}
\end{align}
for every $k=0,\ldots,k_{*}$. Notice that as a consequence of the definition \eqref{E-Gk-Bk-t1}, after any interval of type $B_{k}$ whose closure is properly contained in $[r_{k},r_{k+1})$ there is an interval of type $G_{k}$. The reverse statement is not necessarily true. Namely, notice that for any $l$ in $G_{k}$, we need first to compute the interval $[t_{l},\widetilde{t}_{l+1})$ according to \eqref{E-Gk-ttl} and later we extend it into the interval of type $G_{k}$ $[t_{l},t_{l+1})$. Unfortunately, the slope $\dot{R}$ can both grow or decrease in $[\widetilde{t}_{t+1},t_{l+1})$ and we then lose the control of what is next: either $G_{k}$ or $B_{k}$ set. Nevertheless, this is enough to show that
\begin{equation}\label{E-gk-bk-comp}b_{k}\leq g_{k}+1,\ \mbox{ for all }\ k=0,\ldots,k_{*}.
\end{equation}
Of course, by definition $g_{0}=\cdots=g_{k_{0}-1}=0$. The size of $g_{k}$ for $k=k_{0},\ldots,k_{*}$ will be estimated in Lemma \ref{L-bound-gk}. Finally, for notational simplicity, we shall sometimes enumerate the indices in $G_{k}$ in an increasing manner, namely,
\[G_{k}=\{l_{m}^{k}:\,m=1,\ldots,g_{k}\},\ 
\]
where $\{l_{m}^{k}\}_{1\leq m\leq g_{k}}$ is an increasing sequence for each $k=0,\ldots,k_{*}$

\subsubsection{Bound of the size of $t_0$.}\label{SSS-subdivision-t0}
By Lemma \ref{t-1} we have that that each $T_{-1}^{k}$ can be estimated via \eqref{tm1}. However, we will show that our dyadic choice allows us to get a sharper estimate of $t_{0}.$ More specifically, the cubic exponent for $R_{0}$ in \eqref{tm1} can be relaxed to a quadratic one. This is the content of the following Lemma.
\begin{lemma}[\textbf{Bound of $t_0$}] \label{t_0} Let $t_{0}$ be defined as above and suppose condition \eqref{Hypo-subd} holds. Then, we have that
\[t_{0}\lesssim\frac{1}{KR_{0}^{2}}.
\]
\end{lemma}
\begin{proof}
By construction, it is clear that $k_{0}\leq k_{*}$. By the fundamental theorem of calculus and the definition of $t_{0}$, we obtain that
\[R(r_{k+1})-R(r_{k})=\int_{r_{k}}^{r_{k+1}}\dot{R}(t)\,dt\geq KQR_{k}^{3}(r_{k+1}-r_{k}),
\]
and
\[
R(t_{0})-R(r_{k_{0}})=\int_{k_{0}}^{t_{0}}\dot{R}(t)\,dt\geq KQR_{k_{0}}^{3}(t_{0}-r_{k_{0}}),
\]
for any $k=0,\ldots,k_{0}-1$. Here, we have used the fact that $r_{k}\leq t_{0}\leq T_{-1}^{k}$ for every $k=0,\ldots,k_{0}.$ By estimate \eqref{E-t0-attractors} and the definition of $T_{-1}^{k}$ in \eqref{E-t-1k-attractors} we can control the time derivative of the order parameter in the above integrals. Using the dyadic definition of $r_{k}$ we arrive at the bounds
\begin{equation}\label{E-27}
r_{k+1}-r_{k}	\leq Q\frac{(R(r_{k+1})-R(r_{k}))}{KR_{k}^{3}}\leq\frac{1}{2}\frac{Q}{KR_{k}^{2}},
\end{equation}
and
\begin{equation}\label{E-28}
t_{0}-r_{k_{0}}	\leq\frac{Q(R(t_{0})-R(r_{k_{0}}))}{KR_{k_{0}}^{3}}\leq\frac{1}{2}\frac{Q}{KR_{k_{0}}^{2}},
\end{equation}
for any $k=0,\ldots,k_{0}-1$. To conclude the proof of the lemma, we represent $t_{0}$ via a telescopic sum
\[t_{0}=t_{0}-r_{k_{0}}+\sum_{k=0}^{k_{0}-1}(r_{k+1}-r_{k})\leq\frac{1}{2}\frac{Q}{KR_{k_{0}}^{2}}\sum_{k=0}^{k_{0}}\left(\frac{1}{2}\right)^{k}\leq\frac{Q}{KR_{0}^{2}}.\ 
\]
\end{proof}

\subsubsection{Gain vs loss} \label{SSS-subdivision-gain-vs-loss}
In the forthcoming parts, we  compare the growth of the order parameter $R$ along intervals of type $G_{k}$ with its loss on intervals of type $B_{k}$. To do this precisely, for each $k $ in $\{k_{0},...,k_{*}\},$ we have to give special consideration to the last interval of the subdivision in each $[r_{k},r_{k+1}).$ We will denote such terminal intervals by $[t_{l(k)},t_{l(k)+1})$ in such a way that $t_{l(k)}$ is in $[r_{k},r_{k+1})$ and $t_{l(k)+1}=r_{k+1}.$ We will use the ideas in Collorary \ref{gain_vs_loss}. In the following Lemma, we will see that assumption \eqref{Hypo-subd} implies that the loss in $R^{2}$ in smaller than $4/5$ of the gain (except on possibly the last interval of $[t_{l(k)},t_{l(k)+1}).)$
\begin{lemma}[\textbf{Gain vs loss}]\label{C-gain-vs-loss}
Assume that condition \eqref{Hypo-subd} holds. Then we have that
\[R^{2}(t_{l})-R^{2}(t_{l+1})\leq\frac{4}{5}\left(R^{2}(t_{l_{m}^{k}+1})-R^{2}(\tilde{t}_{l_{m}^{k}+1})\right)\leq\frac{4}{5}\left(R^{2}(t_{l_{m}^{k}+1})-R^{2}(t_{l_{m}^{k}})\right),\ 
\]
for any $l$ in $B_{k}$ and any $l_{m}^{k}$ in $\ensuremath{G_{k}\backslash l(k)}.$ 
\end{lemma}
\begin{proof}
Thanks to Corollary \ref{gain_vs_loss} and Lemma \ref{Rgain} we have that 
\[R^{2}(t_{l})-R^{2}(t_{l+1})\leq(1-\lambda^{2})R^{2}(t_{l})\leq4(1-\lambda)R_{k}^{2}\ \mbox{ and }\ R^{2}\big(t_{l_{m}^{k}+1}\big)-R^{2}\big(\tilde{t}_{l_{m}^{k}+1}\big)\geq\frac{1}{40}\lambda^{4}R_{k}^{3}.\ 
\]
In particular, our thesis holds true as long as one checks the inequality
\[4(1-\lambda)\leq\frac{1}{50}\lambda^{4}R_{k}.
\]
Such inequality is true due to our choice of $\lambda$. Here, we have used the fact that $\alpha=\pi/6$ and condition \eqref{Hypo-subd} implies that $\lambda>179/180.$
\end{proof}

\subsubsection{Number of intervals of type $G_k$}\label{SSS-subdivision-gk}
Our objective here is to obtain an estimate on the numbers $g_{k}$ for $k=k_{0},\ldots,k_{*}$. Recall that due to \eqref{E-gk-bk-comp}, this will yield a control in the number of sets of type $B_{k}$.
\begin{lemma}[\textbf{Bound on $g_k$}]\label{L-bound-gk} Assume that condition \eqref{Hypo-subd} holds. Then, we have that
\[\max(b_{k},g_{k})\lesssim\frac{1}{R_{k}}.
\]
\end{lemma}
\begin{proof}
To prove this, recall that by Lemma \ref{Rgain}, we have that
\begin{equation}\label{E-25}\sum_{l=G_{k}\backslash l(k)}\left(R^{2}(t_{l+1})-R^{2}\big(t_{l}\big)\right)\geq(g_{k}-\chi_{\{l(k)\in G_{k}\}})\frac{\lambda^{4}R_{k}^{3}}{40}.
\end{equation}
Thus, Lemma \ref{C-gain-vs-loss} implies
\begin{equation}\label{E-26}\sum_{l\in B_{k}}(R^{2}(t_{l+1})-R^{2}(t_{l}))\geq-g_{k}\frac{\lambda^{4}R_{k}^{3}}{50}.
\end{equation}
Taking the sum of both the oscillations at good and bad sets, we recover a telescopic sum involving the evaluation of $R^{2}$ at the largest and smallest of the times $t_{l}$ in $[r_{k},r_{k+1})$. Recall that by construction, the oscillation of $R^{2}$ in $[t_{l(k)},t_{l(k)+1})$ is positive, independently on whether $l(k)$ is in $B_{k}$ or $G_{k}.$ By doing this, we obtain that 
\[R_{k+1}^{2}-\lambda^{2}R_{k}^{2}\geq\frac{g_{k}}{200}\lambda^{4}R_{k}^{3}-\frac{1}{40}\chi_{\{l(k)\in G_{k}\}}\lambda^{4}R_{k}^{3}.\ 
\]
Hence, we deduce the bound
\begin{equation}\label{E-gk-bound}g_{k}\leq\frac{200(2-\lambda^{2})R_{k}^{2}}{R_{k}^{3}}+5.
\end{equation}
Here, we have used the fact that assumption \eqref{Hypo-subd} implies that $\lambda>179/180.$ Hence, the desired result follows.
\end{proof}

\subsubsection{Sum of lengths of intervals of type $G_k$.}\label{SSS-subdivision-lengths-Gk} In this section, we control the total diameter of the intervals in $G_k.$ To do this we will consider the sets $\mathring{G}_{k}$  and $\mathring{B}_{k}.$  The set $\mathring{G}_{k}$  is obtained by deleting the biggest element from $G_{k}$ if the last interval in $[r_{k},r_{k+1})$ is of type $G_{k}.$ Otherwise, we let $\mathring{G}_{k}=G_{k}.$
  On the other hand, the set $\mathring{B}_{k}$  is obtaining by deleting the last element in $B_{k}$  in the case where the intervals in $[r_{k},r_{k+1})$  do not end with two or more intervals of type $G_{k}.$  Otherwise, we let $\mathring{B}_{k}=B_{k}.$  Now, we are ready to state our control.
\begin{lemma}\label{L-sum-lengths-Gk}
The sum of the lengths of the interval $[t_{l_{m}^{k}},t_{l_{m}^{k}+1}]$ satisfies
\[\sum_{m=1}^{g_{k}}\left(t_{l_{m}^{k}+1}-t_{l_{m}^{k}}\right)\lesssim\frac{1}{KR_{k}^{2}}.\ 
\]
\end{lemma}
\begin{proof}
Let us first bound the length of each time interval $[t_{l_{m}^{k}},t_{l_{m}^{k}+1})$ of type $G_{k}$ for $m=1,\ldots,g_{k}$. Notice that as defined in \eqref{E-Gk-tl}, we have the identity
\begin{equation}\label{E-21}t_{l_{m}^{k}+1}-t_{l_{m}^{k}}=(\widetilde{t}_{l_{m}^{k}+1}-t_{l_{m}^{k}})+d_{k}.
\end{equation}
Our next goal is to estimate the first term. To such end, we shall use the idea in Lemma \ref{L-bound-gk} and the fundamental theorem of calculus to write
\[R\big(\widetilde{t}_{l_{m}^{k}+1}\big)-R\big(t_{l_{m}^{k}}\big)=\int_{t_{l_{m}^{k}}}^{\widetilde{t}_{l_{m}^{k}+1}}\dot{R}(t)\,dt\geq\frac{\cos^{2}\alpha}{4}K\lambda^{3}R_{k}^{3}\left(\widetilde{t}_{l_{m}^{k}+1}-t_{l_{m}^{k}}\right),\ 
\]
for all $m=1,\ldots,g_{k}.$ Here, we have used \eqref{E-Gk-ttl} to bound the  time derivativeof $R$. Hence, we obtain
\begin{equation}\label{osc_induc}\widetilde{t}_{l_{m}^{k}+1}-t_{l_{m}^{k}}\leq\frac{4}{\cos^{2}\alpha\,K\lambda^{3}R_{k}^{3}}\left(R\big(\widetilde{t}_{l_{m}^{k}+1}\big)-R\big(t_{l_{m}^{k}}\big)\right),\ 
\end{equation}
for all $m=1,\ldots,g_{k}$. By summing over all the intervals of type $\mathring{G}_{k}$ we obtain that
\begin{align}\label{E-22}\begin{aligned}\sum_{l\in\mathring{G}_{k}}\left(\widetilde{t}_{l+1}-t_{l}\right) & \leq\frac{4}{\cos^{2}\alpha\,K\lambda^{3}R_{k}^{3}}\sum_{l\in\mathring{G}_{k}}\left(R\big(\widetilde{t}_{l+1}\big)-R\big(t_{l}\big)\right)\\
 & =\frac{4}{\cos^{2}\alpha\,K\lambda^{3}R_{k}^{3}}\sum_{l\in\mathring{G}_{k}}\left[\left(R\big(t_{l+1}\big)-R\big(t_{l_{m}^{k}}\big)\right)-\left(R\big(t_{l+1}\big)-R\big(\widetilde{t}_{l+1}\big)\right)\right],
\end{aligned}
\end{align}
Let us add and subtract to the first term in \eqref{E-22} the oscillations of $R$ over all the sets of type $\mathring{B}_{k}$. Notice that after doing so the first term becomes a telescopic sum of evaluations of $R$ at points $t_{l}$ in $[r_{k},r_{k+1})$ and it can be easily bounded by the oscillation of $R$ between the largest and smallest $t_{l}$ that lie in $[r_{k},r_{k+1})$. In turns, it can be easily bounded by $R_{k+1}-\lambda R_{k}$ due to the definition of $r_{k+1}$ in \eqref{E-Rk-subd} and the lower bound of the order parameter given by \eqref{E-lowerbound-Rk}. Then, we obtain
\begin{align}\label{E-23}\begin{aligned}\sum_{l\in\mathring{G}_{k}}\left(\widetilde{t}_{l+1}-t_{l}\right) & \leq\frac{4}{\cos^{2}\alpha\,K\lambda^{3}R_{k}^{3}}(R_{k+1}-\lambda R_{k})\\
 & -\frac{4}{\cos^{2}\alpha\,K\lambda^{3}R_{k}^{3}}\left[\sum_{l\in\mathring{B}_{k}}\left(R(t_{l})-R(t_{l+1})\right)+\sum_{l\in\mathring{G}_{k}}\left(R\big(t_{l+1}\big)-R\big(\widetilde{t}_{l+1}\big)\right)\right].
\end{aligned}
\end{align}
Our goal is to show that the term in the second line is non-positive. Indeed, let us use lemmas \ref{Rgain} and \ref{C-gain-vs-loss}  in the second term of \eqref{E-23} to obtain that
\begin{align*}\begin{aligned}\sum_{l\in\mathring{G}_{k}}\left(\widetilde{t}_{l+1}-t_{l}\right) & \leq\frac{4(2-\lambda)}{K\cos^{2}\alpha\,\lambda^{3}R_{k}^{2}}-\frac{4}{5\cos^{2}\alpha\,K\lambda^{3}R_{k}^{3}}\sum_{l\in\mathring{G}_{k}}\left(R\big(t_{l+1}\big)-R\big(\widetilde{t}_{l+1}\big)\right)\\
 & \leq\frac{4(2-\lambda)}{\cos^{2}\alpha\,K\lambda^{3}R_{k}^{2}}.
\end{aligned}
\end{align*}	
Hence, by lemmas \ref{Rgain} and \ref{L-bound-gk} and \eqref{E-21} we deduce that 
\begin{align*}\begin{aligned}\sum_{m=1}^{g_{k}}\left(t_{l_{m}^{k}+1}-t_{l_{m}^{k}}\right) & \leq d_{k}g_{k}+\tilde{t}_{l_{g_{k}+1}^{k}}-t_{l_{g_{k}}^{k}}+\sum_{l\in\mathring{G}_{k}}\left(\widetilde{t}_{l+1}-t_{l}\right)\\
 & \lesssim\frac{1}{KR_{k}^{2}},
\end{aligned}
\end{align*}
where we have used \eqref{osc_induc} and our usual bound on the oscillation to control the difference
\[\tilde{t}_{l_{g_{k}+1}^{k}}-t_{l_{g_{k}}^{k}}.\]
Thus, the desired result follows.
\end{proof}

\subsubsection{Growth of $f^2(\mathbb{T}\setminus (L^+_\gamma(t_0)_t)_\epsilon)$} \label{SSS-subdivision-bound-L2 }
Our goal here is to control the growth of $f^{2}(\mathbb{T}\setminus(L_{\gamma}^{+}(t_{0})_{t})_{\epsilon})$ in each interval $[r_{k},r_{k+1})$, where the parameter $\epsilon$ of the neighborhood is set once for all as follows
\[\epsilon:=\frac{R_{0}}{15}.\ 
\]
Notice that $\epsilon$ has been set so that the attractive property in Corollary \ref{Ga} holds true. To initialize the iterative method, we need to control $f^{2}(\mathbb{T}\setminus(L_{\gamma}^{+}(t_{0})_{t})_{\epsilon})$ at  $t=t_{0}$. Hence, we begin by providing a control of the growth of  $f_{t}^{2}(\mathbb{T})$ during the transient $[0,t_{0}]$.
\begin{lemma}\label{L-L2-t0} Assume condition \eqref{Hypo-subd} holds. Then, we have that 
\[\Vert f_{t_{0}}\Vert_{2}^{2}\leq|| f_{0} ||_{2}^{2}e^{\frac{4Q}{R_{0}}}.\ 
\]
\end{lemma}
\begin{proof}
Thanks to Proposition \ref{l2imeq} we obtain that
\[\Vert f_{t_{0}}\Vert_{2}^{2}\leq|| f_{0} ||_{2}^{2}\exp\left(K\int_{0}^{t_{0}}R(s)\,ds\right).\ 
\]
Then, the main objective is to estimate the time integral of the order parameter. To that end, observe that
\begin{align*}\begin{aligned}\int_{0}^{t_{0}}R(s)\,ds & =\sum_{k=0}^{k_{0}-1}\int_{r_{k}}^{r_{k+1}}R(s)\,ds+\int_{r_{k_{0}}}^{t_{0}}R(s)\,ds\\
 & \leq\sum_{k=0}^{k_{0}-1}R_{k+1}(r_{k+1}-r_{k})+R_{k_{0}+1}(t_{0}-r_{k_{0}})\\
 & \leq Q\sum_{k=0}^{k_{0}}\frac{R_{k}}{KR_{k}^{2}}\\
 & =Q\sum_{k=0}^{k_{0}}\frac{1}{KR_{0}}\bigg(\frac{\sqrt{2}}{2}\bigg)^{k}\\
 & \leq\frac{4Q}{KR_{0}}.
\end{aligned}
\end{align*}
Notice that we have used \eqref{E-27} and \eqref{E-28} to estimate the lengths of the intervals $[r_{k},r_{k+1}).$ Hence, the desired result follows.
\end{proof}
\noindent Let us now begin our study on the primary goal of this section. To do this, let us introduce the following notation that we will use in this part. Define the parameters
\begin{equation}\label{E-Dk}D_{k}:=\max(b_{k},g_{k})(\delta_{k}+d_{k})+\sum_{l=1}^{g_{k}}(\widetilde{t}_{l_{m}^{k}+1}-t_{l_{m}^{k}}),
\end{equation}
for any $k=k_{0},\ldots,k_{*}$. Notice that its size can be controlled in the following way due to lemmas \ref{L-bound-gk} and \ref{L-sum-lengths-Gk} and the values  in \eqref{E-muk-subd}:
\begin{align}\label{E-bound-Dk}\begin{aligned}D_{k} & \lesssim\frac{1}{KR_{k}^{2}}+\frac{1}{R_{k}}\bigg[\frac{1}{KR_{k}}\log\left(\frac{1}{R_{k}}\right)+\frac{1}{KR_{k}}\bigg]\\
 & \lesssim\frac{1}{KR_{k}^{2}}\log\left(1+\frac{1}{R_{k}}\right).
\end{aligned}
\end{align}
Let us also introduce the following sequence of functions $\{F_{k}\}_{k=k_{0}}^{k_{*}}$. We proceed by induction. For $k=k_{0}$, we define
\[F_{k_{0}}(t):=\begin{cases}
\Vert f_{0}\Vert^{2}e^{\frac{4Q}{R_{0}}}e^{2KR_{k_{0}}(t-t_{0})}, & \mbox{for }t\in[t_{0},t_{0}+D_{k_{0}}],\\
\Vert f_{0}\Vert^{2}e^{\frac{4Q}{R_{0}}}e^{2KR_{k_{0}}D_{k_{0}}}e^{-K\frac{R_{k_{0}}\sin\alpha}{2}(t-t_{0}-D_{k_{0}})}, & \mbox{for }t\in[t_{0}+D_{k_{0}},r_{k_{0}+1}).
\end{cases}
\]
Assume that $F_{k-1}$ is given in the interval $[r_{k-1},r_{k})$ and let us define $F_{k}$ in the interval $[r_{k},r_{k+1})$ through the formula
\[F_{k}(t):=\begin{cases}
F_{k-1}(r_{k})e^{2KR_{k}(t-r_{k})}, & \mbox{for }t\in[r_{k},r_{k}+D_{k}],\\
F_{k-1}(r_{k})e^{2KR_{k}D_{k}}e^{-K\frac{R_{k}\sin\alpha}{2}(t-r_{k}-D_{k})}, & \mbox{for }t\in[r_{k}+D_{k},r_{k+1}).
\end{cases}\ 
\]
\begin{lemma}\label{uniform_Fk} Assume condition \eqref{Hypo-subd} holds, then we have that 
\[F_{k}(t)\leq|| f_{0} ||_{2}^{2}e^{\frac{B}{KR_{0}}\log\left(1+\frac{1}{R_{0}}\right)},\ t\in[r_{k},r_{k+1}),\ 
\]
for some universal constant $B$ and for each $k=k_{0},\ldots,k_{*}$.
\end{lemma}
\begin{proof}
By definition, we note that
\[F_{k}(t)\leq F_{k-1}(r_{k})e^{2KR_{k}D_{k}},\ \mbox{ for all }\ t\in[r_{k},r_{k+1}), 
\]
and for every $k=k_{0}+1\ldots,k_{*}$. Also, notice that by contruction, we have that
\[F_{k_{0}}(r_{k_{0}+1})\leq|| f_{0} ||_{2}^{2}\,e^{\frac{4Q}{R_{0}}}e^{2KR_{k_{0}}D_{k_{0}}}.\ 
\]
Then, a simple induction shows that
\begin{equation}\label{Fk_bound}F_{k}(t)\leq|| f_{0} ||_{2}^{2}\,e^{\frac{4Q}{R_{0}}}\prod_{q=k_{0}}^{k}e^{2KR_{q}D_{q}}=|| f_{0}||_{2}^{2}\exp\left(\frac{4Q}{R_{0}}+\sum_{q=k_{0}}^{k}2KR_{q}D_{q}\right).\ 
\end{equation}
Finally, let us use the bound \eqref{E-bound-Dk} on the above sum to achieve
\begin{align*}
\sum_{q=k_{0}}^{k}2KD_{q}R_{q}	\lesssim\sum_{q=k_{0}}^{k}\frac{R_{q}}{KR_{q}^{2}}\log\left(1+\frac{1}{R_{q}}\right)
	\lesssim\frac{1}{KR_{0}}\log\left(1+\frac{1}{R_{0}}\right)\sum_{q=k_{0}}^{k}\bigg(\frac{\sqrt{2}}{2}\bigg)^{q}.
\end{align*}
Hence, the desired result follows.
\end{proof}
\noindent The  sequence $\{F_{k}\}_{k=k_{0}}^{k_{*}}$ has been constructed as a barrier in order to control the map $t\rightarrow f^{2}(\mathbb{T}\setminus(L_{\gamma}^{+}(t_{0})_{t})_{\epsilon})$ at each interval $[r_{k},r_{k+1})$. We achieve this in the following theorem. Such a theorem is the main result in this section. As a byproduct, we derive corollaries \ref{r_kbounda} and \ref{r_*bounda}, which provide the basis for our discussion in Section 3.1 and  Section 3.2.
\begin{theorem}\label{T-growth-L2-antipode} Assume that condition \eqref{Hypo-subd} holds, then we have that
\[f^{2}(\mathbb{T}\setminus(L_{\gamma}^{+}(t_{0})_{t})_{\epsilon})\leq F_{k}(t),\ t\in[r_{k},r_{k+1}),\ 
\]
for each $k=k_{0},\ldots,k_{*}$.
\end{theorem}
\begin{proof}
We proceed by induction:\\

\noindent $\bullet$ \textit{Step 1}: Base case ($k=k_{0}$).\\
Notice that the inequality is true at $t=t_{0}$ thanks to Lemma \ref{L-L2-t0}. Let us now look at each of the intervals of type $G_{k_{0}}$ and $B_{k_{0}}$ and quantify the growth or decay rate of $f^{2}(\mathbb{T}\setminus(L_{\gamma}^{+}(t_{0})_{t})_{\epsilon})$ via lemmas \ref{slidel2}, \ref{BBt1t2} and \ref{Gt1t2}. Specifically, we shall distinguish between three different scenarios for each interval $[t_{l},t_{l+1})$ with $t_{l}$ in $[r_{k_{0}},r_{k_{0}+1})$ :
\begin{enumerate}
\item If the interval is of type $G_{k_{0}}$, then $\dot{R}(t_{l})\geq K\mu_{k_{0}}$ and Lemma \ref{BBt1t2} cannot be used to quantify a decrease estimate of the $L^{2}$ norm. Fortunately, we can at least use Lemma \ref{Gt1t2} on the sliding $L^{2}$ norm in combination with Corollary \ref{Ga} to obtain that
\begin{align*}\begin{aligned}f^{2}\big(L_{\alpha}^{-}(t)\big) & \leq f^{2}(\mathbb{T}\setminus(L_{\gamma}^{+}(t_{0})_{t_{l}})_{\epsilon,t})\\
 & \leq f^{2}(\mathbb{T}\setminus(L_{\gamma}^{+}(t_{0})_{t_{l}})_{\epsilon})e^{2KR_{k_{0}}(t-t_{l})}\\
 & \leq f^{2}(\mathbb{T}\setminus(L_{\gamma}^{+}(t_{0})_{t_{l}})_{\epsilon})e^{2KR_{k_{0}}(t_{l+1}-t_{l})},
\end{aligned}
\end{align*}
for every $t$ in $[t_{l},t_{l+1})$.
\item If the interval is of type $B_{k_{0}}$, then two different possibilities can take place: either $[t_{l},t_{l+1})$ is small or it is large.
\begin{enumerate}
\item If $[t_{l},t_{l+1})$ is small (i.e., $t_{l+1}-t_{l}\leq\delta_{k_{0}}$ ), then Lemma \ref{BBt1t2} cannot be used either. Then, we have to rely on a similar argument to that of type $G_{k}$, and it implies
\begin{align*}
f^{2}\big(L_{\alpha}^{-}(t)\big)	\leq f^{2}(\mathbb{T}\setminus(L_{\gamma}^{+}(t_{0})_{t_{l}})_{\epsilon})e^{2KR_{k_{0}}(t-t_{l})}
	\leq f^{2}(\mathbb{T}\setminus(L_{\gamma}^{+}(t_{0})_{t_{l}})_{\epsilon})e^{2KR_{k_{0}}\delta_{k_{0}}},
\end{align*}
for every $t$ in $[t_{l},t_{l+1})$.
\item Finally, if $[t_{l},t_{l+1})$ is large (i.e., $t_{l+1}-t_{l}>\delta_{k_{0}}$ ) then, we can apply Lemma \ref{BBt1t2}. However, notice that it can only be applied for $t$ in $[t_{l}+\delta_{k_{0}},t_{l+1})$ and, in the remaining part of the interval $[t_{l},t_{l}+\delta_{k_{0}})$ we can only apply the same argument as before supported by Lemma \ref{slidel2} about sliding $L^{2}$ norm. Specifically, for any $t$ in $[t_{l},t_{l}+\delta_{k})$ Lemma \ref{slidel2} implies
\[f^{2}\big(L_{\alpha}^{-}(t)\big)\leq f^{2}(\mathbb{T}\setminus(L_{\gamma}^{+}(t_{0})_{t_{l}})_{\epsilon})e^{2K\delta_{k_{0}}R_{k_{0}}},\ 
\]
Now, for any $t$ in $[t_{l}+\delta_{k},t_{l+1})$ lemmas \ref{BBt1t2} and \ref{Gt1t2} yield
\begin{align*}\begin{aligned}f^{2}(\mathbb{T}\setminus(L_{\gamma}^{+}(t_{0})_{t})_{\epsilon}) & \leq f^{2}(L_{\alpha}^{-}(t_{l}))e^{K\left(2R_{k_{0}}\delta_{k_{0}}-\frac{(t-t_{l}-\delta_{k_{0}})R_{k_{0}}\sin\alpha}{2}\right)}\\
 & \leq f_{t_{l}}^{2}(\mathbb{T}\setminus(L_{\gamma}^{+}(t_{0})_{t_{l}})_{\epsilon})e^{K\left(2R_{k_{0}}\delta_{k_{0}}-\frac{(t-t_{l}-\delta_{k_{0}})R_{0}\sin\alpha}{2}\right)}.
\end{aligned}
\end{align*}
\end{enumerate}
\end{enumerate}
Bearing all those possibilities in mind, let us now show the inequality for $F_{k_{0}}$ in $(t_{0},r_{k_{0}+1})$. Fix any time $t$ in $(t_{0},r_{k_{0}+1})$ and consider the index
\[p:=\max\{l\in\mathbb{N}:\,t_{l}\leq t\}.\ 
\]
Then, we shall repeat the above classification at each $[t_{l},t_{l+1})$ with $l$ in $\{0,\ldots,p-1\}$ ending with $[t_{p},t)$. Also, let us split the indices of intervals of type $B_{k_{0}}$ into two parts corresponding to small or large intervals as in the above discussion, namely,
\begin{align*}
B_{k_0}^S&:=\{l\in B_{k_0}:\, t_{l+1}-t_l\leq \delta_{k_0}\},\\
B_{k_0}^L&:=\{l\in B_{k_0}:\, t_{l+1}-t_l>\delta_{k_0}\}.
\end{align*}
Notice that we then have the disjoint union
\[\{0,\ldots,p-1\}=G_{k_{0},p}\cup B_{k_{0},p}^{S}\cup B_{k_{0},p}^{L},\ 
\]
where $G_{k_{0},p}=G_{k_{0}}\cap\{0,\ldots,p-1\},$ $B_{k_{0},p}^{S}=B_{k_{0}}^{S}\cap\{0,\ldots,p-1\},$ and \[B_{k_{0},p}^{L}=B_{k_{0}}^{L}\cap\{0,\ldots,p-1\}.\] 
\noindent By applying the above discussion  in a recursive way, we obtain that
\begin{align}\begin{aligned}\label{E-29} & f_{t_{p}}^{2}(\mathbb{T}\setminus(L_{\gamma}^{+}(t_{0})_{t_{p}})_{\epsilon})\\
 & \leq f_{t_{0}}^{2}(\mathbb{T})\exp\bigg\{2R_{k_{0}}K\bigg[\sum_{l\in G_{k_{0},p}}(t_{l+1}-t_{l})+\sum_{l\in B_{k_{0},p}^{S}}\delta_{k_{0}}\bigg]\\
 & \hspace{6em}+\sum_{l\in B_{k_{0},p}^{L}}\left(2R_{k_{0}}\delta_{k_{0}}-\frac{(t_{l+1}-t_{l}-\delta_{k_{0}})R_{k_{0}}\sin\alpha}{2}\right)\bigg\}.
\end{aligned}
\end{align}
Similarly, for any $t$ in $(t_{p},t_{p}+\delta_{k_{0}})$ we have that 
\begin{align}\label{E-30}
\begin{aligned} & f^{2}(\mathbb{T}\setminus(L_{\gamma}^{+}(t_{0})_{t})_{\epsilon})\\
 & \leq f_{t_{p}}^{2}(\mathbb{T}\setminus(L_{\gamma}^{+}(t_{0})_{t_{p}})_{\epsilon})\exp\big\{ 2KR_{k_{0}}\big[(t_{p+1}-t_{p})\chi_{\{p\in G_{k_{0}}\}}\\
 &\hspace{14em}+\delta_{k_{0}}\chi_{\{p\in B_{k_{0}}^{S}\}}+\delta_{k_{0}}\chi_{\{p\in B_{k_{0}}^{L}\}}\big]\big\} \\
 & \leq f_{t_{p}}^{2}(\mathbb{T}\setminus(L_{\gamma}^{+}(t_{0})_{t_{p}})_{\epsilon})\exp\left\{ 2KR_{k_{0}}\left[(t_{p+1}-t_{p})\chi_{\{p\in G_{k_{0}}\}}+\delta_{k_{0}}\chi_{\{p\in B_{k_{0}}\}}\right]\right\}.
\end{aligned}
\end{align}
Thus, for any $t$ in $[t_{p}+\delta_{k_{0}},t_{p+1})$ we obtain that
\begin{align}\label{E-31}
\begin{aligned} & f^{2}(\mathbb{T}\setminus(L_{\gamma}^{+}(t_{0})_{t})_{\epsilon})\\
 & \leq f_{t_{p}}^{2}(\mathbb{T}\setminus(L_{\gamma}^{+}(t_{0})_{t_{p}})_{\epsilon})\exp\left\{ 2KR_{k_{0}}\left[(t_{p+1}-t_{p})\chi_{\{p\in G_{k_{0}}\}}+\delta_{k_{0}}\chi_{\{p\in B_{k_{0}}^{S}\}}\right.\right.\\
 & \hspace{5cm}\left.\left.+\left(\delta_{k_{0}}-\frac{(t-t_{p}-\delta_{k_{0}})\sin\alpha}{4}\right)\chi_{\{p\in B_{k_{0}}^{L}\}}\right]\right\} \\
 & \leq f_{t_{p}}^{2}(\mathbb{T}\setminus(L_{\gamma}^{+}(t_{0})_{t_{p}})_{\epsilon})\exp\left\{ 2KR_{k_{0}}\left[(t_{p+1}-t_{p})\chi_{\{p\in G_{k_{0}}\}}+\delta_{k_{0}}\chi_{\{p\in B_{k_{0}}\}}\right.\right.\\
 & \hspace{6.5cm}\left.\left.-\frac{(t-t_{p}-\delta_{k_{0}})R_{0}\sin\alpha}{4}\chi_{\{p\in B_{k_{0}}^{L}\}}\right]\right\}.
\end{aligned}
\end{align}
Putting \eqref{E-29}, \eqref{E-30} and \eqref{E-31} together and recalling $D_{k}$ in \eqref{E-Dk} implies
\begin{align}\label{E-32}\begin{aligned}f_{t}^{2} & (\mathbb{T}\setminus(L_{\gamma}^{+}(t_{0})_{t})_{\epsilon})\\
 & \leq f_{t_{0}}^{2}(\mathbb{T})\exp\bigg\{ 2KD_{k_{0}}R_{k_{0}}-\sum_{l\in B_{k_{0},p}}K\frac{(t_{l+1}-t_{l})R_{k_{0}}\sin\alpha}{2}\\
 &\hspace{5cm}-K\frac{(t-t_{p})R_{k_{0}}\sin\alpha}{2}\chi_{\{p\in B_{k_{0}}\}}\bigg\},
\end{aligned}
\end{align}
where we have absorbed the $\delta_{k_{0}}$ in the  las term into $D_{k_{0}}.$\\

\noindent On the other hand, notice that we can recover $t$ from the following telescopic sum
\begin{align*}
\begin{aligned}
t	&=t-t_{p}+\sum_{l=0}^{p-1}(t_{l+1}-t_{l})+t_{0}\\
	&=t_{0}+(t-t_{p})\chi_{\{p\in G_{k_{0}}\}}+(t-t_{p})\chi_{\{p\in B_{k_{0}}\}}+\sum_{l\in G_{k_{0},p}}(t_{l+1}-t_{l})+\sum_{l\in B_{k_{0},p}}(t_{l+1}-t_{l})\\
	&\leq t_{0}+D_{k_{0}}+(t-t_{p})\chi_{\{p\in B_{k_{0}}\}}+\sum_{l\in B_{k_{0},p}}(t_{l+1}-t_{l}).
\end{aligned}
\end{align*}
Consequently,
\[-(t-t_{p})\chi_{\{p\in B_{k_{0}}\}}-\sum_{l\in B_{k_{0},p}}(t_{l+1}-t_{l})\leq-(t-t_{0}-D_{k_{0}}), 
\]
which can be used to bound the last two terms in the above exponential of \eqref{E-32}. Then, we obtain,
\begin{equation}\label{E-33}f_{t}^{2}(\mathbb{T}\setminus(L_{\gamma}^{+}(t_{0})_{t})_{\epsilon})\leq\begin{cases}
f_{t_{0}}^{2}(\mathbb{T})e^{2KD_{k_{0}}}, & \mbox{ for }\ t\in(t_{p},t_{p}+\delta_{k_{0}}),\\
f_{t_{0}}^{2}(\mathbb{T})e^{2KD_{k_{0}}-\frac{KR_{0}\sin\alpha}{2}(t-t_{0}-D_{k_{0}})}, & \mbox{ for }\ t\in[t_{p}+\delta_{k_{0}},t_{p+1}).
\end{cases}
\end{equation}
Notice that the worst situation is the one where there is no intermediate fall-off, that is, $B_{k_{0},p}^{L}=\emptyset$. Since such scenario dominates all the other possibilities, we shall restrict to it without loss of generality. This amounts to the chain of inequalities
\begin{align*}\begin{aligned}t_{p}+\delta_{k_{0}} & =t_{0}+\delta_{k_{0}}+\sum_{l\in G_{k_{0},p}}(t_{l+1}-t_{l})+\sum_{l\in B_{k_{0},p}^{S}}(t_{l+1}-t_{l})+\sum_{l\in B_{k_{0},p}^{L}}(t_{l+1}-t_{l})\\
 & \leq t_{0}+\sum_{l\in G_{k_{0}}}(t_{l+1}-t_{l})+\max(g_{k},b_{k})\delta_{k}\\
 & \leq t_{0}+D_{k_{0}},\\
\end{aligned}
\end{align*}
that is, $t_{p}+\delta_{k_{0}}\leq t_{0}+D_{k_{0}}$, that leads to restating \eqref{E-33} as follows
\[f_{t}^{2}(\mathbb{T}\setminus(L_{\gamma}^{+}(t_{0})_{t})_{\epsilon})\leq\begin{cases}
f_{t_{0}}^{2}(\mathbb{T})e^{2KD_{k_{0}}R_{k_{0}}}, & \mbox{ for }\ t\in(t_{0},t_{0}+D_{k_{0}}),\\
f_{t_{0}}^{2}(\mathbb{T})e^{2KD_{k_{0}}R_{k_{0}}-\frac{KR_{k_{0}}\sin\alpha}{2}(t-t_{0}-D_{k_{0}})}, & \mbox{ for }\ t\in[t_{0}+D_{k_{0}},r_{k_{0}+1}).
\end{cases}\ 
\]
Finally, use Lemma \ref{L-L2-t0} to relate the $L^{2}$ norm at $t=t_{0}$ and at $t=0.$ Thus, we have showed the claimed bound.\\

\noindent $\bullet$ \textit{Step 2}: Inductive hypothesis.\\
Let us assume that for  certain $k_{0}<k<k_{*}$ we have
\[f^{2}(\mathbb{T}\setminus(L_{\gamma}^{+}(t_{0})_{t})_{\epsilon})\leq F_{q}(t),\ t\in[r_{q},r_{q+1}),\ 
\]
for any $q<k$.\\

\noindent $\bullet$ \textit{Step 3}: Induction step.\\
The proof for the index $k$ becomes a simple consequence of the inductive hypothesis where we need to apply again lemmas \ref{slidel2}, \ref{BBt1t2} and \ref{Gt1t2} repeatedly in the  spirit as in \textit{Step 1} for the base step. 
\end{proof}

}

\noindent As a consequence of Theorem \ref{T-growth-L2-antipode} we obtain the following two  Corollaries.
{\color{black}
\begin{corollary}\label{r_kbounda} Suppose assumption \eqref{Hypo-subd} holds. Then, we have that
\[
r_{k+1}-r_{k}\lesssim\frac{1}{K R_{k}}\frac{1}{R_{0}}\log\bigg(1+\frac{1}{R_{0}}+W^{1/2}||f_{0}||_{_{2}}\bigg),
\]
for any $k\leq k_{*}.$
\end{corollary}
\begin{proof}
Thanks to \eqref{E-bound-Dk}, we may assume, without lost of generality that $r_{k+1}-r_{k}\geq D_{k}.$ 
Now, observe that, by Theorem \ref{T-growth-L2-antipode} and \eqref{Fk_bound} we have that 
\begin{align*}\begin{aligned}f^{2}(\mathbb{T\backslash}\big(L_{\gamma}^{+}(t_{0})\big)_{\epsilon}) & \leq F_{k}(t)\\
 & \leq||f_{0}||_{2}^{2}\,e^{\frac{4Q}{R_{0}}}\bigg(\prod_{q=k_{0}}^{k}e^{2KR_{q}D_{q}}\bigg)\ e^{-K\frac{R_{k}\sin\alpha}{2}(t-r_{k}-D_{k})}.\\
 & \leq||f_{0}||_{2}^{2}\,e^{\frac{Q^{\prime}}{R_{0}}\log\bigg(1+\frac{1}{R_{0}}\bigg)}\ e^{-K\frac{R_{k}\sin\alpha}{2}(t-r_{k}-D_{k})},
\end{aligned}
\end{align*}
for every $t$ in $[r_{k}+D_{k},r_{k+1})$ and some universal constant $Q^{\prime}.$ \\
On the other hand, by Jensen inequality, we have that
\[
\rho(\mathbb{T\backslash}(L_{\gamma}^{+}(t_{0})_{t})_{\epsilon})\leq\sqrt{4\pi Wf^{2}(\mathbb{T\backslash}(L_{\gamma}^{+}(t_{0})_{t})_{\epsilon})}.
\]
Consequently, if we let $m(s)=1-\rho(\mathbb{T\backslash}L_{\gamma}^{+}(t_{0})_{s}\big)_{\epsilon}),$ using Theorem \ref{T-growth-L2-antipode}, we deduce that
\begin{equation}\label{s_mass}
1-m(s)\leq2\sqrt{\pi}|| f_{0} ||_{2}e^{\frac{Q^{\prime}}{2KR_{0}}\log\left(1+\frac{1}{R_{0}}\right)}e^{-K\frac{R_{k}\sin\alpha}{4}(s-r_{k}-D_{k})}.
\end{equation}
For any $s$ in $[D_{k}+r_{k},r_{k+1}].$
On the other hand, by lemmas \ref{G_attractor},  \ref{t-1}, and Corollary \ref{Ga} if we let
\begin{equation}\label{rk_P}
P(t)=\inf_{\theta,\theta^{\prime}\in\big(L_{\gamma}^{+}(t_{0})_{s}\big)_{\epsilon,t}}\cos(\theta-\theta^{\prime}),
\end{equation}
we have that
\[
1-P(t)\leq\max\bigg[\frac{1}{3}R_{k_{0}}e^{-\frac{K}{8}R_{k_{0}}(t-s)},\frac{16}{R_{k_{0}}^{2}}\frac{W^{2}}{K^{2}}\bigg],
\]
for every $t$ in $[s,r_{k+1}].$ \\
Additionally, using  lemmas \ref{G_attractor} and \ref{t-1}, and Corollary \ref{Ga} if we let $L=\big(L_{\gamma}^{+}(t_{0})_{s}\big)_{\epsilon}$ we have that 
\begin{align}\begin{aligned}\label{5-E1}R & (t)\geq\inf_{\theta,\theta^{\prime}\in L_{s,t}}R\cos(\theta-\theta^{\prime})\\
 & \geq m(s)P(t)-(1-m(s))\\
 & =\big(1-(1-m(s))\big)P(t)-(1-m(s))\\
 & \geq P(t)-2(1-m(s))\\
 & \geq1-(1-P(t))-4\sqrt{\pi}W^{\frac{1}{2}}||f_{0}||_{2}e^{\frac{Q^{\prime}}{2R_{0}}\log\left(1+\frac{1}{R_{0}}\right)}e^{-K\frac{R_{k}\sin\alpha}{4}(s-r_{k}-D_{k})}.
\end{aligned}
\end{align}
Now, observe that, by construction {\color{black}
\[
\frac{\sqrt{2}}{2}\geq R\hspace{1em}\text{in}\hspace{1em}[r_{k},r_{k+1}).
\]}
Consequently, by \eqref{s_mass} and \eqref{rk_P}, if we set $t=r_{k+1}$ and $s=r_{k+1}-\frac{8}{KR_{k_{0}}}\log\frac{1}{10R_{k_{0}}}$ in \eqref{5-E1}, and make $C$ smaller within the constrains of \eqref{Hypo-subd} if necessary, we obtain that 
\begin{equation}\label{5-E2}
\frac{1}{3}R_{k_{0}}e^{-\log\frac{1}{10R_{k_{0}}}}+4\sqrt{\pi}W^{\frac{1}{2}}||f_{0}||_{2}e^{\frac{Q^{\prime}}{2R_{0}}\log\left(1+\frac{1}{R_{0}}\right)}e^{-K\frac{R_{k}\sin\alpha}{4}\bigg(r_{k+1}-r_{k}-D_{k}-\frac{8}{KR_{k_{0}}}\log\frac{1}{10R_{k_{0}}}\bigg)}\geq1-\frac{\sqrt{2}}{2}.
\end{equation}
Thus,
\[
4\sqrt{\pi}|| f_{0} ||_{2}W^{1/2}e^{\frac{C_{1}}{R_{0}}\log\left(1+\frac{1}{R_{0}}\right)}e^{-K\frac{R_{k}\sin\alpha}{4}\bigg(r_{k+1}-r_{k}-D_{k}\bigg)}\geq1-\frac{\sqrt{2}}{2}-\frac{1}{30}\geq\frac{1}{10}.
\]
for some universal constant $C_{1}.$\\

\noindent Hence,
\[
\frac{4}{KR_{k}\sin\alpha}\log\big(40\sqrt{\pi}W^{\frac{1}{2}}||f_{0}||_{2}\big)+\frac{4C_{1}}{KR_{0}}\frac{1}{R_{k}\sin\alpha}\log\left(1+\frac{1}{R_{0}}\right)+D_{k}\geq r_{k+1}-r_{k}.
\]
Consequently, using \eqref{E-bound-Dk} the desired result follows.
\end{proof}

\subsection{Proof of Corollary \ref{r_*bounda}.} 
 We will prove the Corollary, by proving that

{\color{black} \noindent 
\begin{equation}\label{rast_mass}
\rho\big(\mathbb{T}\backslash\big(L^{+}_{\gamma}(t_{0})_{s}\big)_{\epsilon,t}\big)\leq e^{-\frac{1}{10}K\sin\alpha(t-T_{0})},
\end{equation}
and
\[
\big(L_{\gamma}^{+}(t_{0})_{s}\big)_{\epsilon,t}\subset L_{\beta}^{+}(t), 
\]
for every $t$ in $[T_{0},\infty).$ Here, 
\[
s=t-\frac{8}{KR_{k_{*}}}\log\frac{1}{40R_{k_{*}}}.
\]
Additionally, recall that $\gamma$ was chosen in \eqref{gamma_def}.\\

\noindent We begin by showing the first equation in \eqref{convex_regime}. To do this, we control $r_{k_{*}}$ via the following telescopic sum and Corollary \ref{r_kbounda}
\begin{align*}\begin{aligned}r_{k_{*}} & =t_{0}+\sum_{k=k_{0}}^{k_{*}}r_{k+1}-r_{k}\\
 & \lesssim\frac{1}{KR_{0}^{2}}+\sum_{k=k_{0}}^{k_{*}}\frac{1}{KR_{k}}\frac{1}{R_{0}}\log\bigg(1+\frac{1}{R_{0}}+W^{1/2}||f_{0}||_{_{2}}\bigg)\\
 & \lesssim\frac{1}{KR_{0}^{2}}+\sum_{k=k_{0}}^{k_{*}}\bigg(\frac{\sqrt{2}}{2}\bigg)^{k}\frac{1}{KR_{0}^{2}}\log\bigg(1+\frac{1}{R_{0}}+W^{1/2}||f_{0}||_{_{2}}\bigg).\\
 & \lesssim\frac{1}{KR_{0}^{2}}\log\bigg(1+\frac{1}{R_{0}}+W^{1/2}||f_{0}||_{_{2}}\bigg)\\
\\
\end{aligned}
\end{align*}
Consequently, by construction, to guarantee the first equation in \eqref{convex_regime} it suffices to take, 
\[
r_{k_{*}}\leq T_{0}\lesssim\frac{1}{KR_{0}^{2}}\log\bigg(1+\frac{1}{R_{0}}+W^{1/2}||f_{0}||_{_{2}}\bigg).
\]
Indeed, recall that by definition $R(r_{k_{*}})\geq\sqrt{2}/2$ and consequently, by \eqref{E-lowerbound-Rk} we have that 
\[
R(t)\geq\frac{\sqrt{2}}{2}\lambda\geq\frac{3}{5},
\]
for every $t$ in $[r_{k_{*}},\infty)$\\

\noindent Now, we proceed to show that we can guarantee the second equation in \ref{convex_regime} by selecting $T_{0}$ within the desired constraints. To achieve this, we argue as in equation \eqref{5-E1} and  \eqref{5-E2} from the proof of Corollary \ref{r_kbounda}, with 
\[
s=t-\frac{8}{KR_{k_{*}}}\log\frac{1}{40R_{k_{*}}},
\]
to obtain that,
\begin{equation}\label{5-E3}
\rho\big(\mathbb{T}\backslash\big(L_{\gamma}^{+}(t_{0})_{s}\big)_{\epsilon,t}\big)\leq4\sqrt{\pi}W^{1/2}||f_{0}||_{2}e^{\frac{Q^{\prime}}{2R_{0}}\log\left(1+\frac{1}{R_{0}}\right)}e^{-K\frac{R_{k_{*}}\sin\alpha}{4}\bigg(t-r_{k_{*}}-D_{k_{*}}-\frac{8}{KR_{k_{*}}}\log\frac{1}{40R_{k_{*}}}\bigg)},
\end{equation}
and
\begin{align}\begin{aligned}\label{5-E4}\inf_{\theta\in\big(L^{+}_{\gamma}(t_{0})_{s}\big)_{\epsilon,t}}\cos(\theta & -\phi)\geq1-\frac{1}{3}R_{k_{*}}e^{-\log\frac{1}{40R_{k_{*}}}}\\
 & -4\sqrt{\pi}W^{1/2}||f_{0}||_{2}e^{\frac{Q^{\prime}}{2R_{0}}\log\left(1+\frac{1}{R_{0}}\right)}e^{-K\frac{R_{k_{*}}\sin\alpha}{4}\bigg(t-r_{k_{*}}-D_{k_{*}}-\frac{8}{KR_{k_{*}}}\log\frac{1}{40R_{k_{*}}}\bigg)},
\end{aligned}
\end{align}
for any $t$ in $\bigg[r_{k_{*}}+D_{k_{*}}+\frac{8}{KR_{k_{*}}}\log\frac{1}{40R_{k_{*}}},\infty\bigg),$
\\
 Thus, since $R_{k_{*}}\geq\sqrt{2}/2$ we see that choosing $T_{0}$ in such a way that
\begin{align}\begin{aligned}\label{T0_select} \frac{1}{KR_{0}^{2}}\log\bigg(1+\frac{1}{R_{0}}+W^{1/2}||f_{0}||_{_{2}}\bigg)\gtrsim T_{0} & \geq\frac{4}{KR_{k_{*}}\sin\alpha}\bigg[\log\frac{4\sqrt{\pi}W^{1/2}||f_{0}||_{2}}{R_{k_{*}}/120}\bigg]\\
 & \hspace{1em}+\frac{Q^{\prime}+16}{2KR_{0}}\log\left(1+\frac{1}{40R_{0}}\right)+r_{k_{*}}+D_{k_{*}}.
\end{aligned}
\end{align}
we can guarantee that condition \eqref{rast_mass} holds
for every $t$ in $[T_{0},\infty).$ Indeed, by \eqref{5-E3}, such a choice of $T_{0}$ together with Lemma \ref{G_attractor} and Corollary \ref{Ga} implies that 
\begin{equation}\label{asymp_angle}
\inf_{\theta\in\big(L^{+}_{\gamma}(t_{0})_{s}\big)_{\epsilon,t}}\cos(\theta-\phi)\geq\frac{59}{60},
\end{equation}
and
\[
\rho\big(\mathbb{T}\backslash\big(L_{\gamma}^{+}(t_{0})_{s}\big)_{\epsilon,t}\big)\leq\frac{1}{120}e^{-K\frac{R_{k_{*}}\sin\alpha}{4}(t-T_{0})},
\]
for every $t$ in $[T_{0},\infty).$ Consequently, the desired result follows from the fact that \eqref{asymp_angle} implies that  $\big(L_{\gamma}^{+}(t_{0})_{s}\big)_{\epsilon,t}\subset L_{\beta}^{+}(t).$ 
$\square$

\begin{section}{Wasserstein stability and applications to the particle system}

{\color{black}
\noindent The main objective of this section is to prove Corollary \ref{Main-corollary}. Before we proceed with the proof, let us introduce some necessary tools and notation. Along this section, we will set a probability density $f_0$ that belongs to $C^1$ and will assume that $g$ has compact support in $[-W,W]$. Indeed, we will assume that $f_0$, $K$ and $W$ satisfies the hypotheses of Theorem \ref{mainresult}. Also, we will consider the unique global-in-time classical solution $f=f(t,\theta,\omega)$ to \eqref{E-KS}.

\begin{definition}[{\bf The random empirical measures}]\label{D-random-empirical}
By the  consistency theorem  of Kolmogorov (see \cite[Theorem 3.5]{V-01}), let us consider a probability space $(E,\mathcal{F},\mathbb{P})$ and set some sequence of random variables for $k\in \mathbb{N}$
\[(\theta_k(0),\omega_k(0)):E\longrightarrow \mathbb{T}\times \mathbb{R},\]
that are i.i.d. with law $f_0$. For every $N\in \mathbb{N}$, let us consider the random variables
\[t\longmapsto(\theta_{1}^N(t),\omega_{1}(0)),\ldots,(\theta_{N}^N(t),\omega_{N}(0))\] 
solving the agent-based system \eqref{Kmodel} issued at the above random initial data. Then, we define the associated random empirical measures as follows
\begin{equation}\label{E-random-empirical}
\mu_{t}^{N}:=\frac{1}{N}\sum_{i=1}^{N}\delta_{(\theta_{i}^N(t),\omega_{i}(0))}(\theta,\omega),
\end{equation}
for every $t\geq0$.
\end{definition}

\noindent The proof of Corollary \ref{Main-corollary} gathers three different tools:\\

\begin{itemize}
\item[-] First, we shall use our main Theorem \ref{mainresult}, that quantifies the rate of convergence of the solution $f=f(t,\theta,\omega)$ towards the global equilibrium $f_\infty$ as $t\rightarrow \infty$.
\item[-] Second, we require a \textit{concentration inequality} to quantify the law of large numbers. More specifically, we need to quantify the rate of convergence in probability $\mathbb{P}$ of $\mu^N_0$ towards $f_0$ as the number of oscillators $N$ tends to infinity.
\item[-] Finally, in order to propagate the above quantification for larger times, we require some \textit{stability estimate for the transportation distance} between $\mu^N_t$ and $f_t$.\\
\end{itemize}

\noindent Those tools will allow us to quantify a time in which a sufficient number of oscillators of the particle system is  concentrated around a neighborhood of the support of the global equilibrium $f_\infty$. This, along with Lemma \ref{G_attractor} (which also holds for the particle system ), will guarantee that the concentration property of oscillators propagates for larger times. Additionally, we will derive the contraction of the diameter if the configuration of oscillators. Before beginning the rigorous proof, let us elaborate on the concentration and stability inequalities.

\subsection{Wasserstein concentration inequality}
It is apparent from the literature that the above random empirical measures $\mu^N_0$ in Definition \ref{D-random-empirical} approximate the initial datum $f_0$ as $N\rightarrow\infty$. Specifically, by the strong Law of Large Numbers (see \cite{V-58}) we obtain that
\[\mu^N_0\overset{*}{\rightharpoonup} f_0,\hspace{0.3cm}\mathbb{P}\mbox{-a.s},\]
in the narrow topology of $\mathbb{P}(\mathbb{T}\times \mathbb{R})$ as $N\rightarrow \infty$. Unfortunately, this is not enough for our purposes as we seek quantitative estimates for the rate of convergence. Such a quantitative control is called \textit{concentration inequality} and there have been many approaches to it in the literature. Most of them require some special structure on the initial data $f_0$ and the sequence of random empirical measures $\mu^N_0$, see \cite{B-11-thesis,B-11,BGV-07}. Specifically, some transportation-entropy inequality is required. To the best of our knowledge, the first result where those assumptions on can be removed was recently introduced in \cite{FG-15}. In our particular setting, it reads as follows.
\begin{lemma}\label{L-concentration-estimate}
Let $f_{0}$ be contained in $\mathbb{P}(\mathbb{T}\times\mathbb{R})$ be any probability measure with a distribution of natural frequencies $g=(\pi_{\omega})_{\#}f_{0}$ and assume that
\begin{equation}\label{E-exponential-moment}\mathcal{E}(g):=\int_{\mathbb{R}}e^{\omega^4}\,dg<\infty.
\end{equation}
Take any sequence $\{(\theta_k(0),\omega_k(0))\}_{k\in \mathbb{N}}$ of i.i.d. random variables with law $f_{0}$ and set the random empirical measures $\mu^N_0$ according to Definition \ref{D-random-empirical}. Then, 
$$
\mathbb{P}\left(W_{2}(\mu_{0}^{N},f_{0})\geq\varepsilon\right)\leq C_1e^{-C_2N\varepsilon^4},
$$
for every $\varepsilon>0$ and $N$ in $\mathbb{N}$. Here, $C_{1}$ and $C_{2}$ are two positive constants that depend neither on $\varepsilon$ nor on $N$, but only depend on $\mathcal{E}(g)$.
\end{lemma}

\begin{proof}
Take $d=2$, $p=2$, $\gamma=1$ and $\beta=4$ in \cite[Theorem 2]{FG-15}.
\end{proof}

\noindent In the above result, we used the classical quadratic Wasserstein distance $W_2$, namely,
\[W_2(\mu^N_0,f_0)=\left(\inf_{\gamma\in \Pi(\mu^N_0,f_0)}\int_{\mathbb{T}^2\times \mathbb{R}^2}(d(\theta,\theta')^2+(\omega-\omega')^2)\,d\gamma\right)^{1/2}.\]
However, as discussed in Remark \ref{R-W2-notappropriate} in Section 3, such distance is not appropriate for this problem due to the fact that the standard quadratic distance on the product Riemannian manifold $\mathbb{T}\times \mathbb{R}$ provides a cost functional which is not dimensionally correct. Indeed, we corrected such situation by scaling  $\omega$. Let us recall the scaled quadratic Wasserstein distance (see Definition \ref{D-scaled-wasserstein}),
\[SW_2(\mu^N_0,f_0)=\left(\inf_{\gamma\in \Pi(\mu,\nu)}\int_{\mathbb{T}^2\times \mathbb{R}^2}\left(d(\theta,\theta')^2+\frac{(\omega-\omega')^2}{K^2}\right)\,d\gamma\right)^{1/2}.\]
Let us note that by scaling, we can adapt the above Lemma \ref{L-concentration-estimate} to the right transportation distance $SW_2$. Specifically, let us consider the dilation with respect to $\omega$
\[\mathcal{D}_K(\omega):=\frac{\omega}{K},\ \mbox{ for }\omega\in \mathbb{R}.\]
Then, we can define the following scaled objects:
\[f_{0,K}:=(\text{Id}\otimes\mathcal{D}_K)_{\#}f_0\ \mbox{ and }\ \mu^N_{0,K}:=(\text{Id}\otimes\mathcal{D}_K)_{\#}\mu^N_0.\]
Notice that $f_{0,K}$ is contained in $\mathbb{P}(\mathbb{T}\times \mathbb{R})$ and the empirical measures  $\mu^N_{0,K}$ are i.i.d. variables with law $f_{0,K}$. Interestingly, we obtain the relation
\[SW_2(\mu^N_0,f_0)=W_2(\mu^N_{0,K},f_{0,K}).\]
Then, applying Lemma \ref{L-concentration-estimate} to the scaled objects, we obtain the following result.

\begin{lemma}\label{L-concentration-estimate-scaled}
Let $f_{0}$ be a probability density in $C^1(\mathbb{T}\times\mathbb{R})$, assume that the distribution of natural frequencies $g=(\pi_{\omega})_{\#}f_{0}$ has compact support in $[-W,W]$ and that condition \eqref{main_cali} in Theorem \ref{mainresult} holds true. Take any sequence $\{(\theta_k(0),\omega_k(0))\}_{k\in \mathbb{N}}$ of i.i.d. random variables with law $f_{0}$ and set the random empirical measures $\mu^N_0$ according to Definition \ref{D-random-empirical}. Then, 
\begin{equation}\label{E-deviation}\mathbb{P}\left(SW_{2}(\mu_{0}^{N},f_{0})\geq\varepsilon\right)\leq C_1\exp\left(-C_2N\varepsilon^{4}\right),
\end{equation}
for every $\varepsilon>0$ and $N$ in $\mathbb{N}$. Here, $C_{1}$ and $C_{2}$ are two positive universal constants.
\end{lemma}

\begin{remark}
Notice that, according to Lemma \ref{L-concentration-estimate}, the above $C_1$ and $C_2$ only depend upon $\mathcal{E}(g_K)$ where $g_K:=\mathcal{D}_{K\#}g$. Since $g$ has compact support in $[-W,W]$ we obtain that
\[1\leq \mathcal{E}(g_K)\leq e^{\frac{W^4}{K^4}},\]
so that $C_1$ and $C_2$ will ultimately depend only on $\frac{W}{K}$. However, under the assumptions \eqref{main_cali} in Theorem \ref{mainresult} $\frac{W}{K}$ is smaller than a universal constant. Consequently, $\mathcal{E}(g_K)$ can be made smaller than a universal constant arbitrarily close to $1$. This justifies that $C_1$ and $C_2$ can be considered universal constants.
\end{remark}

\subsection{Wasserstein stability estimate}
The study of \textit{Wasserstein stability estimates} or \textit{Dobrushin-type estimates} for measure-valued solutions to kinetic equations is a classical topic. Depending on the degree of regularity of the interaction kernel, an appropriate transportation distance has to be considered. In particular, the starting works by R. Dobrushin and H. Neunzert (see \cite{Do-79,N-84}) show that the bounded-Lipschitz distance is appropriate for Lipschitz-continuous interaction kernels. This type of inequalities has been generalized to some specific kernels with more limited regularity. In particular, the right transportation distance for gradient flows associated with $-\lambda$-convex is the quadratic Wasserstein distance $W_2$ (see \cite{CJLV-16}). Indeed, we do not necessarily need an underlying gradient structure, but only require that the interaction kernel is one-sided Lipschitz-continuous. This was proved in \cite[Theorem 4.7]{P-19-arxiv} for the Kuramoto model with weakly singular weights, that in our case provides the following stability estimate for $W_2$
\begin{equation}\label{Wstability}W_{2}(f_{t},\bar{f}_{t})\leq e^{\big(2K+\frac{1}{2}\big)t}W_{2}(f_{0},\bar{f}_{0}),
\end{equation}
which holds for any two measured valued solution to \eqref{E-KS}. Notice that units are not correct in the above inequality, and this is again due to the fact that $W_2$ is not dimensionally correct in this problem (recall \ref{R-W2-notappropriate}). Instead, we can replace $W_2$ with $SW_2$ (see Definition \ref{D-scaled-wasserstein}) to recover the following result.

\begin{lemma}\label{L-stability-SW}
Consider $K>0$ and let $f$ and $\bar f$ be weak measured-valued solutions to \eqref{E-KS} with initial data $f_0$ and $\bar f_0\in \mathbb{P}_2(\mathbb{T}\times \mathbb{R})$. Then, we have that
$$SW_2(f_t,\bar f_t)\leq e^{\frac{5}{2}Kt}SW_2(f_0,\bar f_0),$$
for every $t\geq 0$.
\end{lemma}

\begin{proof}
Consider an optimal transference plan $\gamma_0$ joining $f_0$ to $\bar f_0$, i.e.,
$$\gamma_0\in\Pi(f_0,\bar f_0):=\left\{\gamma\in \mathbb{P}((\mathbb{T}\times \mathbb{R})\times (\mathbb{T}\times \mathbb{R})):\,(\pi_1)_{\#}\gamma=f_0\ \mbox{ and }\ (\pi_2)_{\#}\gamma=\bar f_0\right\},$$
such that
$$SW_2(f_0,\bar f_0)^2=\int_{\mathbb{T}\times \mathbb{R}}\int_{\mathbb{T}\times \mathbb{R}}d_K((\theta_1,\omega_1),(\theta_2,\omega_2))^2\,d_{((\theta_1,\omega_1),(\theta_2,\omega_2))}\gamma_0.$$
Here $\pi_1$ and $\pi_2$ represent the projections
\begin{align*}
\pi_1((\theta,\omega),(\theta',\omega'))&=(\theta,\omega),\\
\pi_2((\theta,\omega),(\theta',\omega'))&=(\theta',\omega').
\end{align*}
Let us consider the following competitor at time $t$ via push-forward, namely,
$$\gamma_t:=(\mathbb{X}_{0,t}\otimes \overline{\mathbb{X}}_{0,t})_{\#}\gamma_0\in \mathbb{P}((\mathbb{T}\times \mathbb{R})\times (\mathbb{T}\times \mathbb{R})),$$
where $\mathbb{X}_{0,t}(\theta,\omega)=(\Theta_{0,t}(\theta,\omega),\omega)$ and $\overline{\mathbb{X}}_{0,t}(\theta,\omega)=(\overline{\Theta}_{0,t}(\theta,\omega),\omega)$ are the characteristic flows associated with the transport fields $v[f]$. Since $\gamma_t\in \Pi(f_t,\bar f_t)$, then
\begin{align*}
\frac{1}{2}SW_2(f_t,\bar f_t)^2&\leq \int_{\mathbb{T}\times\mathbb{R}}\int_{\mathbb{T}\times\mathbb{R}}\frac{1}{2}d_K((\theta_1,\omega_1),(\theta_2,\omega_2))^2\,d_{((\theta_1,\omega_1),(\theta_2,\omega_2))}\gamma_t\\
&=\int_{\mathbb{T}\times \mathbb{R}}\int_{\mathbb{T}\times \mathbb{R}}\frac{1}{2}d_K(\mathbb{X}_{0,t}(\theta_1,\omega_1),\overline{\mathbb{X}}_{0,t}(\theta_2,\omega_2))^2\,d_{((\theta_1,\omega_1),(\theta_2,\omega_2))}\gamma_0=:I(t).
\end{align*}
Our final goal is to derive some G\"{o}nwal-type inequality for $I$. Fix $(\theta_1,\omega_1),(\theta_2,\omega_2)\in \mathbb{T}\times \mathbb{R}$ and define the following curves in $\mathbb{T}$
$$\Theta(t):=\Theta_{0,t}(\theta_1,\omega_1)\ \mbox{ and }\ \overline{\Theta}(t):=\overline{\Theta}_{0,t}(\theta_2,\omega_2),$$
and the associated characteristic curves in $\mathbb{T}\times \mathbb{R}$,
\begin{align*}
\mathbb{X}(t)&:=\mathbb{X}_{0,t}(\theta_1,\omega_1)=(\Theta(t),\omega_1),\\
\overline{\mathbb{X}}(t)&:=\overline{\mathbb{X}}_{0,t}(\theta_2,\omega_2)=(\overline{\Theta}(t),\omega_2).
\end{align*}
Set a minimizing geodesic $x_t:[0,1]\longrightarrow\mathbb{T}\times \mathbb{R}$ joining $\mathbb{X}(t)$ to $\overline{\mathbb{X}}(t)$, for every fixed $t>0$. Notice that the following function
$$t\longmapsto \frac{1}{2}d^2_K(\mathbb{X}(t),\overline{\mathbb{X}}(t)),$$
is Lipschitz continuous. Then, we can take derivatives and show that
\begin{equation}\label{E-deriv-squared-dist}
\frac{d}{dt}\frac{1}{2}d^2_K(\mathbb{X}(t),\overline{\mathbb{X}}(t))\leq -\left<(v[f_t](\mathbb{X}(t)),0),x_t'(0)\right>-\left<(v[\bar f_t](\overline{\mathbb{X}}(t)),0),-x_t'(1)\right>,
\end{equation}
for almost every $t\geq 0$. Let us now consider $\theta(t):=\overline{\overline{\Theta}(t)-\Theta(t)}$, the representative of $\overline{\Theta}(t)-\Theta(t)$ modulo $2\pi$ that lies in $(-\pi,\pi]$. We find two different cases:

\noindent $\bullet$ \textit{Case 1:} $\theta(t)\in (-\pi,\pi)$. In this case, the only minimizing geodesic reads
$$x_t(s)=(e^{i(\Theta(t)+s\theta(t))},\omega_1+s(\omega_2-\omega_1)),\ s\in [0,1].$$
Then, the \eqref{E-deriv-squared-dist} reads
$$\frac{d}{dt}\frac{1}{2}d^2_K(\mathbb{X}(t),\overline{\mathbb{X}}(t))\leq (v[\bar f_t](\overline{\Theta}(t),\omega_2)-v[f_t](\Theta(t),\omega_1))\theta(t),$$
for almost every $t\geq 0$.

\noindent $\bullet$ \textit{Case 2:} $\theta(t)=\pi$. In this second case there are exactly two minimizing geodesics
$$x_{t,\pm}(s)=(e^{i(\Theta(t)\pm \pi s)},\omega_1+s(\omega_2-\omega_1)),\ s\in [0,1].$$
Then, we restate \eqref{E-deriv-squared-dist} as follows
$$\frac{d}{dt}\frac{1}{2}d^2_K(\mathbb{X}(t),\overline{\mathbb{X}}(t))\leq (v[\bar f_t](\overline{\Theta}(t),\omega_2)-v[f_t](\Theta(t),\omega_1))(\pm \pi),$$
for almost every $t\geq 0$. To sum up, we achieve the following estimate
\begin{multline*}
\frac{d}{dt}\frac{1}{2}d^2_K(\mathbb{X}_{0,t}(\theta_1,\omega_1),\overline{\mathbb{X}}_{0,t}(\theta_2,\omega_2))\\
\leq (v[f_t](\Theta_{0,t}(\theta_1,\omega_1),\omega_1)-v[\bar f_t](\overline{\Theta}_{0,t}(\theta_2,\omega_2),\omega_2))\overline{\Theta_{0,t}(\theta_1,\omega_1)-\overline{\Theta}_{0,t}(\theta_2,\omega_2)},
\end{multline*}
for every $\theta_1,\theta_2\in \mathbb{T}$, each $\omega_1,\omega_2\in \mathbb{R}$ and almost every $t\geq 0$. Using the dominated convergence theorem, we show that $I$ is absolutely continuous and taking derivatives under the integral sign implies
\begin{multline}\label{K-S-kinetic-E-7}
\frac{dI}{dt}\leq \int_{\mathbb{T}\times \mathbb{R}}\int_{\mathbb{T}\times \mathbb{R}} (v[f_t](\Theta_{0,t}(\theta_1,\omega_1),\omega_1)-v[\bar f_t](\overline{\Theta}_{0,t}(\theta_2,\omega_2),\omega_2))\\
\times\overline{\Theta_{0,t}(\theta_1,\omega_1)-\overline{\Theta}_{0,t}(\theta_1,\omega_2)}\,d_{((\theta_1,\omega_1),(\theta_2,\omega_2)}\gamma_0,
\end{multline}
for almost every $t\geq 0$. Also, note that
\begin{align*}
v[f_t](\theta,\omega)&=\omega-K\int_{\mathbb{T}\times \mathbb{R}}\sin(\theta-\Theta_{0,t}(\theta_1',\omega_1'))\,d_{(\theta_1',\omega_1')}f_0,\\
v[\bar f_t](\theta,\omega)&=\omega-K\int_{\mathbb{T}\times \mathbb{R}}\sin(\theta-\overline{\Theta}_{0,t}(\theta_2',\omega_2'))\,d_{(\theta_2',\omega_2')}\bar f_0.
\end{align*}
Since $(\pi_1)_{\#}\gamma_0=f_0$ and $(\pi_2)_{\#}\gamma_0=\bar f_0$, then
\begin{align}
v[f_t](\theta,\omega)&=\omega-K\int_{\mathbb{T}\times \mathbb{R}}\int_{\mathbb{T}\times \mathbb{R}}\sin(\theta-\Theta_{0,t}(\theta_1',\omega_1'))\,d_{((\theta_1',\omega_1'),(\theta_2',\omega_2'))}\gamma_0,\label{K-S-kinetic-E-8}\\
v[\bar f_t](\theta,\omega)&=\omega-K\int_{\mathbb{T}\times \mathbb{R}}\int_{\mathbb{T}\times \mathbb{R}}\sin(\theta-\overline{\Theta}_{0,t}(\theta_2',\omega_2'))\,d_{((\theta_1',\omega_1'),(\theta_2',\omega_2'))}\gamma_0.\label{K-S-kinetic-E-9}
\end{align}
Putting \eqref{K-S-kinetic-E-8}-\eqref{K-S-kinetic-E-9} into \eqref{K-S-kinetic-E-7} amounts to
\begin{multline}\label{K-S-kinetic-E-10}
\frac{d I}{dt}\leq \int_{(\mathbb{T}\times \mathbb{R})^4}(\omega_1-\omega_2)\,\overline{\Theta_{0,t}(\theta_1,\omega_1)-\overline{\Theta}_{0,t}(\theta_2,\omega_2)}\,d_{((\theta_1,\omega_1),(\theta_2,\omega_2))}\gamma_0\,d_{((\theta_1',\omega_1'),(\theta_2',\omega_2'))}\gamma_0\\
-K\int_{(\mathbb{T}\times \mathbb{R})^4}(\sin(\Theta_{0,t}(\theta_1,\omega_1)-\Theta_{0,t}(\theta_1',\omega_1'))-\sin(\overline{\Theta}_{0,t}(\theta_2,\omega_2)-\overline{\Theta}_{0,t}(\theta_2',\omega_2')))\\
\times\overline{\Theta_{0,t}(\theta_1,\omega_1)-\overline{\Theta}_{0,t}(\theta_2,\omega_2)}\,d_{((\theta_1,\omega_1),(\theta_2,\omega_2))}\gamma_0\,d_{((\theta_1',\omega_1'),(\theta_2',\omega_2'))}\gamma_0,
\end{multline}
for almost every $t\geq 0$. By  Young's inequality, it is clear that
\begin{multline*}
(\omega_1-\omega_2)\,\overline{\Theta_{0,t}(\theta_1,\omega_1)-\overline{\Theta}_{0,t}(\theta_2,\omega_2)}\\
\leq \frac{K}{2}\overline{\Theta_{0,t}(\theta_1,\omega_1)-\overline{\Theta}_{0,t}(\theta_2,\omega_2)}^2+\frac{(\omega_1-\omega_2)^2}{2K}\\
=\frac{K}{2}d_K(\mathbb{X}_{0,t}(\theta_1,\omega_1),\overline{\mathbb{X}}_{0,t}(\theta_2,\omega_2))^2.
\end{multline*}
This, along with a clear symmetrization argument in the second term implies
\begin{multline*}
\frac{dI}{dt}\leq KI(t)\\
-\frac{K}{2}\int_{(\mathbb{T}\times \mathbb{R})^4}(\sin(\Theta_{0,t}(\theta_1,\omega_1)-\Theta_{0,t}(\theta_1',\omega_1'))-\sin(\overline{\Theta}_{0,t}(\theta_2,\omega_2)-\overline{\Theta}_{0,t}(\theta_2',\omega_2')))\\
\times\left(\overline{\Theta_{0,t}(\theta_1,\omega_1)-\overline{\Theta}_{0,t}(\theta_2,\omega_2)}-\overline{\Theta_{0,t}(\theta_1',\omega_1')-\overline{\Theta}_{0,t}(\theta_2',\omega_2')}\right)\\
\times\,d_{((\theta_1,\omega_1),(\theta_2,\omega_2))}\gamma_0\,d_{((\theta_1',\omega_1'),(\theta_2',\omega_2'))}\gamma_0,
\end{multline*}
for almost every $t\geq 0$. Now, using the Lipschitz property of the sine function we achieve the inequality
$$\frac{dI}{dt}\leq (K+4K)I,\ \mbox{ for a.e. }\ t\geq 0.$$
Integrating the inequality and using that
$$I(0)=\int_{\mathbb{T}\times \mathbb{R}}\int_{\mathbb{T}\times \mathbb{R}}\frac{1}{2}d_K((\theta_1,\omega_1),(\theta_2,\omega_2))^2\,d_{((\theta_1,\omega_1),(\theta_2,\omega_2))}\gamma_0=\frac{1}{2}SW_2(f_0,\bar f_0)^2,$$
yields the desired result.
\end{proof}

\subsection{Probability of mass concentration and diameter contraction}
\noindent Now, we are ready to begin the proof of Corollary \ref{Main-corollary}. Let $L$ and $L_{1/2}$ be intervals of diameter $2/5$ and $1/5$ centered around the order parameter $\phi_\infty$ of $f_{\infty}$. Recall that by Corollary \ref{convex_regime} we obtain 
\[
R_\infty=\lim_{t\rightarrow\infty}R(t)\geq 3/5. 
\]
Looking at the structure of the stable equilibria $f_\infty$ in \eqref{equilibria} (that corresponds to $g^-=0$, that is, no antipodal mass), we observe that for any $(\theta,\omega)$ in $\supp f_\infty$ we have the relation
\[\theta=\phi_\infty+\arcsin\left(\frac{\omega}{KR_\infty}\right).\]
In particular,
\[\vert \theta-\phi_\infty\vert\leq \arcsin\left(\frac{W}{KR_\infty}\right)\leq \arcsin\left(\frac{5}{3}\frac{W}{K}\right).\]
Then, we can select $C$ in \eqref{main_cali}, so that we have that
\begin{equation}\label{support}
\supp f_{\infty}\subseteq L_{\frac{1}{2}}\times[-W,W].
\end{equation}
Notice that the choice of the diameter of $L$ is somehow arbitrary and is subordinated to the size of the universal constant $C$ in Theorem \ref{mainresult} (the smaller $C$, the smaller the diameter of $L$). For simplicity, we have set it to $2/5$ but it can be generalized to sharper values. We divide the proof into the following steps:\\

\noindent $\bullet$ \textit{Step a:} We control the mass of $\mu_{t}^{N}$ and $f_{t}$ in $\text{\ensuremath{\mathbb{T}\backslash}}L$, namely,
\begin{align}
\mu_{t}^{N}((\mathbb{T}\backslash L)\times\mathbb{R})&\leq 25\,SW_{2}(\mu_{t}^{N},f_{\infty})^2,\label{mass_mu}\\
\rho_{t}(\mathbb{T}\backslash L)&\leq 25\,SW_{2}(f_{t},f_{\infty})^2,\label{mass_f}
\end{align}
for any $t> 0$.\\

\noindent Fix $t>0$ and let $\gamma_{t}\in\mathbb{P}\big((\mathbb{T}\times\mathbb{\mathbb{R}})\times (\mathbb{T}\times\mathbb{\mathbb{R}})\big)$ be an optimal transport plan between $\mu_{t}^{N}$ and $f_{\infty}$ for the scaled Wasserstein distance $SW_2$.  Then, we have that
\begin{align*}\begin{aligned}SW_{2}(\mu_{t}^{N},f_{\infty})^{2} & =\int_{(\mathbb{T}\times\mathbb{R})^{2}}d_K((\theta,\omega),(\theta',\omega'))^{2}d\gamma_{t}\\
 & \geq\int_{\left((\mathbb{T}\backslash L)\times\mathbb{R}\right)\times\left(L_{1/2}\times\mathbb{R}\right)}d(\theta,\theta^{\prime})^{2}d\gamma_t\\
 & \geq\text{\ensuremath{\frac{1}{25}\gamma_t\big((\mathbb{T\backslash}L)\times\mathbb{R}\big)\times\big(L_{1/2}\times\mathbb{R})\big)}}\\
 & =\frac{1}{25}\bigg[\gamma_t\big(((\mathbb{T}\backslash L)\times\mathbb{R})\times(\mathbb{T\times\mathbb{R}})\big)-\gamma_t\big(((\mathbb{T}\backslash L)\times\mathbb{R})\times\mathbb{\,}((\mathbb{T}\backslash L_{1/2})\times\mathbb{R})\big)\bigg]\\
 & \geq\frac{1}{25}\bigg[(\gamma_t\big(((\mathbb{T}\backslash L)\times\mathbb{R})\times(\mathbb{T\times\mathbb{R}})\big)-\gamma_t\big((\mathbb{T}\times\mathbb{R})\times\big((\mathbb{T}\backslash L_{1/2})\times\mathbb{R})\big)\bigg]\\
 & =\frac{1}{25}\bigg[\mu_{t}^{N}((\mathbb{T}\backslash L)\times\mathbb{R})-f_{\infty}((\mathbb{T}\backslash L_{1/2})\times\mathbb{R})\bigg].
\end{aligned}
\end{align*}
Thus, using  the inclusion \eqref{support}, we observe that the second term in the last line of the above inequality vanishes and we obtain \eqref{mass_mu}. Similarly, using the above argument with $\mu_{t}^N$ replaced with $f_{t}$, we deduce that  \eqref{mass_f}.

\noindent $\bullet$ \textit{Step b:} We claim that we can select $T_{0}$ satisfying that
\begin{equation}\label{T1_size}
T_{0}\lesssim\frac{1}{KR_{0}^{2}}\log\bigg(1+W^{1/2}||f_{0}||_{2}+\frac{1}{R_{0}}\bigg),
\end{equation}
and with the  additional property that 
\begin{equation}\label{T1_distance}
SW_2(f_t,f_\infty)\leq \frac{1}{\sqrt{500}} e^{-\frac{1}{40}K(t-T_0)},
\end{equation}
for every $t$ in $[T_{0},\infty)$.\\

\noindent To show this, take $Q_1$ large enough and $T_0$ verifying
$$T_0\leq \frac{Q_1}{KR_0^2}\log\left(1+W^{1/2}\Vert f_0\Vert_2+\frac{1}{R_0}\right),$$
so that we meet the constraints in Theorem \ref{mainresult}. Then, using \eqref{fibered_conv} and Proposition \ref{P-order-fibered-distance} we obtain that
\begin{equation}
SW_2(f_t,f_\infty)\leq Q_2 e^{-\frac{1}{40}K(t-T_0)},
\end{equation}
for all $t$ in $[T_0,\infty)$ and some universal constant $Q_2$. Notice that by taking $Q_1$ large enough, we can make $Q_2$ arbitrarily small (e.g. $Q_2= \frac{1}{\sqrt{500}}$). This concludes the proof of the claim.\\

\noindent $\bullet$ \textit{Step c:} We compute $N$ in $\mathbb{N}$ and $d_{N}>0$ for each $N\geq N^*$ so that
\begin{equation}\label{overlap}
\mathbb{P}\left(SW_{2}(\mu_{t}^{N},f_{t})\leq\frac{1}{\sqrt{500}}e^{-\frac{1}{40}K(t-T_{0})}\right)\geq1-C_{1}e^{-C_{2}N^{\frac{1}{2}}},
\end{equation}
for any $t$ in $[T_{0},T_{0}+d_{N}]$ and any $N\geq N^*$.\\

\noindent First, for each $N$ in $\mathbb{N}$ let us set the scale
\begin{equation}\label{E-scaled-epsilon}
\varepsilon_{N}:=N^{-\frac{1}{8}}.
\end{equation}
Now, we define $N^*$ as follows
\begin{equation}\label{N*_def}
N^*:=\min\left\{N\in\mathbb{N}:\,\varepsilon_Ne^{\frac{5K}{2}T_0}\leq \frac{1}{\sqrt{500}}\right\},
\end{equation}
so that, by definition, we get the bound
\[N^*\geq 500^4e^{20KT_0}.\]
Fix any $N\geq N^*$. Notice that $N^*$ has been defined in \eqref{N*_def} so that there exists $d_N>0$ with the property
\begin{equation}\label{dN_def}
\varepsilon_{N}e^{\frac{5K}{2}(T_{0}+d_{N})}=\frac{1}{\sqrt{500}}e^{-\frac{1}{40}Kd_{N}},
\end{equation}
Indeed, by dividing \eqref{dN_def} over \eqref{N*_def}, we can quantify $d_{N}$ in terms of $N^{*}$ as follows
\[
\frac{\varepsilon_{N}}{\varepsilon_{N^{*}}}e^{\frac{5K}{2}d_{N}}\geq e^{-\frac{1}{40}Kd_{N}}.
\]
Consequently, we have that 
\[
d_{N}\geq\frac{5}{101K}\log\frac{N}{N^{*}}.
\]
By construction, letting $\varepsilon=\varepsilon_{N}$ in the concentration inequality \eqref{E-deviation} of Lemma \ref{L-concentration-estimate-scaled}, we obtain the following quantification
\begin{equation}\label{E-conc-ineq-scaled}
\mathbb{P}\left(SW_{2}(\mu_{0}^{N},f_{0})\geq\varepsilon_{N}\right)\leq C_{1}e^{-C_{2}N^{\frac{1}{2}}},
\end{equation}
for every $N\in \mathbb{N}$. 
Thus, by monotonicity of the exponential function, we conclude that for any $t\in [T_0,T_0+d_N]$ we have that
\begin{align*}\begin{aligned}C_{1}e^{-C_{2}N^{\frac{1}{2}}} & \geq\mathbb{P}\left(SW_{2}(\mu_{0}^{N},f_{0})\geq\varepsilon_{N}\right)\\
 & \geq\mathbb{P}\left(SW_{2}(\mu_{t}^{N},f_{t})\geq\varepsilon_{N}e^{\frac{5K}{2}t}\right)\\
 & \geq\mathbb{P}\left(SW_{2}(\mu_{t}^{N},f_{t})\geq\varepsilon_{N}e^{\frac{5K}{2}(T_{0}+d_{N})}\right)\\
 & =\mathbb{P}\left(SW_{2}(\mu_{t}^{N},f_{t})\geq\frac{1}{\sqrt{500}}e^{-\frac{1}{40}Kd_{N}}\right)\\
 & \geq\mathbb{P}\left(SW_{2}(\mu_{t}^{N},f_{t})\geq\frac{1}{\sqrt{500}}e^{-\frac{1}{40}K(t-T_{0})}\right),
\end{aligned}
\end{align*}
where in the first inequality we have used the concentration inequality \eqref{E-conc-ineq-scaled}, in the second one we have used the stability estimate in Lemma \ref{L-stability-SW} and the remaining ones follow from our choice of $d_N$ in \eqref{dN_def} and $t$ in $[T_0,T_0+d_N]$. That ends the proof of \eqref{overlap}.\\

\noindent $\bullet$ \textit{Step d:} We quantify the probability of mass concentration of $\mu^N_t$ in the interval $L$, namely,
\begin{equation}\label{mass_prob_bound}
\mathbb{P}\left(\mu_{t}^{N}(L\times \mathbb{R})\geq1-\frac{1}{5}e^{-\frac{1}{20}K(t-T_{0})}\right)\geq 1-C_{1}e^{-C_{2}N^{\frac{1}{2}}},
\end{equation}
for every $t$ in $[T_{0},T_{0}+d_{N})$ and any $N\geq N^*$.\\

\noindent Now, by \eqref{mass_mu}, \eqref{T1_distance} and triangular inequality we have that
\begin{align*}\begin{aligned}\mu_{t}^{N}((\mathbb{T}\backslash L)\times\mathbb{R}) & \leq25\,SW_{2}(\mu_{t}^{N},f_{\infty})^{2}\\
 & \leq50\bigg[SW_{2}(\mu_{t}^{N},f_{t})^{2}+SW_{2}(f_{t},f_{\infty})^{2}\bigg]\\
 & \leq50\bigg[SW_{2}(\mu_{t}^{N},f_{t})^{2}+\frac{1}{500}e^{-\frac{1}{20}K(t-T_{0})}\bigg],
\end{aligned}
\end{align*}
for every $t$ in  $[T_{0},T_{0}+d_{N})$. Hence, we obtain
$$\mu^N_t(L\times \mathbb{R})\geq 1-\frac{1}{10}e^{-\frac{3}{10}K(t-T_0)}-50\,SW_2(\mu^N_t,f_t)^2,$$
for each $t$ in $[T_0,T_0+d_N]$. This, along with \eqref{overlap} concludes the proof of \eqref{mass_prob_bound}\\

\noindent $\bullet$ \textit{Step e:} We quantify the probability of mass concentration and diameter contraction along the time interval $[s,\infty)$ for any $s$ in $[T_{0},T_{0}+d_N].$\\

\noindent 
We are now ready to finish the proof of Corollary \ref{Main-corollary}. Let us consider $N\geq N^*$, $s$ in $[T_0,T_0+d_N],$ and any realization of the random empirical measure $\mu^N$ (recall Definition \ref{D-random-empirical}) so that the condition within \eqref{mass_prob_bound} holds. Hence, by construction, we obtain that at such realization
\[p:=\inf_{\theta,\theta^{\prime}\in L}\cos(\theta-\theta^{\prime})\geq\frac{4}{5}\ \mbox{ and }\ m:=\mu^N_s(L\times \mathbb{R})\geq 1-\frac{1}{5}e^{-\frac{1}{20}K(s-T_{0})}\geq \frac{4}{5}.\]
Then, we obtain the relation
\[mp-(1-m)=\frac{4}{5}\cdot\frac{4}{5}-\left(1-\frac{4}{5}\right)=\frac{11}{25}.\]
In particular, take $\sigma:=2/5$ and notice that the above relations along with the assumption \eqref{main_cali} in Theorem \ref{mainresult} guarantee the condition \eqref{Gcali} within the hypotheses of Lemma \ref{G_attractor}. Notice that such result also holds true for the particle system. Consequently, it asserts that for such realization of $\mu^N$ we can consider a time-dependent interval $L_s^N(t)$ with $t\geq s$ so that $L_s^N(s)=L$ and
\begin{align}\label{E-massconc-diamcont}
\begin{aligned}
\mu^N_t(L^N_s(t)\times \mathbb{R})&\geq 1-\frac{1}{5}e^{-\frac{1}{20}K(s-T_{0})},\\
1-\inf_{\theta,\theta'\in L_{s}^N(t)}\cos(\theta-\theta')&\leq\max\left\{\frac{1}{5}e^{-\frac{K}{10}(t-s)},25\frac{W^2}{K^2}\right\},
\end{aligned}
\end{align}
for any $t\geq s$. Indeed, we have that $L_s^N(t)=\pi_\theta(\mathbb{X}^N_{s,t}(L\times [-W,W]))$, where $\mathbb{X}^N_{s,t}$ represents the flow of the particle system, that is, the flow of $v[\mu^N]$. Our final goal is to simplify the last condition in \eqref{E-massconc-diamcont}. To such an end, let us consider $D^N_s(t):=\diam(L^N_s(t))$ and notice that such inequality implies that 
\begin{equation}\label{E-prob-cont-diam-1}
2\frac{(D_{s}^{N}(t))^{2}}{5}\leq1-\cos(D_{s}^{N}(t))\leq\max\left\{ \frac{1}{5}e^{-\frac{K}{10}(t-s)},25\frac{W^{2}}{K^{2}}\right\},
\end{equation}
for any $t\geq s$. In particular, we obtain  \eqref{E-D-contraction}. Thus, Corollary \ref{Main-corollary} follows
}
\end{section}


\bibliographystyle{amsplain} 
\bibliography{JM-DP-ref}

\end{document}